\documentclass[10pt,a4paper]{amsart}
\usepackage[english]{babel}
\usepackage{amssymb,xspace}
\usepackage{amstext}
\usepackage[mathscr]{eucal}
\theoremstyle{plain}
\usepackage{amsbsy,amssymb,amsfonts,latexsym,eucal,eufrak,amscd}
\usepackage[dvips]{graphicx}
\usepackage{epsfig}
\usepackage[all]{xy}
\usepackage{pstricks}
\usepackage{fancyhdr}
\newcommand{\ds }{\ensuremath{\displaystyle}}

\newcommand{\st }{\ensuremath{\scriptstyle}}

\newcommand{\R }{\ensuremath{\mathbb R}}
\newcommand{\C }{\ensuremath{\mathbb C}}
\newcommand{\Q }{\ensuremath{\mathbb Q}}
\newcommand{\Z }{\ensuremath{\mathbb Z}}

\renewcommand{\P }{\ensuremath{\mathbb P}}
\newcommand{\HHH }{\ensuremath{\mathbb H}}

\DeclareMathOperator{\bl}{Bl}

\DeclareMathOperator{\ch}{ch}

\DeclareMathOperator{\reg}{reg}
\DeclareMathOperator{\supp}{supp}
\DeclareMathOperator{\td}{td}

\DeclareMathOperator{\tor}{Tor}
\DeclareMathOperator{\ext}{Ext}
\DeclareMathOperator{\fit}{Fitt}
\DeclareMathOperator{\pr}{pr}
\DeclareMathOperator{\torr}{tors}
\DeclareMathOperator{\gr}{Gr}
\DeclareMathOperator{\coh}{coh}

\DeclareMathOperator{\tors}{tors}

\DeclareMathOperator{\ran}{rank}

\newcommand{\dd }{\ensuremath{\mathcal{D}}}

\newcommand{\oo }{\ensuremath{\mathcal{O}}}
\newcommand{\hh }{\ensuremath{\mathcal{H}}}
\newcommand{\ff }{\ensuremath{\mathcal{F}}}
\newcommand{\kk }{\ensuremath{\mathcal{K}}}
\newcommand{\ii }{\ensuremath{\mathcal{I}}}
\newcommand{\cc }{\ensuremath{\mathcal{C}}}
\newcommand{\g }{\ensuremath{\mathcal{G}}}
\newcommand{\ttt }{\ensuremath{\mathcal{T}}}

\newcommand{\nn }{\ensuremath{\mathcal{N}}}
\newcommand{\eee }{\ensuremath{\mathcal{E}}}

\newcommand{\LL }{\ensuremath{\mathcal{L}}}
\newcommand{\zz}{\ensuremath{\mathcal{Z}}}
\newcommand{\qq}{\ensuremath{\mathcal{Q}}}

\newtheorem{theorem}{Theorem}[section]

\newtheorem{lemma}[theorem]{Lemma}

\newtheorem{proposition}[theorem]{Proposition}

\newtheorem{corollary}[theorem]{Corollary}

{\theoremstyle{definition}\newtheorem{notation}[theorem]{Notation}}

{\theoremstyle{definition}}

{\theoremstyle{definition}}

{\theoremstyle{definition}\newtheorem{example}[theorem]{Example}}

{\theoremstyle{definition}\newtheorem{definition}[theorem]{Definition}}

{\theoremstyle{definition}}

{\theoremstyle{definition}\newtheorem{remark}[theorem]{Remark}}

\address{Institut de Math\'{e}matiques de Jussieu, UMR 7586\\
Case 247\\ Universit\'{e} Pierre et Marie Curie\\
4, place Jussieu\\
F-75252 Paris Cedex 05\\
France}
\email{jgrivaux@math.jussieu.fr}

\newcommand{\tore}[3]{\tor_{#1}(#2,#3)}

\newcommand{\bwe}{\ensuremath{\bigwedge}}

\newcommand{\we}{\ensuremath{\wedge}}

\newcommand{\oti}{\ensuremath{\otimes}}

\newcommand{\he}{^{\vphantom{*}} }
\newcommand{\be}{_{\vphantom{i}} }

\newcommand{\tix }{\ensuremath{\ti{X}}}
\newcommand{\tiy }{\ensuremath{\ti{Y}}}

\newcommand{\ba}[1]{\ensuremath{\overline{#1}}}

\newcommand{\ti }[1]{\ensuremath{\widetilde{#1}}}

\newcommand{\rb }{\ensuremath{\raisebox}}

\newcommand{\tim }{\ensuremath{\times}}

\newcommand{\oun}{\ensuremath{\mathcal{O}(1)}}

\newcommand{\ee }{\ensuremath{^{\, *}}}

\newcommand{\iz }{\ensuremath{_{Z}}}

\newcommand{\bh }{\ensuremath{\ba{\ch}}}

\newcommand{\ie }{\ensuremath{i_{E}}}

\newcommand{\tid }{\ensuremath{\ti{D}}}

\newcommand{\pti }[1]{\ensuremath{{ _{^{\, \centerdot}_{\, \, \scriptstyle{#1}}}}}}

\newcommand{\bop }{\ensuremath{\bigoplus\limits}}

\newcommand{\suq }{\ensuremath{\subseteq}}
\newcommand{\pz }[1]{\ensuremath
{\raisebox{0. ex}{$\centerdot$}\raisebox{-1ex}{$\scriptstyle{#1}$}}\, }

\newcommand{\chtz }{\ensuremath{\ch_{Z}}}

\newcommand{\chp }{\ensuremath{\ch_{p}}}

\newcommand{\hdp }{\ensuremath{H^{2p}_{D}}}

\newcommand{\pe }{\ensuremath{^{\, !}}}

\newcommand{\kan }{\ensuremath{G}}

\newcommand{\dpp }{\ensuremath{D'}}

\newcommand{\cro }[1]{[#1]}

\entrymodifiers={+!!<0pt,\fontdimen22\textfont2>}

\def\apl#1#2#3{#1\mkern -4 mu:\mkern - 8 mu
\xymatrix{#2\!\ar[r]&\!#3}
}

\def\aplpt#1#2#3#4{#1\mkern -4 mu:\mkern - 8 mu
\xymatrix{#2\!\ar[r]&\!#3#4}
}

\def\sutrgd#1#2#3{
\xymatrix{
0\ar[r]&#1\ar[r]&#2\ar[r]&#3\ar[r]&0
}
}

\def\sutrgdpt#1#2#3#4{
\xymatrix{
0\ar[r]&#1\ar[r]&#2\ar[r]&#3\ar[r]&0#4
}
}

\def\sutrgpt#1#2#3#4{
\xymatrix{
0\ar[r]&#1\ar[r]&#2\ar[r]&#3#4
}
}

\def\sutr#1#2#3{
\xymatrix{
#1\ar[r]&#2\ar[r]&#3
}
}

\def\sutrd#1#2#3{
\xymatrix{
#1\ar[r]&#2\ar[r]&#3\ar[r]&0
}
}

\def\flcourte{\xymatrix@C=10pt{
\ar[r]&
}}

\def\flgd#1#2{\xymatrix{#1\!
\ar[r]&\!#2
}}

\def\flcourtegd#1#2{\xymatrix@C=15pt{\!\!#1\!
\ar[r]&\!#2\!\!
}}

\def\flgdin#1#2{\xymatrix@C=3ex{\!\!\scriptstyle{#1}
\ar[r]&\scriptstyle{#2}\!\!\!
}}

\def\fldouble{\xymatrix@1{
\ar@{->>}[r]&
}}

\def\fle#1#2{
\xymatrix@1{
#1
\ar[r]&#2
}}

\def\flex#1#2#3{
{\xymatrix@1{
#1
\ar[r]^{#3}&#2
}}
}

\def\fledouble#1#2{
{\xymatrix@1{
#1
\ar@{->>}[r]&{#2}
}}
}

\def\flexdouble#1#2#3{
{\xymatrix@1{
#1
\ar@{->>}[r]^{#3}&{#2}
}}
}

\def\diagca#1#2#3#4#5#6#7#8{\xymatrix@1{
#1
\ar[d]_{#6}\ar[r]_{#5}&#2\ar[d]_{#7}\\
#3
\ar[r]_{#8}&#4
}}

\def\sutrois#1#2#3{
{\xymatrix@1{
#1
\ar[r]&#2
\ar[r]&#3
}}
}

\def\sutroiszerogdprime#1#2#3{
{\xymatrix@1{
0
\ar@<-0.5mm>[r]&#1
\ar@<-0.5mm>[r]&#2
\ar@<-0.5mm>[r]&#3
\ar@<-0.5mm>[r]&0
}}
}

\def\fleprime#1#2{
\xymatrix@1{
#1
\ar[r]&#2
}}

\def\sutroisnom#1#2#3#4#5{
{\xymatrix@1{
#1
\ar[r]^{#4}&#2
\ar[r]^{#5}&#3
}}
}

\def\sutroiszerogd#1#2#3{
{\xymatrix@1{
0
\ar[r]&#1
\ar[r]&#2
\ar[r]&#3
\ar[r]&0
}}
}
\def\strgdexp#1#2#3#4#5#6{
{\xymatrix@1{
0
\ar[r]&\rb{#2ex}{$#1$}
\ar[r]&\rb{#4ex}{$#3$}
\ar[r]&\rb{#6ex}{$#5$}
\ar[r]&0
}}
}

\def\sutroiszerog#1#2#3{
{\xymatrix@1{
0
\ar[r]&#1
\ar[r]&#2
\ar[r]&#3
}}
}

\def\suxtroiszerogd#1#2#3#4#5{
{\xymatrix@1{
0
\ar[r]&#1
\ar[r]^{#4}&#2
\ar[r]^{#5}&#3
\ar[r]&0
}}
}

\def\suquatre#1#2#3#4{
{\xymatrix@1{
#1
\ar[r]&#2
\ar[r]&#3
\ar[r]&#4
}}
}

\def\suxquatre#1#2#3#4#5#6#7{
{\xymatrix@1{
#1
\ar[r]^{#5}&#2
\ar[r]^{#6}&#3
\ar[r]^{#7}&#4
}}
}

\def\sucinq#1#2#3#4#5{
{\xymatrix@1{
#1
\ar[r]&#2
\ar[r]&#3
\ar[r]&#4
\ar[r]&#5
}}
}

\def\suxcinq#1#2#3#4#5#6#7#8#9{
{\xymatrix@1{
#1
\ar[r]^{#6}&#2
\ar[r]^{#7}&#3
\ar[r]^{#8}&#4
\ar[r]^{#9}&#5
}}
}

\DeclareMathOperator{\spa}{span}

\DeclareMathOperator{\id}{id}

\DeclareMathOperator{\im}{Im}

\newcommand{\Mg }{\ensuremath{\mathfrak{M}}}

\title{Chern classes in Deligne cohomology for coherent analytic sheaves}
\author{Julien Grivaux}
\begin{document}
\begin{abstract}
In this article, we construct Chern classes in rational Deligne
cohomology for coherent sheaves on a smooth complex
compact manifold. We
prove that these classes
satisfy the functoriality property under pullbacks,
the Whitney formula and
the Grothendieck-Riemann-Roch theorem
for projective morphisms between smooth
complex compact manifolds.
\end{abstract}
\maketitle
\tableofcontents

\section{Introduction}
Let $X$ be a smooth differentiable manifold and $E$ be a complex vector
bundle of rank $r$
on $X$. The Chern-Weil theory (see \cite[Ch.\!\! 3 \S\,3]{GrHa}) constructs
classes $c_{i}(E)^{\textrm{top}}$, $1\leq i\leq r$,
with values in
the de Rham cohomology $H^{2i}(X,\R)$, which generalize
the first Chern class of a line bundle in
$H^{2}(X,\Z)$ obtained by the exponential exact sequence.
These classes are compatible with pullbacks under smooth
morphisms and verify the Whitney sum formula
\[
c_{k}(E\oplus F)^{\textrm{top}}=
\sum _{i+j=k}c_{i}(E)^{\textrm{top}}\, c_{j}(F)
^{\textrm{top}}.
\]
There exist more refined ways of defining $c_{i}(E)^{\textrm{top}}$
in $H^{2i}(X,\Z)$. The first method is due to Chow (see the introduction
of \cite{Gr1}). The idea is to define explicitly the Chern classes of
the universal bundles of the grassmannians and to write any complex
vector bundle as a quotient of a trivial vector bundle. Of
course, computations have to be done on the grassmannians to
check the compatibilities. Note that in the holomorphic or in the algebraic
context, a vector bundle is not in general a quotient of a trivial
vector
bundle. Nevertheless, if $X$ is projective, this is true after tensorising by
a sufficiently high power of an ample line bundle and the construction
can be adapted (see \cite{Bry}).
\par\medskip
A more intrinsic construction is the splitting method, introduced by
Grothendieck in \cite{Gr1}.
Let us briefly recall how it works.
By the Leray-Hirsh theorem, we know
that
$H\ee\be(\P(E),\Z)$ is a free module over $H\ee(X,\Z)$ with basis
$1,\, \alpha ,\dots,\alpha ^{r-1} $, where $\alpha $
is the opposite of the first Chern
class of the relative
Hopf bundle on $\P(E)$. Now the Chern classes
of $E$ are
uniquely defined by the relation
\[
\alpha ^{r}+p\ee c_{1}(E)^{\textrm{top}}\, \alpha ^{r-1}
+\cdots +
p\ee c_{r-1}^{\textrm{top}}(E)\,\alpha
+p\ee c_{r}\he(E)^{\textrm{top}}=0.
\]
(see Grothendieck \cite{Gr1}, Voisin \cite[Ch.\!\! 11 \S\,2]{Vo1}, and Zucker
 \cite[\S\,1]{Zuc}).
\par\medskip
The splitting method works amazingly well in various contexts, provided
that we have
\begin{enumerate}
  \item [--] the definition of the first Chern class of a line bundle,

  \item [--] a structure theorem for the cohomology
  ring of a projective
  bundle considered as a module over the cohomology ring of the base.
\end{enumerate}
Let us now examine the algebraic case. Let $X$ be a smooth algebraic variety
  over a field $k$ of characteristic zero, and $E$ be an
  algebraic bundle on $X$. Then
  the splitting principle allows to define $c_{i}(E)^{\textrm{alg}}$
\begin{enumerate}
  \item [--] in the Chow ring $CH^{i}(X)$ if $X$ is quasi-projective,

  \item [--] in the algebraic de Rham cohomology group $H^{2i}_{\textrm{DR}}
  (X/k)$.
\end{enumerate}
Suppose now that $k=\C$. Then Grothendieck's comparison theorem
(see \cite{Gr3}) says that we have a ca\-no\-nical isomorphism
between
$H_{DR}^{2i}(X/\C\, )$ and $H^{2i}(X^{\textrm{an}},\C\, )$.
It is important to notice that the class
$c_{i}(E)^{\textrm{alg}}$ is mapped to $(2\pi \sqrt{-1})^{i}
c_{i}(E)^{\textrm{top}}$ by this morphism.
\par\medskip
Next, we consider the problem in the abstract analytic setting.
Let $X$ be a smooth complex analytic manifold and
$E$ be a holomorphic vector bundle on $X$.
We denote by
$\mathcal{A}^{p,q}_{\C}(X)$ the space of complex differential forms
of type $(p,q)$ on $X$ and we put $\mathcal{A}_{\C}(X)=
\bigoplus_{p,q}\mathcal{A}^{p,q}_{\C}(X)$. The Hodge filtration on $\mathcal{A}
_{\C}(X)$ is defined by
$F^{i}\mathcal{A}_{\C}(X)=\bigoplus_{p\geq i, q}
\mathcal{A}_{\C}^{p,q}(X)$. It induces a filtration $F^{i}H^{k}(X,\C\, )$
on $H^{k}(X,\C\, )$. For a detailed exposition see \cite[Ch.\!\! 7 and 8]{Vo1}. Let
$\Omega _{X}^{\mkern 2 mu{\ds\bullet}\vphantom{\geqslant i}}$
be the holomorphic de Rham complex on $X$. This is a complex of locally
free sheaves.
We can consider the analytic de Rham cohomology
$\HHH^{k+i}(X,\Omega_{X}^{{\ds\bullet}\geqslant i})$ which is the
hypercohomology of the truncated de Rham complex.
The maps of complexes
$\flgd{\Omega _{X}^{{\ds\bullet}\geqslant  i}}
{\Omega _{X}^{\mkern 2 mu{\ds\bullet}\vphantom{\geqslant i}}}$
and
$\flgd{\Omega _{X}^{{\ds\bullet}\geqslant  i}}{\Omega _{X}^{ \,  i}[-i]}$
give two maps
$\flgd{\HHH^{k+i}(X,\Omega _{X}^{{\ds\bullet}\geqslant  i})}
{F^{i}H^{k}(X,\C\, )}$
and
$\flgd{\HHH^{k+i}(X,\Omega _{X}^{{\ds\bullet}\geqslant  i})}
{H^{k}(X,\Omega _{X}^{ \,  i}).}$
In the compact K\"{a}hler case, the first map is an isomorphism, but it is no
longer true in the general case.
We will denote by
$H^{p,q}\be(X)$ the cohomology classes in $H^{p+q}(X,\C\, )$ which admit a
representative in $\mathcal{A}^{p,q}\be(X)$.
\par\medskip
If $E$ is endowed with the Chern connection associated to a
hermitian metric, the de Rham representative of $c_{i}(E)^{\textrm{top}}$
obtained by Chern-Weil theory
is of type $(i,i)$ and is unique modulo
$d(F^{i}\be\mathcal{A}_{X}^{2i-1})$.
This allows to define $c_{i}(E)$ in
$\HHH^{2i}\be(X,
\Omega _{X}^{{\ds\bullet}\geqslant  i})$, and then in $H^{i,i}(X)$
and $H^{i}(X,\Omega _{X}^{i})$. The notations for
these three classes will be
$c_{i}(E)^{\textrm{an}}$,
$c_{i}(E)^{\textrm{hodge}}$ and
$c_{i}(E)^{\textrm{dolb}}$.
\par\medskip
If we forget the holomorphic structure of $E$, we can consider its
topological Chern classes $c_{i}(E)^{\textrm{top}}$ in the Betti
cohomology groups
$H^{2i}(X,\Z)$. The image of $c_{i}(E)^{\textrm{top}}$ in
$H^{2i}(X,\C\, )$ is $c_{i}(E)^{\textrm{hodge}}$.
Thus $c_{i}(E)^{\textrm{top}}$ is an integral cohomology class whose
image in $H^{2i}(X,\C\, )$ lies in $F^{i}H^{2i}(X,\C\, )$. Such classes are
called \textit{Hodge classes} of weight $2i$.

The Chern classes
of $E$ in $H^{2i}(X,\Z)$, $F^{i}H^{2i}(X,\C\, )$, $\HHH^{2i}\bigl(X,
\Omega _{X}^{{\ds\bullet} \geqslant i }\bigr)$ and $H^{i}
\bigl(X,\Omega _{X}^{\, i}\bigr)$
appear
in the following diagram:
\[
\xymatrix@1@C=30pt@R=30pt
{
{}&
H^{2i}(X,\Z)\ni c_{i}(E)^{\textrm{{top}}}\ar[d]\\
c_{i}(E)^{\textrm{an}}\in
\HHH^{2i}(X,\Omega _{X}^{{\ds\bullet}\geqslant  i})\ar[r]\ar[d]
&
F^{i}H^{2i}(X,\C\, )\ni
c_{i}(E)^{\textrm{hodge}}\\
c_{i}(E)^{\textrm{dolb}}\in H^{i}(X,\Omega _{X}^{\, i})&
}
\]
\par\medskip
This means in particular that these different classes are compatible with the
Hodge decomposition in the compact K\"{a}hler case, and in general via the
$\flgd{\textrm{\!Hodge}}{\textrm{de Rham\!}}$ spectral sequence.
Furthermore, the knowledge of $c_{i}(E)^{\textrm{an}}$ allows to
obtain the two other classes $c_{i}(E)^{\textrm{dolb}}$ and
$c_{i}(E)^{\textrm{hodge}}$, but the converse is not true. Thus $c_{i}
(E)^{\textrm{an}}$ contains more information (except torsion) than the
other classes in the diagram.
\par\medskip
Recall now the Deligne cohomology groups $H^{p}_{D}(X,\Z(q))$
(see \cite[\S\,1]{EsVi} and Section \ref{SubSectionDeux}). We will be mainly
interested in the cohomology groups $H^{2i}_{D}(X,\Z(i))$. They admit
natural maps to $H^{2i}(X,\Z)$ and to
${\HHH^{2i}(X,\Omega _{X}^{ {\ds\bullet}\geqslant i})}$, which are
compatible with the diagram above. Furthermore, there is an exact
sequence
$\xymatrix{
0\ar[r]&\HHH^{2i-1}(X,
  \Omega _{X}^{  {\ds\bullet}\leqslant i-1})\, /
  _{\ds H^{2i-1}(X,\Z)}\ar[r]& H^{2i}_{D}(X,\Z(i))\ar[r]&H^{2i}(X,\Z),
}$ (see \cite{Vo1} and Proposition
\ref{PropositionUnComplementQuatreUnArt1} (i)). Thus a Deligne class is
a strong refinement of any of the above mentioned classes. The splitting
method works for the construction of $c_{i}(E)$ in
$H^{2i}_{D}(X,\Z(i))$, as explained in \cite[\S\,4]{Zuc}, and
\cite[\S\,8]{EsVi}.
These Chern classes, as all the others constructed above, satisfy the
following properties:
\begin{enumerate}
  \item [--] they are functorial with respect to pullbacks.

  \item [--] if $\sutrgd{E}{F}{G}$ is an exact sequence of vector
  bundles, then for all $i$
  \[
c_{i}(F)=\sum _{p+q=i}c_{p}(E)\,  c_{q}(G).
  \]
\end{enumerate}
The last property means that the total Chern class
$c=1+c_{1}+\cdots +c_{n}$
is
defined on the Grothendieck group $K(X)$ of holomorphic vector bundles
on $X$ and satisfies the additivity property
$c(x+x')=c(x)c(x')$.
\par\medskip
Now, what happens if we work with coherent sheaves instead of locally
free ones? If $X$ is quasi-projective and $\ff$ is an algebraic
coherent sheaf on
$X$, there exists a locally free resolution
\[
\xymatrix{
0\ar[r]&E_{1}\ar[r]&\cdots\ar[r]&\ar[r]E_{N}&\ff\ar[r]&0.
}
\]
(This is still true under the weaker assumption
that $X$ is a regular
separated scheme over $\C$ by Kleiman's lemma; see
\cite[II, 2.2.7.1]{SGA6}).
The total Chern class of $\ff$ is defined by
\[
c(\ff)=
c(E_{N}\he)\, c(E_{N-1}\he)^{-1}c(E_{N-2}\he)\dots
\]
The class $c(\ff)$ does not
depend on the locally free resolution (see \cite[\S\,4 and 6]{BoSe}).
More formally,
if $G(X)$ is the Grothendieck group of coherent sheaves on $X$, the
canonical map $\, \, \apl{\iota }{K(X)}{G(X)}$ is an isomorphism. The inverse
is given by $\flgd{[\ff]}{[E_{N}\he]-[E_{N-1}\he]+[E_{N-2}\he]-\dots }$
\par\medskip
In what follows, we consider the complex analytic case.
The problem of the existence of global locally free resolutions in
the analytic case has been opened for a long time. For
smooth complex
surfaces, such resolutions always exist by
\cite{Sch}. More recently, Schr\"{o}er and Vezzosi proved
in \cite {ScVe} the
same result for singular separated surfaces.
Nevertheless, for varieties of dimension at least $3$, a
negative answer to the question is provided by the following
counterexample of Voisin:
\par\medskip
\textbf{Theorem} \cite{Vo2}
\textit{On any generic complex torus of dimension greater than $3$,
the ideal sheaf of a point
does not admit a global locally free resolution.}
\par\medskip
Worse than that, even if $\ff$ admits a globally free resolution
$E^{\mkern 2 mu\ds\bullet}$, the method of Borel and Serre \cite{BoSe} does not
prove that $c(E_{N}\he)\, c(E_{N-1}\he)^{-1}c(E_{N-2}\he)\dots $
is independent of $E^{\mkern 2 mu\ds\bullet}$. In fact, the crucial point in their
argument is that \textit{every} coherent sheaf should have a resolution.
\par\medskip
Nevertheless Borel-Serre's method applies in a weaker context if we
consider Chern classes in $H\ee(X,\Z)$. Indeed, if $\ff$ is a coherent
sheaf on $X$ and $\mathcal{C}_{X}^{\omega }$ is the sheaf of
real-analytic functions on $X$, then $\ff\oti_{\oo_{X}\he}\mathcal{C}
_{X}^{\omega }$ admits a locally free real-analytic resolution by the
Grauert vanishing theorem \cite{Gra}. We obtain
by this method topological Chern classes
$c_{i}(\ff)^{\textrm{top}}$ in $H^{2i}(X,\Z)$.
\par\medskip
It is natural to
require that $c_{i}$ should take its values in more refined rings depending on
the holomorphic structure of $\ff$ and $X$. Such a construction has been
carried out by Atiyah for the Dolbeault cohomology ring in
\cite{Ati}. Let us briefly describe his method: the exact sequence
\[
\sutrgd{\ff\oti\Omega _{X}^{1}}{\wp_{X}^{1}(\ff)}{\ff}
\]
of principal parts of $\ff$ of order one
(see \cite[\S\,16.7]{EGA})
gives an extension class
(the Atiyah class) $a(\ff)$ in $\ext_{\oo_{X}\he}^{1}(\ff,\ff\oti\Omega
_{X}^{1})$. Then $c_{p}(\ff)$ is the trace of the $p$-th
Yoneda product of $a(\ff)$. These classes are used by O'Brian, Toledo and
Tong in \cite{OBToTo1} and \cite{OBToTo2} to prove the
Grothendieck-Riemann-Roch theorem on abstract manifolds in the Hodge
ring.
The Atiyah class has been constructed by Grothendieck and Illusie for
perfect complexes (see \cite[Ch.\!\! 5]{Il}).
Nevertheless, if $X$ is not a K\"{a}hler manifold,
there is no good relation between $H^{p}(X,\Omega _{X}^{\, p})$
and $\HHH^{2p}(X,\Omega _{X}^{  {\ds\bullet}\geqslant p})$, as the
Fr\"{o}licher spectral sequence may not degenerate at
$E_{1}$ for example.
\par\medskip
In this context, the most satisfactory construction was obtained
by Green in his unpublished thesis (see \cite{Gre} and
\cite{ToTo}). He proved the following theorem
\par\medskip
\textbf{Theorem 1} \cite{Gre}, \cite{ToTo}
\textit{Let $\ff$ be a coherent sheaf on $X$. Then there exist
Chern classes $c_{i}(\ff)^{\emph{Gr}}$ in
the analytic de Rham cohomology groups
$\HHH^{2i}(X,\Omega _{X}^{  {\ds\bullet}\geqslant i})$ which
are compatible with Atiyah Chern classes and topological Chern
classes.}
\par\medskip
In order to avoid the problem of nonexistence of locally free
resolutions, he introduced the notion of a
\textit{simplicial resolution} by simplicial vector bundles with
respect to a given covering. Green's basic result is the
following:
\par\medskip
\textbf{Theorem 2} \cite{Gre}, \cite{ToTo}
\textit{Any coherent sheaf on a smooth complex compact manifold
admits a finite simplicial resolution by simplicial holomorphic
vector bundles.}
\par\medskip
The next step in order to obtain Theorem 1 above, is to define
the Chern classes of a simplicial vector bundle. For this, Green
uses Bott's construction (see \cite{Bot}) which
can be adapted to the
simplicial context. Though, it is not clear how to extend
Green's method to Deligne cohomology.
\par\medskip
Let us now state the main result of this article:
\begin{theorem}\label{MainTheoremArt1}
Let $X$ be a complex compact manifold. For every coherent sheaf $\ff$ on
$X$, we can define classes $c_{p}(\ff)$
and $\ch_{p}(\ff)$ in $H_{\textrm{D}}^{2p}(X,\Q(p))
$
such that:
\begin{enumerate}\label{MainTheorem}
  \item [(i)] For every exact sequence
  $\sutrgd{\ff}{\g}{\hh}$ of coherent sheaves on $X$, we have $c(\g)=
  c(\ff)c(\hh)$ and $\ch(\g)=\ch(\ff)+\ch(\hh)$. The total Chern class
  $\apl{c}{G(X)}{H\ee_{D}(X,\Q)^{\tim}}$ is a group morphism and the
  Chern character $\apl{\ch}{G(X)}{H\ee_{D}(X,\Q)}$ is a ring morphism.

  \item [(ii)] If $\apl{f}{X}{Y}$
  is holomorphic and $y$ is an element of $
  \kan(Y)$, then $c(f\pe \be y)=f\ee c(y)$, where
  $\apl{f\pe}{\kan(Y)}{\kan(X)}$
  is the pullback in analytic \mbox{$K$-theory}.

  \item [(iii)] If $\eee$ is a locally free sheaf,
  then $c(\eee)$ is the usual Chern
  class in rational Deligne cohomology.

  \item [(iv)] If $Z$ is a smooth closed submanifold of $X$ and $\ff$ a
  coherent sheaf on $Z$, the Grothendieck-Riemann-Roch (GRR)
  theorem is valid
  for $\bigl(i_{Z}\he,\ff\bigr)$, namely
  \[
\ch\bigl(i_{Z*}\he\ff\bigr)=i_{Z*}\he\Bigl(\ch(\ff)\, \td
  \bigl(N_{Z/X}\he\bigr)^{-1}\Bigr).
\]
  \item [(v)] If $\apl{f}{X}{Y}$ is a projective morphism between smooth
  complex compact manifolds, for every coherent sheaf $\ff$ on $X$, we
  have the Grothendieck-Riemann-Roch theorem
  \[
\ch\bigl(f_{!}\he[\ff]\bigr)\td(Y)=f_{*}\he[\ch(\ff)]\td(X).
\]
  \end{enumerate}
\end{theorem}
Our approach is completely different from \cite{BoSe}. Indeed,
Voisin's result prevents from using locally free resolutions.
Our geometric starting point, which will be used instead of locally free
resolutions, is the following (Theorem \ref{TheoremeUnChClArt1}):
\par\medskip
\textbf{Theorem}\label{newtheo}
\textit{Let $\ff$ be a coherent sheaf on $X$
of generic rank $r$. Then there exists a bimeromorphic
morphism $\apl{\pi }{\tix}{X}$ and a locally free sheaf $\qq$ on
$\tix$ of rank $r$, together with a surjective map
$\flgd{\pi \ee\ff}{\qq}$.}
\par\medskip
It follows that up to torsion sheaves, $\pi \pe[\ff]$ is locally free.
This will allow us to define our Chern classes by induction on the
dimension of the base. Of course, we will need to show that our Chern
classes satisfy the Whitney formula and are independent of the
bimeromorphic model $\tix$.
\par\medskip
The theorem above is a particular case of
Hironaka's flattening theorem (see \cite{Hiro2} and in the
algebraic case \cite{GrRa}). Indeed, if we apply Hironaka's result to
the couple $(\ff,\id)$, there exists a bimeromorphic map
$\apl{\sigma }{\ti{X}}{X}$ such that $\sigma \ee\ff\bigm/{\bigl(\sigma \ee\ff
\bigr)_{
\textrm{tor}}\he}$ is flat with respect to the identity morphism, and thus
locally free. For the sake of completeness,
we include an elementary proof of Theorem
\ref{TheoremeUnChClArt1}.
\par\medskip
Property (iv) of Theorem \ref{MainTheorem}
is noteworthy.
The lack of global resolutions
(see \cite{Vo2}) prevents from using the proofs of
Borel, Serre and of Baum, Fulton and McPherson (see \cite[Ch.\!\! 15 \S\,2]{Ful}).
The equivalent formula in the
topological setting is proved in \cite{AtHi}. In the holomorphic
context, O'Brian, Toledo and Tong (\cite{OBToTo3}) prove this formula
for the Atiyah Chern classes when there exists a
retraction from $X$ to $Z$, then they establish (GRR) for a projection and
they deduce that (GRR) is valid for any holomorphic map, so
\textit{a posteriori} for an immersion
(see \cite{OBToTo2}). Nevertheless,
our result does not give a new proof of (GRR) formula for an immersion in the
case of the Atiyah Chern classes. Indeed, the compatibility
between our
construction and the Atiyah Chern classes is a consequence of the (GRR)
theorem for an immersion in both theories, as explained further.
\par\medskip
Property (v) is an immediate consequence of (iv), as originally noticed
in \cite{BoSe}, since the natural map from $G(X)\oti_{\Z}\he G(\P^{N})$
to $G(X\tim \P^{N})$ is surjective (see \cite[Expos\'{e} VI]{SGA6}
and \cite{Bei}).
Yet, we do not obtain the (GRR)
theorem for a general holomorphic map between smooth complex compact
manifolds.
\par\medskip
Remark that $c_{p}(\ff)$ is only constructed in the rational
Deligne
cohomology group $H_{\textrm{D}}^{2p}(X,\Q(p))$.
The reason is that we
make full use of the Chern character, which has denominators, and
thus determines the total Chern class only up to torsion classes.
We think that it could be
possible to define $c_{p}(\ff)$ in $H_{\textrm{D}}^{2p}(X,\Z(p))$
following our approach, but with huge computations.
For $p=1$,
$c_{1}(\ff)$ can be easily constructed in
$H^{2}_{D}(X,\Z(1))$.
Indeed,
it suffices to define $c_{1}(\ff)=c_{1}(\det\ff)$, where
$\det\ff$ is the determinant line bundle of $\ff$
(see \cite{Knu}).
\par\medskip
It is interesting to compare the classes of
Theorem \ref{MainTheoremArt1}
with other existing theories. We adopt a more general setting by
using Theorem \ref{TheoremeUnChClArt1}.
We prove that, for any cohomology ring
satisfying
reasonable properties, a theory of Chern classes can be
completely determined if
we suppose that the GRR formula is
valid for immersions. More precisely,
our statement is the following
(Theorem \ref{ThPartSixTrois}):
\par\medskip
\textbf{Theorem}
\textit{Under the hypotheses $(\alpha )$-$(\delta )$
of page \pageref{unicity} on the cohomology ring,
a theory of Chern classes for coherent sheaves on smooth complex
compact manifolds which satisfies the GRR theorem for
immersions,
the Whitney additivity formula and the functoriality formula
is completely determined by the first Chern class of holomorphic
line bundles.}
\par\medskip
This theorem yields compatibility results
(Corollary \ref{DernierCorollaire}):
\par\medskip
\textbf{Corollary}
\textit{The classes
of Theorem \ref{MainTheoremArt1} are compatible with the
rational topological Chern classes and the Atiyah Chern classes.}
\par\medskip
Nevertheless, since GRR for immersions does not seem to be known
for the Green Chern classes if $X$ is not K\"{a}hler, the
theorem above does not give the compatibility in this setting.
In fact, the compatibility is \emph{equivalent} to the GRR
theorem for immersions for the Green Chern classes.
\par\medskip
Let us mention the link of our construction with secondary
characteristic classes.
\par\medskip
We can look at a subring of the ring of Cheeger-Simons characters on $X$
which are the ``holomorphic'' characters (that is the
\mbox{$G$-cohomology} defined in \cite[\S\,4]{Esn}, or equivalently the
restricted differential characters defined in \cite[\S\,2]{Bry}). This
subring can be mapped onto the Deligne cohomology ring, but not in an
injective way in general. When $E$ is a holomorphic vector bundle with a
compatible connection, the Cheeger-Simons theory (see \cite{ChSi})
produces Chern classes with values in this subring. It is known that
these classes are the same as the Deligne classes (see \cite{Bry} in the
algebraic case and \cite[\S\,5]{Zuc} for the general case). When $E$ is
topologically trivial, this construction gives the so-called secondary
classes with values in the intermediate jacobians of $X$ (see \cite{Nad}
for a different construction, and \cite{Ber} who proves the link with the
generalized Abel-Jacobi map). The intermediate jacobians of $X$ have been
constructed in the K\"{a}hler case by Griffiths. They are complex tori
(see \cite[Ch.\!\! 12 \S\,1]{Vo1}, and \cite[\S\,7 and 8]{EsVi}). If $X$ is
not K\"{a}hler, intermediate jacobians can still be defined but they are
no longer complex tori.
\par\medskip
Our result provides similarly refined Chern classes for coherent
sheaves, and in particular, secondary invariants for coherent sheaves
with trivial topological Chern classes.
\par\medskip
The organization of the paper is the following. We recall
in Part \ref{SectionDeux} the
basic properties of Deligne cohomology and Chern classes for locally
free sheaves in Deligne cohomology. The necessary results of analytic
\mbox{$K$-theory} with support are grouped in Appendix
\ref{AppendixArt1}; they will be
used extensively throughout the paper. The rest of the article is devoted
to the proof of Theorem \ref{MainTheoremArt1}. The construction of the
Chern classes is achieved by induction on $\dim X$. In Part
\ref{SectionConstructionOfChernClassesArt1} we perform the induction step
for torsion sheaves using the (GRR) formula for the immersion of
smooth divisors;
then we prove a d\'{e}vissage theorem which enables us to break any
coherent sheaf into a locally free sheaf and a torsion sheaf on a
suitable modification of $X$; this is the key of the construction of
$c_{i}(\ff)$ when $\ff$ has strictly positive rank. The Whitney formula,
which is a part of the induction process, is proved in Part
\ref{SectionWhitneyFormulaArt1}. After several reductions, we use a
deformation argument which leads to the deformation space of the
normal cone of a smooth hypersurface. We establish the (GRR) theorem for
the immersion of a smooth hypersurface in Part
\ref{SectionConstructionOfChernClassesArt1}, in Part
\ref{GRRForAnImmersionClassesArt1} we recall how to deduce the general
(GRR) theorem for an immersion from this particular case, using excess
formulae, then we deduce uniqueness results using
Theorem \ref{TheoremeUnChClArt1}.
\par\bigskip
\textbf{Acknowledgement.} I wish to thank Claire Voisin for introducing
me to this beautiful subject and for many helpful discussions.
I also thank Pierre Schapira for enlightening conversations.
\section{Notations and conventions}

All manifolds are complex smooth analytic connected manifolds. All the
results are clearly valid for non connected ones, by reasoning on the
connected components.
\par\medskip
Except in Section \ref{SectionDeux}, all manifolds are compact. By
\textit{submanifold}, we always mean a \textit{closed} submanifold.
\par\medskip
\textit{Holomorphic vector bundles}
\par\medskip
The rank of a holomorphic vector bundle is well defined since the manifolds are
connected. If $E$ is a holomorphic vector bundle, we denote by $\eee$
the associated locally free sheaf.
The letter $\eee$ always denotes a locally free sheaf.
\par\medskip
\textit{Coherent sheaves}
\par\medskip
If $\ff$ is a coherent analytic sheaf on
a smooth connected manifold $X$, then $\ff$ is locally free
outside a proper analytic subset $S$ of $X$ (see \cite{GrRe}). Then
$U=X\backslash S$ is connected.
By definition, the generic rank of $\ff$ is the rank of
the locally free sheaf $\ff_{\vert U}\he$.
If the generic rank of $\ff$ vanishes, $\ff$ is supported in a proper
analytic subset $Z$ of $X$, it is therefore annihilated by the action of a
sufficiently high power of the ideal sheaf $\ii_{Z}$.
Conversely, if $\ff$ is a torsion sheaf, $\ff$ is identically zero
outside a proper analytic subset of $X$, so it has generic rank zero.
The letter $\ttt$ always denotes a torsion sheaf.
\par\medskip
\textit{Divisors}
\par\medskip
A strict normal crossing divisor $D$ in $X$ is a formal sum
$m_{1}D_{1}+\cdots +m_{N}\he D_{N}\he$, where $D_{i}$, $1\leq i\leq N$,
are smooth transverse hypersurfaces and $m_{i}$, $1\leq i\leq N$, are nonzero
integers. If all the coefficients $m_{i}$ are positive, $D$ is
\textit{effective}. In that case, the associated
reduced divisor $D^{\textrm{red}}\be$ is the effective divisor
$D_{1}+\cdots +D_{N}\he$. A strict simple normal crossing divisor is
\textit{reduced} if for all $i$, $m_{i}=1$. We make no
difference between a reduced divisor and its support.
\par\smallskip
By a \textit{normal crossing divisor} we always
mean a \textit{strict} normal crossing divisor.
If $D$ is an effective simple normal crossing divisor, it defines an
ideal sheaf\ \,$\ii_{D}\he=\oo_{X}\he(-D)$. The associated quotient sheaf is
de\-no\-ted by $\oo_{D}\he$.
\par\smallskip
We use frequently Hironaka's desingularization theorem \cite{Hiro1} for
complex spaces as stated in \cite[Th.\!\! 7.9 and 7.10]{AG}.
\par\medskip
\textit{\textrm{Tor} sheaves}
\par\medskip
Let $\apl{f}{X}{Y}$ be a holomorphic map and $\ff$ be a coherent sheaf
on $Y$. We denote by $\tor_{i}(\ff,f)$ the sheaf
$\tor_{i}^{f^{-1}\oo_{Y}\he}\bigl(f^{-1}\ff,\oo_{X}\he\bigr)$.
\par\medskip
\textit{Grothendieck groups}
\par\medskip
The Grothendieck group of coherent analytic sheaves
(resp.\! of torsion coherent analytic sheaves) on a comp\-lex space $X$
is denoted by $G(X)$ (resp.\! $G_{\textrm{tors}}\he(X)$).
If $\ff$ is a coherent analytic sheaf on $X$, $[\ff]$ denotes its class
in $G(X)$.
The notation $G_{Z}\he(X)$ is defined in Appendix
\ref{AppendixArt1}. In order to avoid subtle confusions, we never
use here the Grothendieck group of locally free sheaves.
\section{Deligne cohomology and Chern classes for locally free
sheaves}
\label{SectionDeux}
In this section, we will expose the basics of Deligne cohomology for the
reader's convenience. For a more detailed exposition, see
\cite[\S\,1, 6, 7, 8]{EsVi},
\cite[Ch.\!\! 12]{Vo1}, and \cite{ElZu}.
\subsection{Deligne cohomology}\label{SubSectionDeux}
\begin{definition}\label{DefinitionUnComplementQuatreUnArt1}
Let $X$ be a smooth complex manifold and let $p$ be a
nonnegative integer. Then
\begin{enumerate}
  \item [--] The Deligne complex $\Z_{D,X}\he(p)$ of $X$ is the
  following complex of sheaves
  \[
\xymatrix{
  0\ar[r]&\Z_{X}\he\ar[r]^{(2i\pi )^{p}}
  &\oo_{X}\he\ar[r]^{d}&\ \cdots \ \ar[r]^{d}&\Omega
  ^{p-1}_{X},
  }
\]
   where $\Z_{X}\he$ is in degree zero. Similarly, the rational Deligne
  complex $\Q_{D,X}\he(p)$ is the same complex as above with $\Z_{X}\he$
  replaced by $\Q_{X}\he$.
  \par\bigskip
  \item [--] The Deligne cohomology groups $H_{D}^{i}(X,\Z(p))$ are
  the hypercohomology groups defined by
  \[
H^{i}_{D}(X,\Z(p))=\HHH^{i}(X,\Z_{D,X}\he(p)).
\]
  The rational Deligne cohomology groups are defined by the same formula
  as the hypercohomology groups of the rational Deligne complex.

  \item [--] The same
  definition holds for the Deligne cohomology
  with support in a closed subset $Z$:
  \[
H^{i}_{D,Z}(X,\Z(p))=\HHH^{i}_{Z}(X,\Z_{D,X}\he(p)).
\]
  \item [--] The total Deligne cohomology group of $X$ is
  $H\ee_{D}(X)=\bop_{k,p}H^{k}_{D}(X,\Z(p))$. We will denote by
  $H\ee_{D}(X,\Q)$ the total rational Deligne cohomology group.
\end{enumerate}
\end{definition}
\begin{example}
${}\he$
\begin{enumerate}
  \item [--] $H_{D}^{i}(X,\Z(0))$ is the usual Betti cohomology
  group $H^{i}(X,\Z)$.

  \item [--] $\Z_{D,X}\he(1)$ is quasi-isomorphic to $\oo\ee_{X}[-1]$ by the
  exponential exact sequence. Thus we have a group isomorphism
  $H^{2}_{D}(X,\Z(1))\simeq H^{1}\bigl(X,\oo\ee_{X}\bigr)\simeq
  \textrm{Pic}(X)$. The first Chern class of a line bundle $L$ in
  $H_{D}^{2}(X,\Z(1))$ is the element of
  $\textrm{Pic}(X)$ defined by $c_{1}(L)=\{L\}$.

  \item [--] $H^{2}_{D}(X,\Z(2))$ is the group of flat holomorphic line
  bundles, i.e.  holomorphic line bundles with a holomorphic
  connection (see \cite[\S\,1]{EsVi} and \cite{Esn}).

\end{enumerate}
\end{example}
For geometric interpretations of higher
  Deligne cohomology groups, we refer the reader to
  \cite{Gaj}.
  \par\medskip
Some fundamental properties of Deligne cohomology are listed below:
\begin{proposition}\label{PropositionUnComplementQuatreUnArt1}
${}\he$
\begin{enumerate}
  \item [(i)] We have an exact sequence
  $\sutrgd{\Omega _{X}^{  {\ds\bullet}\leqslant p-1}[-1]}
  {\Z_{D,X}\he(p)}{\Z_{X}\he}\!.$
\par\medskip
  In particular,
  $H^{2p}_{D}(X,\Z(p))$ fits into the exact sequence
  \[
\xymatrix{
  H^{2p-1}\be(X,\Z)\ar[r]&
  \HHH^{2p-1}\bigl(X,
  \Omega _{X}^{  {\ds\bullet}\leqslant p-1}\bigr)
  \ar[r]&
  H^{2p}_{D}(X,\Z(p))\ar[r]&H^{2p}(X,\Z)
  }\!.
\]

  \item [(ii)] The complex $\Z_{D,X}\he(p)[1]$ is quasi-isomorphic to
  the cone of the morphism
  \[
\xymatrix@C=10ex{\Z_{X}\he\oplus \Omega _{X}^{ {\ds\bullet}
   \geqslant p}
  \ar[r]^-{(2i\pi )^{p}\!,\, i}&\Omega _{X}^{\mkern 2mu {\ds\bullet}
  \vphantom{\geqslant i}}
  }.
\]
  Thus we have a long exact sequence:
  \[
 \xymatrix@C=18pt{
  \cdots \ar[r]&H^{k-1}(X,\C\, )\ar[r]&H^{k}_{D}(X,\Z(p))\ar[r]&\HHH^{k}\bigl(
  X,\Omega _{X}^{  {\ds\bullet}\geqslant p}\bigr)
  \oplus H^{k}(X,\Z)\ar[r]&H^{k}(X,\C\, )\ar[r]&\cdots
  }
\]
  and a similar exact sequence can be written with support in a closed subset
  $Z$.

  \item [(iii)] A cup-product
  \[
\flgd{H_{D}^{i}(X,\Z(p))\oti_{\Z}\he H^{j}_{D}(X,\Z(q))}
  {H^{i+j}_{D}(X,\Z(p+q))}
\]
  is defined
  and endows $H\ee_{D}(X)$ with a ring
  structure.


  \item [(iv)] If $\apl{f}{X}{Y}$ is a holomorphic map between two smooth
  complex manifolds, we have a pullback morphism
  $\apl{f\ee}{H_{D}^{i}(Y,\Z(p))}{H_{D}^{i}(X,\Z(p))}$ which is a ring
  morphism.


  \item [(v)] If $X$ is smooth, compact, and if $E$ is a holomorphic
  vector bundle on $X$ of rank $r$, then $H\ee_{D}(\P(E))$ is a free
  \mbox{$H\ee_{D}(X)$-module} with basis $1,\, c_{1}(\oo_{E}\he(1)),\dots ,
  c_{1}(\oo_{E}\he(1))^{r-1}$.

  \item [(vi)] For every $t$ in
  $\P^{1}$, let $j_{t}$ be the inclusion $
  \xymatrix{X\simeq X\tim\{t\}\, \ar@{^{(}->}[r]&X\tim \P^{1}}$. Then the
  pullback morphism $\apl{j_{t}\ee}{H\ee_{D}(X\tim\P^{1})}{H\ee_{D}(X)}$
  is independent of $t$ (homotopy principle).

%
%
\end{enumerate}
\end{proposition}
The assertions (i) and (ii) are obvious.
The cup product in (iii) comes from a morphism of complexes
$\mkern -4 mu\flgd{\Z(p)\oti_{\Z}\he\Z(q)}{\Z(p+q)}$, see
\cite[\S\,1]{EsVi}.
This morphism if functorial with respect to pullbacks, which
gives (iv).
Property (v) is proved by d\'{e}vissage using the exact
sequence
\[
\sutrgd{\Omega _{X}^{p-1}[-p]}{\Z_{D,X}\he(p)}{\Z_{D,X}\he(p-1)}
\]
and the five lemma (see \cite[\S\,8]{EsVi}).
Property (vi) is a consequence of (v): if
$\alpha $ is a Deligne class in $H\ee_{D}(X\tim\P^{1})$, we can write
$\alpha =\textrm{pr}\ee_{1}\lambda +\textrm{pr}_{1}\ee\mu \, .\,
c_{1}\bigl(\textrm{pr}_{2}\ee\oo_{\P^{1}}\he(1)\bigr)$. Thus
$j_{t}\ee\alpha =\lambda $.
\par\medskip
Remark that (vi) is false if we replace $\P^{1}(\C\, )$ by
$\C$, in contrast with the algebraic case. Indeed, take an
elliptic curve $S$ and choose an isomorphism
$\apl{\phi }{S}{\textrm{Pic}^{0}(S)}$. There is a universal
line bundle $\LL$ on
$S\times S$ such that
for all $x$ in
$S$,
$\LL_{\vert S\times x}\he\simeq\phi (x)$. Let
$\apl{\pi }{\C}{S}$ be the universal covering map of $S$.
Consider the class $\alpha =c_{1}\bigl[(\textrm{id} ,
\pi )\ee\be\LL\bigr]$, then for all
$t$ in $\C$, $j_{t}\ee\alpha =c_{1}\bigl[\phi (\pi (t))\bigr]
$.
\par\bigskip
We will now consider more refined properties of Deligne
cohomology.
\begin{proposition}[see \mbox{\cite[\S\,2]{ElZu}}]\label{PropositionUnInsertC}
${}\he$
\begin{enumerate}
  \item [(i)] If $X$ is a smooth complex manifold and $Z$ is a smooth
  subma\-nifold of $X$ of codimension $d$, there exists a cycle class $\{Z\}_{D}
  \he$ in $H_{D,Z}^{2d}(X,\Z(d))$ compatible with the Bloch cycle class
  (see \cite[\S\,5]{Blo}, \cite[\S\,6]{EsVi})
  and the topological cycle
  class.
  If $Z$ and $Z'$ intersect transversally, $\{Z\cap Z'\}_{D}\he
  =\{Z\}_{D}\he \, .\, \{Z'\}_{D}\he$.
  If $Z$ is a smooth hypersurface of $X$, the image of
  $\{Z\}_{D}\he$ in $H^{2}_{D}(X,\Z(1))\simeq\emph{Pic}(X)$ is
  the class of $\oo_{X}\he(Z)$.
  \item [(ii)]
  More generally, let $\apl{f}{X}{Y}$ be a proper holomorphic map
  between smooth complex manifolds and $d=\dim Y-\dim X$. Then there
  exists a Gysin morphism
  \[
\apl{f_{*}\he}{H_{D}^{2p}(X,\Z(q))}{H_{D}^{2(p+d)}(Y,\Z(q+d))}
\]
  compatible with the usual Gysin morphisms in integer and analytic de
  Rham cohomology. If $Z$ is a smooth submanifold of codimension
  $d$ of $X$ and $\apl{i_{Z}\he}{Z}{X}$ is the canonical
  inclusion, then $i_{Z*}\he(1)$ is the image of
  $\{Z\}_{D}\he$ in $H^{2d}_{D}(X,\Z(d))$.
\end{enumerate}
\end{proposition}
\par\medskip
The point (i) is easy to understand.
By Proposition \ref{PropositionUnComplementQuatreUnArt1} (ii),
since $H_{Z}^{2d-1}(X,\Z)=0$, we have an exact sequence
\[
\sutrgpt{H_{D,Z}^{2d}(X,\Z(d))}{H_{Z}^{2d}(X,\Omega_{X}^{{\ds\bullet}
\geqslant d})\oplus H_{Z}^{2d}(X,\Z)}{H_{Z}^{2d}(X,\C\, )}{.}
\]
The couple $\bigl((2i\pi )^{d}\{Z\}_{\textrm{Bloch}}\he,
\{Z\}_{\textrm{top}}\he\bigr)$
is mapped to $0$ in
$H_{Z}^{2d}(X,\C\, )$ (see \cite[\S\,7]{EsVi}). Therefore, it defines a unique
element $\{Z\}_{D}\he$ in $H_{D,Z}^{2d}(X,\Z(d))$.
\par\medskip
For (ii), we introduce the sheaves $\dd_{X,\Z}^{k}$ of locally integral
currents of degree $k$ as done in \cite[\S\,2]{ElZu}, and \cite[\S\,2.2]{GiSo}.
These sheaves are, in a way to be properly defined,
a completion of the currents induced
by smooth integral chains on $X$
(see \cite[\S\,2.1]{Kin} and
\cite[\S\,4.1.24]{Fed}). Then
\begin{enumerate}
  \item [--] $\dd_{X,\Z}^{\mkern 2 mu{\ds\bullet}\vphantom{\geqslant i}}$
  is a soft resolution of $\Z_{X}\he$.

  \item [--] $\dd_{X,\Z}^{k}$ is a subsheaf of $\dd_{X}^{k}$ stable
  by
  push-forward under proper $C^{\infty }\be$ maps,
  where $\dd_{X}^{k}$ is the sheaf
  of usual currents of degree $k$ on $X$.

\end{enumerate}
Thus $\Z_{D,X}\he(p)[1]$ is quasi-isomorphic to the cone of the morphism
$\aplpt{\bigl([2i\pi ]^{p},i\bigr)}
{\dd_{X,\Z}^{\mkern 2 mu{\ds\bullet}\vphantom{\geqslant i}}
\oplus F^{p}\be\dd_{X}^{\mkern 2 mu{\ds\bullet}\vphantom{\geqslant i}}}
{\dd_{X}^{\mkern 2 mu{\ds\bullet}\vphantom{\geqslant i}}}{.}$
We will denote by $\ti{\Z}_{D,X}\he(p)$ this cone shifted by
minus one.
Since the sheaves $\dd_{X,\Z}^{k}$, $F^{p}\be\dd_{X}^{k}$ and
$\dd_{X}^{k}$ are acyclic, $Rf_{*}\he\Z_{D,X}\he(p)[1]$ is
quasi-isomorphic to the cone of the morphism
$\aplpt{\bigl([2i\pi ]^{p},i\bigr)}
{f_{*}\he\dd_{X,\Z}^{\mkern 2 mu{\ds\bullet}\vphantom{\geqslant i}}
\oplus f_{*}\he F^{p}\be \dd_{X}^{\mkern 2 mu{\ds\bullet}\vphantom{\geqslant i}}}
{f_{*}\he\dd_{X}^{\mkern 2 mu{\ds\bullet}\vphantom{\geqslant i}}}{.}$
The push-forward of currents by $f$ gives an explicit morphism
$\apl{f_{*}\he}{f_{*}\he\ti{\Z}_{D,X}\he(p)}{\ti{\Z}_{D,Y}
\he(p+d)[2d]}$ and then a morphism
$\apl{f_{*}\he}{Rf_{*}\he{\Z}_{D,X}\he(p)}{{\Z}_{D,Y}
\he(p+d)[2d]}$ in the derived category $\mathcal{D}^{b}\be
\bigl(\textrm{Mod}(\Z_{Y}\he)\bigr)$.
We get
the Gysin morphism by taking the hypercohomology on $Y$.
\par\medskip
The compatibility between $\{Z\}_{D}\he$ and $i_{Z*}\he(1)$
is shown in \cite{ElZu}.
\par\medskip
We now state all the properties of the Gysin morphism needed here. The
points (vi) and (vii) use Chern classes
of vector bundles. They will be defined in the
next section.

\begin{proposition}\label{PropositionDeuxInsertC}
${}\he$
\begin{enumerate}
  \item [(i)] $f_{*}\he$ is compatible with the composition of maps and
  satisfies the projection formula.
  \[
f_{*}\he\bigl(x\, .\, f\ee\be\, y\bigr)=f_{*}\he x\, .\, y.
\]
In particular, if\quad $\Gamma _{f}\he\subseteq X\times Y$ is the
graph of $f$ and if $X$ is compact, then for every Deligne class
$\alpha $ in $X$,
\[
f_{*}\he\alpha =p_{2*}\he\Bigl(p_{1}\ee\alpha \, .\,
\{\Gamma _{f}\he\}_{D}\he\Bigr).
\]
  \item [(ii)] Consider the cartesian diagram
  \[
\xymatrix@R=30pt@C=40pt{
Y\tim Z\, \ar@{^{(}->}[r]^{i_{Y\tim Z}\he}\ar[d]_{p}&X\tim Z\ar[d]^{q}\\
Y\, \ar@{^{(}->}[r]_{i_{Y}\he}&X
}
\]
Then\quad  $q\ee\be i_{Y*}\he=i_{Y\tim Z*}\he \, p\ee\be$.

  \item [(iii)] If $\apl{f}{X}{Y}$ is proper and generically finite of degree $d$,
  then $f_{*}\he f\ee\be=d\tim\id$.

  \item [(iv)] Consider the cartesian diagram, where
  $Y$ and $Z$ are compact and intersect transversally:
  \[
\xymatrix@R=30pt@C=50pt{
W\,\, \be \ar@{^{(}->}[r]^{i\he_{W\to Y}}
\ar@{_{(}->}[d]_{i\he_{W\to Z}}
&Y\be\ar@{_{(}->}[d]^{i_{Y}\he}
\\
Z\, \, \ar@{^{(}->}[r]_{i_{Z}\he}&X
}
\]
Then\quad
$i_{Y}\ee\, i_{Z*}\he =i_{\flcourtegd{\st W}{\st Y*}}\he i\ee
_{\flcourtegd{\st W}{\st Z}}$.

\item[(v)] Let $\apl{f}{X}{Y}$ be a surjective map
between smooth complex compact manifolds, and let $D$ be a
smooth hypersurface of $Y$ such that $f^{-1}\be (D)$ is a simple normal
crossing divisor. Let
us write $f\ee\be D=m_{1}\tid_{1}+\cdots +m_{N}\he\tid_{N}\he$. Let
$\apl{\ba{f}_{i\he}}{\tid_{i}\he}{D}$ be
the restriction of $f$ to
$\tid_{i}\he$. Then
\[
f\ee\be\, i_{D*}\he=\ds\sum _{i=1}^{N}m_{i}\,
i_{\tid_{i}*}\he \ba{f}_{i}{}\ee.
\]
\item [(vi)] Let $X$ be compact, smooth, and let $Y$ be a smooth
submanifold of codimension $d$ of $X$. Let $\tix$ be the blowup of $X$
along $Y$, as shown in the following diagram:
\[
\xymatrix{
E\ar[r]^{j}\ar[d]_{q}&\tix\ar[d]^{p}\\
Y\ar[r]_{i}&X
}
\]
Then we have an isomorphism
\[
\xymatrix@R=-2ex@C=10ex{
H\ee_{D}(X)\oplus\bop_{i=1}^{d-1}
H\ee_{D}(Y)\ar[r]&H\ee_{D}(\ti{X})\hphantom{\!\!
\ds\sum _{i=1}^{p-1}q\ee y_{i}\pti{E}\{E\}_{D}^{i}\pti{E}\{E\}_{D}^{i}}\\
\hphantom{H\ee_{D}(\, }
\bigl(x,(y_{i})_{1\leq i\leq d-1}\he\bigr)\ar@{|->}[r]&p\ee x+
\ds\sum _{i=1}^{d-1}j_{*}\he\Bigl[y_{i}\he\, c_{1}\bigl(
\oo_{N_{Y/X}\he}\he(-1)\bigr)^{i-1}\Bigr].
}
\]
  In particular, if $\alpha $ is a Deligne class on $X$ such that $j\ee\be\alpha$
  is the pullback of a Deligne class on $Y$, then $\alpha $ is the
  pullback of a unique Deligne class on $X$.
  \[
{}\he
\]
  Moreover, if $F$ is the excess conormal bundle defined by the exact
  sequence
  \[
\sutrgdpt{F}{q\ee\be N_{Y/X}\ee}{N_{E/\tix}\ee}{,}
\]
  we have the excess formula $p\ee\be \, i_{*}\he\, \alpha =
  j_{*}\he\bigl(q\ee\be \alpha \, \, c_{d-1}(F\ee\be)\bigr)$.

\item[(vii)] If $Y$ is a smooth compact submanifold of $X$ of codimension $d$,
we have the auto-intersection formula
\[
i_{Y}\ee\, i_{Y*}\he\alpha
=\alpha \, {c}_{d}
\bigl(N_{Z/X}\he\bigr).
\]
\end{enumerate}
\end{proposition}
\begin{proof}
(i) and (ii)  hold at the
level of complexes. More precisely, the Gysin morphism $f_{*}\he$
is functorial
at the complexes level $\ti{\Z}_{D}\he(\, .\, )$. For the
projection formula, we use the complex $\ti{\Z}_{D,X}\he(\, .\, )$
for the variable $x$ and
the complex $\Z_{D,Y}\he(\, .\, )$ for the variable
$y$. To prove the last formula of (i), we remark that
$p_{2}$ is proper since $X$ is compact. Then we write
\begin{align*}
f_{*}\he\alpha &=p_{2*}\he(\textrm{id},f)_{*}\he\alpha =
p_{2*}\he(\textrm{id},f)_{*}\he\bigl[(\textrm{id},f)\ee\be
p_{1}\ee\alpha \bigr]\\
&=p_{2*}\he\bigl(p_{1}\ee\alpha \, .\, (\textrm{id},f)_{*}\he(1)\bigr)=
p_{2*}\bigl(p_{1}\ee\alpha \, .\, \{\Gamma
_{f}\he\}_{D}\he\bigr).
\end{align*}
For (ii), we can pull-back currents under $p$ and $q$ since
these morphisms are submersions. Then $p\ee\be$ and
$q\ee\be$ are defined for the complexes $\ti{\Z}_{D}\he(\, .\,
)$ and (ii) is an equality of complexes morphisms.
\par\medskip
For (iii), it is enough by the projection formula to prove that
$f_{*}\he(1)=d$, which is well known.
\par\medskip
The formulae (iv) and (v) are of the same type.
Let us prove (v) for
instance. We will define first some notations:
Let $\Gamma $ be the graph of $i_{D}\he\mkern -4 mu:\mkern - 8 mu
\xymatrix{D\, \ar@{^{(}->}[r]&\!Y}$\!\! and
$\Gamma _{i}\he$ be the graph
 of   $i_{\tid_{i}}\he\mkern -4 mu:\mkern - 8 mu
\xymatrix{\tid_{i}\, \ar@{^{(}->}[r]&\!X}$\!\!. We define
$\Gamma'_{i} =
  \bigl(\ba{f}_{i}\he,\id\bigr)_{*}\he\bigl(\Gamma _{i}\he\bigr)\suq D\times
  X$.
We call $\apl{p_{1}}{D\tim Y}{D}$ and
  $\apl{p_{2}}{D\tim Y}{Y}$ the first and second projections. In the
  same manner, we define the projections $\aplpt{p_{1}'}{D\tim X}{D}{,}$
  $\aplpt{p_{2}'}{D\tim X}{X}{,}$
$\aplpt{p_{1,i}'}{\tid_{i}\tim X}{\tid_{i}}{,}$ and
$\aplpt{p_{2,i}'}{\tid_{i}\tim X}{X}{.}$


%
%
%
%
%
%
%
%
We have
$(\id,f)\ee\be\{\Gamma \}_{D}\he=\ds\sum _{i=1}^{N}m_{i}\he\, \{\Gamma
'_{i}\}_{D}\he$
(this can be seen using explicit
description of the Bloch cycle class, see \cite{Blo}). Then
\begin{align*}
f\ee\be\, i_{D*}\he\, \alpha &=f\ee\be\,  p_{2*}\he
\, \bigl(p_{1}\ee\alpha\,  .\, \{\Gamma \}_{D}\he\bigr)
&&\textrm{by (i)}\\
&=p'_{2*}(\id,f)\ee\be\bigl(p_{1}\ee\, \alpha \, .\, \{\Gamma \}_{D}
\he\bigr)&&\textrm{by (ii)}\\
&=\sum _{i=1}^{N}m_{i}\he\, p'_{2*}\bigl(p'_{1}{}{\ee}\alpha \, .\,
\{\Gamma '_{i}\}
_{D}\he\bigr)&&\textrm{by the projection formula}\\
&=\sum _{i=1}^{N}m_{i}\he\, p'_{2*}\bigl(p'_{1}{}{\ee}\alpha \, .\,
(\ba{f}_{i}\he,\id)_{*}\he\{\Gamma _{i}\}
_{D}\he\bigr)\\
&=\sum _{i=1}^{N}m_{i}\he\, p_{2,i*}'\bigl(p'_{1,i}{}\ee
\ba{f}_{i}{}\ee\alpha \, .\,
\{\Gamma _{i}
\}_{D}\he\bigr)&&\textrm{by the projection formula}\\
&=\sum _{i=1}^{N}m_{i}\he\, i_{\tid_{i}*}\he\, \ba{f}_{i}{}\ee\, \alpha .
\end{align*}
\par\medskip
Before dealing with (vi), we prove (vii) when $Y$ is a hypersurface.
In the case of the \'{e}tale cohomology, it is possible to
assume that $\alpha =1$ (see
\cite[Expos\'{e} VII, \S\,4]{SGA5} and \cite[Cycle \S\,1.2]{SGA4}).
Remark that this is no longer possible here, for there is
no purity theorem in Deligne cohomology.
\par\medskip
We use the deformation
to the normal cone, an idea which goes back to Mumford
(see \cite{MuLS} and \cite[Expos\'{e} VII \S\,9]{SGA5}).
The aim was originally to prove the same formula in the Chow
groups.
Let $M_{Y/X}\he$ be the blowup of $X\tim \P^{1}$ along
$Y\tim \{0\}$, $\tix$ be the blowup of $X$ along $Y$, and
$M_{Y/X}^{\circ}=M_{Y/X}\he\backslash \tix$.
Then we have an injection
$
F\mkern -4 mu:\mkern - 8 mu\xymatrix{
Y\tim \P^{1}\, \ar@{^{(}->}[r]&M^{0}_{Y/X}}
$
over $\P^{1}$
(see \cite[Ch.\!\! 5]{Ful}
\S\,5.1).
We denote the inclusions $\xymatrix{N_{Y/X}\he\ar@{^{(}->}[r]&
M_{Y/X}^{\circ}}$ and

$\xymatrix{Y\, \ar@{^{(}->}[r]&N_{Y/X}\he}$ by $j_{0}\he$ and $i$,
the projections of ${\bigl(Y\tim \P^{1}\bigr)\tim M_{Y/X}^{\circ}}$
(resp. $Y\tim\P^{1}$, resp.
${\bigl(Y\tim \P^{1}\bigr)\tim N_{Y/X}\he}$, resp. ${Y\tim N_{Y/X}\he}$)
on its first and second factor by $\textrm{pr}_{1}$ and $\textrm{pr}_{2}$
(resp $\widetilde{\textrm{pr}}_{1}$ and $\widetilde{\textrm{pr}}_{2}$,
resp. $\textrm{pr}'_{1}$ and $\textrm{pr}'_{2}$, resp.
$\textrm{pr}''_{1}$ and $\textrm{pr}''_{2}$). Besides,
$\Gamma \suq Y\tim\P^{1} \tim M_{Y/X}^{\circ}$ is the
graph of $F$, and
$\Gamma '\suq Y\tim N_{Y/X}\he$ is the graph of $i$.
Finally, $\Gamma ''=
\bigl(i_{0}\he,\id_{N_{Y/X}\he}\he\bigr)_{*}\he\Gamma '
\suq Y\tim\P^{1}\tim N_{Y/X}\he$,
   where
   $i_{0}\he \mkern -4 mu:\mkern - 8 mu\xymatrix{
Y\tim \{0\}\, \ar@{^{(}->}[r]&Y\tim \P^{1}
  }$ is the injection of the central fiber.
Remark that $\textrm{pr}'_{2}$ and $\textrm{pr}''_{2}$ are proper maps
since $Y$ is compact.
  \par\medskip
We have $\bigl(i_{0}\he,\id_{N_{Y/X}\he}\he\bigr)_{*}\he\{\Gamma '\}_{D}\he
=\{\Gamma ''\}_{D}\he$ and
$\bigl(\id_{Y\tim \P^{1}}\he,j_{0}\he\bigr)\ee_{\be}\{\Gamma \}_{D}\he
=\{\Gamma ''\}_{D}\he$. Let
$\gamma =F_{*}\he\bigl(\ti{\textrm{pr}}_{1}\ee\alpha \bigr)$. Then
\begin{align*}
j_{0}\ee\gamma &=j_{0}\ee\, \textrm{pr}_{2*}
\he\bigl(\textrm{pr}_{1}\ee
\ti{\textrm{pr}}_{1}\ee\alpha \, .\, \{\Gamma \}_{D}\he\bigr)
=\textrm{pr}'_{2*}\Bigl[\bigl(\id_{Y\tim\P^{1}}\he,j_{0}\he\bigr)
\ee\be
\bigl(\textrm{pr}_{1}\ee
\ti{\textrm{pr}}_{1}\ee\alpha \, .\, \{\Gamma \}_{D}\he\bigr)
\Bigr]
&&\textrm{by (ii)}\\
&=\textrm{pr}'_{2*}\Bigl[\bigl(\id_{Y\tim\P^{1}}\he,j_{0}\he
\bigr)\ee\be\, \,
\textrm{pr}_{1}\ee
\ti{\textrm{pr}}_{1}\ee\alpha \, .\, \{\Gamma ''\}_{D}\he\Bigr]\\
&=\textrm{pr}'_{2*}\bigl(i_{0}\he, \id_{N_{Y/X}\he}\he\bigr)_{*}
\he\,
\Bigl(\bigl(i_{0}\he,\id_{N_{Y/X}\he}\he\bigr)\ee\be\, \,
\bigl(\id_{Y\tim\P^{1}}\he,j_{0}\he\bigr)\ee\be\, \,
\textrm{pr}_{1}\ee\ti{\textrm{pr}}_{1}\ee\alpha \, .\,
\{\Gamma '\}_{D\he}\Bigr)
&&\textrm{by the projection formula}\\
&=\textrm{pr}_{2*}''\bigl(\textrm{pr}_{1}''{}^{*}
\alpha \, .\, \{\Gamma '\}_{D}\he
\bigr)=i_{*}\he\alpha .
\end{align*}
By the homotopy principle (Proposition
\ref{PropositionUnComplementQuatreUnArt1} (vi)),
the class $F\ee\be\gamma _{\vert Y\tim
\{t\}}\he$ is independent of $t$. If $t\not = 0$, we have clearly
$F\ee\be\gamma _{\vert Y\tim
\{t\}}\he=i_{Y}\ee i_{Y*}\he\alpha $. For $t=0$,
$F\ee\be\gamma _{\vert Y\tim
\{0\}}\he =i\ee\be j\ee_{0}\gamma =i\ee\be i_{*}\he\alpha $.
Let $\apl{\pi }{N_{Y/X}\he}{Y} $ be the projection of $N_{Y/X}\he$ on $Y$.
Then $\alpha =i\ee\be\pi \ee\be\alpha $. Thus
\[
i\ee\be i_{*}\he\alpha =i\ee\be i_{*}\he(i\ee\be\pi \ee\be\alpha )=
i\ee\be\bigl(\pi \ee\be\alpha \, .\, \overline{\{Y\}}_{D}\he\bigr)
=\alpha \, .\,
i\ee\be\overline{\{Y\}}_{D}\he,
\]
where $\overline{\{Y\}}_{D}\he$
is the cycle class of $Y$ in $N_{Y/X}\he$.
\par
Now
$\overline{\{Y\}}_{D}\he=c_{1}\bigl(\oo_{N_{Y/X}\he}\he(Y)\bigr)$,
so that
$i\ee\be\overline{\{Y\}}_{D}\he
=c_{1}\bigl(N_{Y/N_{Y/X}\he}\he\bigr)=c_{1}\bigl(
N_{Y/X}\he\bigr)$.
\par\medskip
We can now prove (vi).
Its first part is straightforward using d\'{e}vissage as in
Proposition \ref{PropositionUnComplementQuatreUnArt1} (v)
and the analogous result
in Dolbeault cohomology and in integer cohomology.
\par\medskip
If $\alpha $ is a Deligne class on $\ti{X}$, we can write
$\alpha =p\ee x+
\ds\sum _{i=1}^{d-1}j_{*}\he\Bigl[y_{i}\he\, c_{1}\bigl(
\oo_{N_{Y/X}\he}\he(-1)\bigr)^{i-1}\Bigr]$.
Since $E$ is a hypersurface of $\tix$, by the formula proved above
$j\ee\be j_{*}\he\, \lambda  =\lambda  \,
c_{1}\bigl(N_{E/\tix}\he\bigr)=\lambda \, \, c_{1}\bigl(\oo_{N_{Y/X}\he}
\he(-1)\bigr)$
for any Deligne class $\lambda $ on $E$.
We obtain $j\ee\be\alpha =q\ee\be i\ee\be x+\ds
\sum _{i=1}^{d-1}(-1)^{i}y_{i}\, c_{1}\bigl(
\oo_{N_{Y/X}\he}\he(1)\bigr)^{i}$. Since $j\ee\be\alpha =q\ee\be\delta
$, all the classes $y_{i}$ vanish by Proposition
\ref{PropositionUnComplementQuatreUnArt1} (v). Thus
$\alpha =p\ee\be x$. By Proposition \ref{PropositionDeuxInsertC} (iii),
$x=p_{*}\he\alpha $.
\par\medskip
For the excess formula, let $\alpha $ be a Deligne class on $Y$. We define
$
\beta=j_{*}\he\bigl(q\ee\be\alpha \,\, c_{d-1}\he(F\ee\be)\bigr).
$
Then, by Proposition \ref{PropositionDeuxComplementQuatreUnArt1}
(i) and (ii),
\[
j\ee\be\beta =
\bigl[q\ee\be\alpha \,\,  c_{d-1}\he(F\ee\be)\bigr]
\, c_{1}\bigl(N_{E/\tix}\he\bigr)
=q\ee\be\bigl[\alpha \, c_{d}\he\bigl(N_{Y/X}\he\bigr)\bigr].
\]
By the discussion above, $\beta $ comes from the base, so that
\[
\beta =p\ee\be p_{*}\he\beta =
p\ee\be i_{*}\he q_{*}\he\bigl(q\ee\be\alpha \,
c_{d-1}\he(F\ee\be)\bigr)=
p\ee\be i_{*}\he\bigl[\alpha\, \,
q_{*}\he\bigl(c_{d-1}\he(F\ee\be)\bigr)\bigr]=
p\ee\be i_{*}\he\alpha
\]
for
$q_{*}\he\bigl(c_{d-1}\he(F\ee\be)\bigr)=1$ (see \cite[Lemme 19.b]{BoSe}).
\par\medskip
--- Proof of (vii). The formula is already true for $d=1$. We blowup $X$
along $Y$, use the excess formula (vi) and we obtain:
\begin{align*}
q\ee\be i\ee_{Y}i_{Y*}\he\, \alpha &=j\ee\be p\ee\be i_{Y*}\he\, \alpha=
j\ee\be\bigl[j_{*}\he\bigl(q\ee\be\alpha \, c_{d-1}\he(F\ee\be)\bigr)\bigr]\\
&=q\ee\be\alpha \,\,  c_{d-1}\he\bigl(F\ee\be\bigr)\,
c_{1}\bigl(N_{E/\tix}\bigr)=q\ee\be\bigl(\alpha \, c_{d}\he\bigl(N_{Y/X}\he\bigr)
\bigr).
\end{align*}
Since $q\ee\be$ is injective, we get the result.
\end{proof}
\subsection{Chern classes for holomorphic vector bundles}
We refer to \cite[\S\,4]{Zuc} and \cite[\S\,8]{EsVi} for all this section.
From now on, we will suppose that $X$ is smooth of dimension $n$.
Let $E$ be a
holomorphic bundle on $X$ of rank $r$,
and $\P(E)$ be the projective bundle
of $E$ endowed with the
line bundle $\oo_{E}\he(1)$.
Let $\alpha =c_{1}\bigl(\oo_{E}\he (1)\bigr)$.
By property (v) of
Proposition
\ref{PropositionUnComplementQuatreUnArt1},
we can define $\bigl(c_{i}(E)\bigr)_{1\leq
i\leq  r}\he$ by the relation
\[
\alpha ^{r}+p\ee c_{1}(E)\,\alpha ^{r-1}+\cdots +p\ee c_{r}(E)=0
\quad  \textrm{in}\
H_{D}^{2r}\bigl(\P(E),\Z(r)\bigr).
\]
Thus $c_{i}(E)$ is an element of $H^{2i}_{D}(X,\Z(i))$.
The knowledge of the Chern classes $c_{i}(E)$ allows to construct
exponential Chern classes $\ch_{i}(E)$, $0\leq i\leq n$, in
$H^{2i}_{D}(X,\Q(i))$.
These classes are obtained
as the values of certain universal polynomials with rational
coefficients on $c_{1}(E),\dots ,c_{r}\he(E)$ (see \cite{Hir}).
They can also be constructed with the
splitting principle using projective towers
(see \cite{Gr1}). They are completely
characterized by the following facts:
\begin{enumerate}
  \item [--] they satisfy the Whitney additivity formula (Proposition
  \ref{PropositionDeuxComplementQuatreUnArt1} (i));

  \item [--] they satisfy the functoriality formula under pullbacks
  (Proposition \ref{PropositionDeuxComplementQuatreUnArt1} (ii));

  \item [--] if $L$ is a line bundle, $\ch(L)=e^{c_{1}(L)}$.
\end{enumerate}
\begin{definition}\label{DefinitionDeuxComplementQuatreUnArt1}
The total Chern class of $E$ is the element
$c(E)$ of $H_{D}\ee(X)$ defined by
\[
c(E)=1+c_{1}(E)+\cdots +c_{r}\he(E).
\]
The Chern character of $E$ is the element $\ch(E)$ of $H_{D}\ee(X,\Q)$
defined by
\[
\ch(E)=\ch_{0}(E)+\cdots +\ch_{n}\he(E).
\]
\end{definition}
The splitting machinery gives the following proposition:
\begin{proposition}\label{PropositionDeuxComplementQuatreUnArt1}
${\he}$
\begin{enumerate}
  \item [(i)] If
  $\sutrgd{E}{F}{G}$ is an exact sequence of vector bundles on $X$, the
  Whitney formula holds:
  $c(F)=c(E)c(G)$ and $\ch(F)=\ch(E)+\ch(G)$.

  \item [(ii)] If $f$ is a holomorphic map between $X$ and $Y$ and if $E$
  is a holomorphic bundle on $Y$, we have $c\bigl(f\ee
  E\bigr)=f\ee c(E)$ and $\ch\bigl(f\ee
  E\bigr)=f\ee \ch(E)$.

  \item [(iii)] If $E$ and $F$ are two holomorphic vector bundles on
  $X$, then $\ch(E\oti F)=\ch(E)\ch(F)$.

\end{enumerate}
\end{proposition}
\begin{notation}\label{RemarqueUnComplementQuatreUnArt1}
From now on,
  if $\eee$ is a locally free sheaf and $E$ is the associated holomorphic vector
  bundle,
  we will denote by $\ba{\ch}(\eee)$
  the Chern character $\ch(E)$. Thus $\ba{\ch}$ is well
  defined on a basis of any dimension but only for locally free
  sheaves, and we will make use of it in our construction.
%

\end{notation}
\section{Construction of Chern classes}
\label{SectionConstructionOfChernClassesArt1}
The construction of exponential Chern classes $\chp(\ff)$ in
$\hdp(X,\Q(p))$ for an arbitrary coherent sheaf $\ff$ on $X$
will be done by induction on $\dim X$.
If $\dim X=0$, $X$ is a point and everything is obvious.
\par
Let us now precisely state the induction hypotheses
$(\textrm{H}_{n})$:\label{RecurrenceUnChClArt1}
\begin{enumerate}
  \item [$(\textrm{E}_{n})$] If $\dim X\leq n$ and $\ff$ is a coherent analytic
  sheaf on $X$, then the Chern classes $\chp(\ff)$ are defined in
  $\hdp(X,\Q(p))$.

  \item [$(\textrm{W}_{n})$] If $\dim X\leq n$ and
  $\sutrgd{\ff}{\g}{\hh}$ is an exact sequence of
  analytic sheaves on $X$, then
  $\ch(\g)=\ch(\ff)+\ch(\hh)$.
  This means that $\ch$ is defined on
  $\kan(X)$ and is a group morphism.

  \item [$(\textrm{F}_{n})$] If $\dim X\leq n$, $\dim Y\leq n$, and if
  $\apl{f}{X}{Y}$ is a holomorphic map, then
  for all $y$ in $\kan(Y)$, $\ch(f\pe y)=f\ee\ch(y)$.

  \item [$(\textrm{C}_{n})$] If $\dim X\leq n$, the Chern classes are compatible
  with those constructed in Part \ref{SectionDeux}
  on locally free sheaves, i.e.
  $\ch(\ff)=\ba\ch(\ff)$ for all locally free sheaf $\ff$.

  \item [$(\textrm{P}_{n})$]
  If
  $\dim X\leq n$, $\ch$ is a ring morphism: if $x$, $y$ are two
  elements of $\kan(X)$,
  $\ch(x\, .\, y)=\ch(x)\ch(y)$ and $\ch(1)=1$.

  \item [$(\textrm{RR}_{n})$] If $Z$ is a smooth hypersurface of $X$,
  where $\dim X\leq n$, then the (GRR) theorem holds for $i_{Z}\he$:
  for every coherent sheaf $\ff$ on  $Z$, $\ch(i_{Z*}\he\ff)=
  i_{Z*}\he\bigl(\ch^{Z}(\ff)\td(N_{Z/X}\he)^{-1}\bigr)$.

\end{enumerate}
For the definition of analytic \mbox{$K$-theory}
and related operations we refer to
\cite{BoSe}.
\begin{remark}
To avoid any confusion, for a coherent sheaf on $Z\subsetneq X$, we use the
  notation $\ch^{Z}(\ff)$, emphasizing the fact that the class is taken on
  $Z$.
\end{remark}

From now on, we will suppose that all the properties of the induction
hypotheses $(\textrm{H}_{n-1})$
above are true.
\begin{theorem}\label{TheoremTroisTrois}
Assuming hypotheses $(\emph{H}_{n-1})$, we can define a Chern character
for analytic coherent sheaves on compact complex manifolds of dimension
$n$. It further satisfies $(\emph{P}_{n})$, $(\emph{F}_{n})$,
$(\emph{RR}_{n})$, $(\emph{W}_{n})$ and $(\emph{C}_{n})$.
\end{theorem}
Let us briefly explain the organization of the proof of this theorem.
In \S\,\ref{SubsectionTorsionsheaves}, we construct the Chern character
for torsion sheaves. In \S\,\ref{CasGeneral}, we construct the Chern
character for arbitrary coherent sheaves, using the results of
\S\,\ref{CasPositif}. Properties $(\textrm{RR}_{n})$ for a smooth
hypersurface and $(\textrm{C}_{n})$ will be obvious consequences of the
construction.
In \S\,\ref{ProofWhitney}, we prove $(\textrm{W}_{n})$ and
then $(\textrm{F}_{n})$ and $(\textrm{P}_{n})$ using the preliminary
results of \S\,\ref{CasLibre} and \S\,\ref{SectionDeuxWhitneyFormulaArt1}.
Finally, we prove $(\textrm{RR}_{n})$ in \S\,\ref{GRRForAnImmersionClassesArt1}.
\subsection{Construction for torsion sheaves}\label{SubsectionTorsionsheaves}
In this section, we define Chern classes for torsion sheaves by forcing
the Grothendieck-Riemann-Roch formula for immersions of
smooth hypersurfaces.
Let $\kan_{\textrm{tors}}\he(X)$ denote the Grothendieck group of the
abelian category of torsion sheaves on $X$.
We will prove the following
version of Theorem \ref{TheoremTroisTrois}\ \,  for torsion
sheaves:
\begin{proposition}\label{PropositionUnChClArt1}
We can define exponential Chern classes for torsion sheaves on any
\mbox{$n$-dimensional} complex manifold such that:
\begin{enumerate}
  \item [(i) $(\textrm{W}_{n})$]   If $\sutrgd{\ff}{\g}{\hh}$ is
  an exact sequence of
  torsion sheaves on $X$ with $\dim X\leq n$, then
  $\ch(\g)=\ch(\ff)+\ch(\hh)$. This means
  that $\ch$ is a group morphism defined on
  $\kan_{\mathrm{tors}}\he(X)$.

  \item [(ii) $(\textrm{P}_{n})$]  Let $\eee$ be a locally free sheaf
  and $x$ be an
  element of $\kan_{\mathrm{tors}}\he(X)$. Then
  $\ch([\eee].\, x)=\ba\ch\he(\eee)\, . \, \ch(x)$.

  \item [(iii) $(\textrm{F}_{n})$]  Let $\apl{f}{X}{Y}$
  be a holomorphic map where
  $\dim X\leq n$ and $\dim Y\leq n$, and $\ff$ be a coherent sheaf
  on $Y$ such that $\ff$ and $f\ee\ff$ are torsion sheaves.
  Then
  $\ch(f\pe[\ff])=f\ee\ch(\ff)$.

  \item [(iv) $(\textrm{RR}_{n})$]  If $Z$ is a smooth hypersurface of
  $X$ and $\ff$ is coherent on $Z$, then
  \[
\ch\bigl(i_{Z*}\he\ff\bigr)
  =i_{Z*}\he\bigl(\ch^{Z}\be(\ff)\td(N_{Z/X})^{-1}\bigr).
\]
\end{enumerate}
\end{proposition}
We will proceed in three steps. In \S\,\ref{SubSectionUn}, we perform the
construction for coherent sheaves supported in a smooth hypersurface.
In \S\,\ref{SubSectionDeuxPrime}, we deal with sheaves supported in a simple
normal crossing divisor. In \S\,\ref{SubSectionTrois}, we study the
general case.
\subsubsection{}\label{SubSectionUn}
Let $Z$ be a smooth hypersurface of $X$ where $\dim X\leq n$. For $\g$
coherent on $Z$, we define $\ch\bigl(i_{Z*}\he\g\bigr)$ by the GRR
formula
$\ch\bigl(i_{Z*}\he\g\bigr)=i_{Z*}\he\bigl(\ch^{Z}\be
(\g)\td(N_{Z/X})^{-1}\bigr)$,
where $\ch^{Z}\be(\g)$ is defined by $(\textrm{E}_{n-1})$.
\par
If $\sutrgd{\g'}{\g}{\g''}$ is an exact sequence of coherent sheaves on
$Z$, by $(\textrm{W}_{n-1})$, we have
$\ch^{Z}\be(\g)=\ch\be^{Z}(\g')+\ch\be
^{Z}(\g'')$. Thus $\ch\bigl(i_{Z*}\he\g\bigr)=\ch\bigl(i_{Z*}\he\g'\bigr)
+\ch\bigl(i_{Z*}\he\g''\bigr)$.
We obtain
now a well-defined morphism
\[
\xymatrix{
\kan(Z)\ar[r]^{\sim}_{i_{Z*}\he}&\kan_{Z}\he(X)\ar[d]^{\ch_{Z}\he}\\
&H\ee_{D}(X,\Q)
}
\]
Remark that if $\g$ is a coherent sheaf on $X$ which can be written
$i_{Z*}\he\ff$, then the hypersurface $Z$ is not necessarily unique.
If $Z$ is chosen, $\ff$ is of course unique.
This
is the reason why we use the
notation $\ch_{Z}\he(\g)$. We will see in Proposition
\ref{NewPropChClArt1} that $\ch_{Z}\he(\g)$ is in fact independent of $Z$.
\par\medskip
The assertions of the following proposition are
particular cases of $(\textrm{C}_{n})$,
$(\textrm{F}_{n})$, and $(\textrm{P}_{n})$.
\begin{proposition}\label{LemmeDeuxChClArt1}
Let $Z$ be a smooth hypersurface of $X$.
\begin{enumerate}


  \item [(i)] For all $x$ in $\kan_{Z}\he(X)$,
  $\ch^{Z}\be\bigl(i_{Z}\pe x\bigr)
  =i_{Z}\pe\ch_{Z}\he(x)$.

  \item [(ii)] If $\eee$\! is a locally free sheaf on $X$ and
  $x$ is an element of $\kan_{Z}\he(X)$, then
  \[
\ch_{Z}\he(\cro\eee.\,
  x)=\ba\ch\he(\eee)\, .\, \ch_{Z}\he(x).
\]
\end{enumerate}
\end{proposition}
\begin{proof}
(i) We have $x=\ba{x}\pz{Z}[i_{Z*}\he\oo_{Z}\he]$ in
$\kan_{Z}\he(X)$, where $\ba{x}$ is defined in Appendix
\ref{SousSectionDeuxAppendixArt1}. Thus,
\begin{align*}
i_{Z}\ee\chtz\he(x)&=i_{Z}\ee i_{Z*}\he\bigl(\ch^{Z}\be(\ba{x})\td\bigl(
N_{Z/X}\he\bigr)^{-1}\bigr)\\
&=\ch^{Z}\be(\ba{x})\td\bigl(N_{Z/X}\he\bigr)^{-1}c_{1}\bigl(N_{Z/X}\he\bigr)
&&\textrm{by Proposition \ref{PropositionDeuxInsertC} (vii)}\\
&=\ch^{Z}\be(\ba{x})\biggl[1-e^{-\ds c_{1}\bigl(N_{Z/X}\he\bigr)}\biggr]\\
&=\ch^{Z}\be(\ba{x})\ch^{Z}\be\bigl(i_{Z}\pe i_{Z*}\he\oo_{Z}\he\bigr)
&&\textrm{by $(\textrm{C}_{n-1})$}\\
&=\ch^{Z}\be\bigl(\ba{x}\, .\, i_{Z}\pe i_{Z*}\he\oo_{Z}\he\bigr)
&&\textrm{by $(\textrm{P}_{n-1})$}\\
&=\ch^{Z}\be\bigl(i_{Z}\pe x\bigr).
\end{align*}
\par\medskip
(ii) We have
\begin{align*}
\chtz\he\bigl(\, [\eee]\, .\, x\bigr)&=\chtz\he\bigl(i_{Z*}\he
\bigl(i_{Z}\pe[\eee]\, .\, \ba{x}\bigr)\bigr)\\
&=i_{Z*}\he\Bigl(\ch^{Z}\be\bigl(i_{Z}\pe\, [\eee]\, .\, \ba{x}\bigr)
\td\bigl(N_{Z/X}\he\bigr)^{-1}\Bigr)\\
&=i_{Z*}\he\Bigl(i_{Z}\ee\, \ba{\ch}\bigl(  \eee \bigr)
\ch^{Z}\be(\ba{x})\td\bigl(N_{Z/X}\he\bigr)^{-1}\Bigr)
&&\textrm{by Proposition \ref{PropositionDeuxComplementQuatreUnArt1} (ii),
$(\textrm{P}_{n-1})$ and $(\textrm{C}_{n-1})$}\\
&=\ba{\ch}\bigl( \eee \bigr)\, i_{Z*}\he\Bigl(\ch^{Z}
\be(\ba{x})\td\bigl(
N_{Z/X}\bigr)^{-1}\Bigr)&&\textrm{by the projection formula}\\
&=\ba{\ch}\bigl( \eee \bigr)\chtz\he(x).
\end{align*}
\end{proof}
\subsubsection{}\label{SubSectionDeuxPrime} Let $D$ be a divisor in $X$ with simple
normal crossing.
By Proposition \ref{PropositionUnTraduction},
we have an exact sequence:
\[
\sutrd{\bop_{i<j}\kan_{D_{ij}}\he(X)}{\bop_{i}\kan_{D_{i}}\he(X)}
{\kan_{D}\he(X)}\!\!.
\]
Let us consider the morphism $\bop_{i}\ch_{D_{i}}\he$. If $\ff$
belongs to $\kan(D_{ij})$, then
\[
\ch_{D_{i}}\he\bigl(i_{D_{ij}*}\he\ff\bigr)=
i_{D_{i}*}\he\Bigl(\ch^{D_{i}}\be\bigl(i_{\flgdin{D_{ij}}{D_{i}*}}\ff\bigr)
\td\bigl(N_{D_{i}/X}\he\bigr)^{-1}\Bigr)
=
i_{D_{ij}*}\he\Bigl(\ch\be^{D_{ij}}(\ff)\td\bigl(N_{D_{ij}/X}\he
\bigr)^{-1}\Bigr)
\]
because of $(\textrm{RR}_{n-1})$ and the multiplicativity of the Todd
class.
\par
Thus
$\ch_{D_{i}}\he\bigl(i_{D_{ij}*}\he\ff\bigr)=
\ch_{D_{j}}\he\bigl(i_{D_{ij}*}\he\ff\bigr)$, and we get a map
$\apl{\ch_{D}\he}{\kan_{D}\he(X)}{H\ee_{D}(X,\Q)}$
such that the diagram
\[
\xymatrix{
\bop_{i}\kan_{D_{i}}\he(X)\ar[r]\ar[dr]_{\bop_{i}\ch_{D_{i}}\he}
&\kan_{D}\he(X)\ar[r]\ar[d]^{\ch_{D}\he}&0\\
&H\ee_{D}(X,\Q)
}
\]
is commutative.
\begin{proposition}\label{LemmeTroisChClArt1}
The classes $\ch_{D}\he$ have the following properties:
\begin{enumerate}
  \item [(i)] If $\eee$ is a locally free sheaf on  $X$ and $x$ is an
  element of $\kan_{D}\he(X)$, then
  \[
\ch_{D}\he(\cro\eee\, .\, x)=
  \ba\ch(\eee).\, \ch
  _{D}\he(x).
\]
  \item [(ii)] \label{petitdeuxi}
  Let $\ti D$ be an effective simple normal crossing
  divisor in $X$  such that $\ti{D}^{\textrm{red}}\be=D$.
  Then
  \[
\ch_{D}\he\bigl(\oo_{\tid }\bigr)=1-
  \ba\ch\bigl(\oo_{X}(-\tid)\he\bigr).
  \]
  \item [(iii)] \label{petittroisi}\emph{(First lemma of functoriality)} Let
  $\apl{f}{X}{Y}$ be a surjective map. Let $D$ be a reduced divisor in $Y$
  with simple normal crossing such that $f^{-1}(D)$ is also a divisor
  with simple normal crossing in $X$. Then for all $y$ in
  $\kan_{D}\he(Y)$,
  $\ch_{f^{-1}(D)}\he\bigl(f\pe y\bigr)=f\ee\ch_{D}\he(y)$.

  \item[(iv)] \emph{(Second lemma of functoriality)} Let $Y$ be a
  smooth submanifold of $X$ and $D$ be a reduced divisor in $X$
  with simple normal crossing. Then, for every $x$ in $\kan_{D}\he(X)$,
  $\ch^{Y}\be\bigl(i\pe_{Y} x\bigr)=i\ee_{Y}\ch_{D}\he(x)$.

\end{enumerate}
\end{proposition}
\begin{proof}
We start with two technical lemmas which will be crucial for the proof
of (ii) and (iii).
\begin{lemma}\label{LemUnInsertA}
Let $D=m_{1}D_{1}+\cdots +m_{N}\he D_{N}\he$ be an effective
simple normal crossing
divisor in $X$, and $\mu $ be the element of $H_{D}\ee(X,\Q)$
defined by
\[
\mu =\sum _{k\geq 1}\dfrac {(-1)^{k-1}}{k!}\Bigl(m_{1}\{D_{1}\}_{D}\he
+\cdots +m_{N}\he\{D_{N}\he\}_{D}\he\Bigr)^{k-1}.
\]
Then there exist $u_{i}$ in $G_{D_{i}}\he(X)$, $1\leq i\leq N$, and
$\zeta _{ij}$ in $H_{D}\ee(D_{ij})$, $1\leq i,j\leq N$, $i\not = j$, such
that
\begin{enumerate}
  \item [(a)] $u_{1}+\cdots +u_{N}\he=\oo_{D}\he$ in $G_{D^{\textrm{red}}}(X)$.

  \item [(b)] $\zeta_{ij}=-\zeta_{ji}$, $1\leq i,j\leq N$, $i\not = j$.

  \item [(c)] $\ch(\ba{u}_{i})\td\bigl(N_{D_{i}/X}\he\bigr)^{-1}
  -m_{i}\, i_{D_{i}}\ee\, \mu =
  \ds\sum _{\genfrac{}{}{0pt}{1}{j=1}{j\not =i}}^{N}i_{\flcourtegd{\scriptstyle D_{ij}}{
  \scriptstyle D_{i}*}}
  \zeta _{ij},\quad  1\leq i\leq
  N$.
\end{enumerate}
\end{lemma}
\begin{proof}
We proceed by induction on the number $N$ of branches of $D^{\textrm{red}}\be$.
\par
If $N=1$, we must prove that
$\ch(\ba{u}_{1})\td\bigl(N_{D_{1}/X}\he\bigr)^{-1}
  =m_{1}\, i_{D_{1}}\ee\, \mu $, where $u_{1}=\oo_{m_{1}D_{1}}\he$. In
$G_{D_{1}}\he(X)$ we have
$\oo_{m_{1}D_{1}}\he=\ds\sum _{q=0}^{m_{1}-1}i_{D_{1}*}\he\bigl(
N_{D_{1}/X}^{*\, \oti q}\bigr)$, thus\quad
$\ds\ba{u}_{1}=\sum _{q=0}^{m_{1}-1}N_{D_{1}/X}^{*\, \oti q}$. Therefore
\begin{align*}
\ch(\ba{u}_{1})\td\bigl(N_{D_{1}/X}\he\bigr)^{-1}&=
\left(\sum _{q=0}^{m_{1}-1}e^{\ds-q\, c_{1}\bigl(N_{D_{1}/X}\he\bigr)}\right)
\, \dfrac {1-e^{\ds-c_{1}\bigl(N_{D_{1}/X}\he\bigr)}}
{c_{1}\bigl(N_{D_{1}/X}\he\bigr)} \\
&=\dfrac {1-e^{-\ds m_{1}c_{1}\bigl(N_{D_{1}/X}\he\bigr)}}
{c_{1}\bigl(N_{D_{1}/X}\he\bigr)}=m_{1}\, i_{D_{1}}\ee\, \mu .
\end{align*}
\par\medskip
Suppose that the lemma holds for divisors $D'$ such that $D_{\textrm{red}}'$
has $N-1$ branches. Let $D=m_{1}D_{1}+\cdots +m_{N}\he D_{N}\he$ and
$D'=m_{1}D_{1}+\cdots +m_{N-1}\he D_{N-1}\he$. By induction, there
exist $u_{i}'$ in $G_{D_{i}}\he(X)$, $1\leq i\leq N-1$, and
$\zeta '_{ij}$ in $H_{D}\ee(D_{ij})$, $1\leq i,j\leq N-1$, $i\not = j$,
satisfying properties (a), (b), and (c) of Lemma
\ref{LemUnInsertA}. For $0\leq k\leq m_{N}\he$, we introduce the divisors
$Z_{k}=m_{1}D_{1}+\cdots +m_{N-1}\he D_{N-1}\he+kD_{N}\he$. We have
exact sequences
\[
\sutrgdpt{i_{D_{N}\he}\ee\oo_{X}\he(-Z_{k})}{\oo_{Z_{k+1}}\he}
{\oo_{Z_{k}}\he}{.}
\]
Thus, in $G_{D^{\textrm{red}}}\he(X)$, we have
\[
\oo_{D}\he=\oo_{D'}\he+i_{D_{N}\he *}\he
\left[\sum _{q=0}^{m_{N-1}\he}i_{D_{N}\he}\ee\oo_{X}\he(-Z_{q})\right]
=\oo_{D'}\he+i_{D_{N}\he *}\he i_{D_{N}\he}\ee\left[
\oo_{X}\he(-D')\sum _{q=0}^{m_{N-1}\he}\oo_{X}\he\bigl(-qD_{N}\he\bigr)
\right].
\]
We choose
$
\begin{cases}
u_{N}\he=i_{D_{N}\he *}\he i_{D_{N}\he}\ee\Bigl[
\oo_{X}\he(-D')\ds\sum _{q=0}^{m_{N-1}\he}\oo_{X}\he\bigl(-qD_{N}\he\bigr)
\Bigr]\\
u_{i}=u'_{i}\qquad  \textrm{for}\ 1\leq i\leq N-1.
\end{cases}
$
\par\bigskip
Let $i$ be such that $1\leq i\leq N-1$. Then
\begin{align*}
\ch(\ba{u}_{i})&\td\bigl(N_{D_{i}/X}\he\bigr)
-m_{i}\, i_{D_{i}}\ee\, \mu =
\ch(\ba{u}_{i}')\td\bigl(N_{D_{i}/X}\he\bigr)
-m_{i}\, i_{D_{i}}\ee\, \mu'+m_{i}\, i_{D_{i}}\ee\, (\mu '-\mu )\\
&=\sum _{\genfrac{}{}{0pt}{1}{j=1}{j\not =i}}^{N-1}i_{\flcourtegd{\scriptstyle D_{ij}}{
\scriptstyle D_{i}*}}\zeta '_{ij}
+m_{i}\,\,  i_{D_{i}}\ee
\Biggl[\, \sum _{k=1}^{\infty }
\dfrac {(-1)^{k}}{k!}
\sum _{j=1}^{k-1}\dbinom{k-1}{j}\bigl(m_{1}\{D_{1}\}_{D}\he+\cdots\\
&\hspace*{34.834ex}\cdots
+
m_{N-1}\he\{D_{N-1}\he\}_{D}\he\bigr)^{k-1-j}\bigl(m_{N}\he\{D_{N}\he\}
_{D}\he\bigr)^{j}
\Biggr]\qquad \textrm{by induction}\\
&=\sum _{\genfrac{}{}{0pt}{1}{j=1}{j\not =i}}^{N-1}i_{\flcourtegd{\scriptstyle D_{ij}}{
\scriptstyle D_{i}*}}\zeta '_{ij}
+m_{i}\,\, i_{\flcourtegd{\scriptstyle D_{iN}\he}{\scriptstyle D_{i}\he *}}
i_{D_{iN}}\ee
\Biggl[\, \sum _{k=1}^{\infty }
\dfrac {(-1)^{k}}{k!}
\sum _{j=1}^{k-1}\dbinom{k-1}{j}\bigl(m_{1}\{D_{1}\}_{D}\he+\cdots\\
&\hspace*{34.834ex}\cdots
+
m_{N-1}\he\{D_{N-1}\he\}_{D}\he\bigr)^{k-1-j}m_{N}^{j}\{D_{N}\he\}_{D}^{j-1}
\Biggr].
\end{align*}
For the last equality, we have used that
\[
i_{D_{i}}\ee\bigl(\alpha\,
\{D_{N}\he\}_{D}\he\bigr)=i_{D_{i}}\ee\alpha \, \{D_{iN}\he\}_{D}\he
=i_{\flcourtegd{\scriptstyle D_{iN}\he}{\scriptstyle D_{i}\he *}}
\bigl(i_{D_{iN}\he}\ee\alpha \bigr)
\]
where $\{D_{iN}\he\}_{D}\he$ is the cycle class of $D_{iN}\he$ in
$D_{i}$.
\par\medskip
Let us define
\[
\begin{cases}
\zeta _{ij}=\zeta ' _{ij}
&\textrm{if}\ 1\leq i,j\leq N-1,\ i\not = j\\
\zeta _{iN}=m_{i}\, \,  i_{D_{iN}}\ee
\Biggl[\,\ds \sum _{k=1}^{\infty }
\dfrac {(-1)^{k}}{k!}
\sum _{j=1}^{k-1}\dbinom{k-1}{j}\bigl(m_{1}\{D_{1}\}_{D}\he+\cdots\\
\hspace*{25ex}\cdots +
m_{N-1}\he\{D_{N-1}\he\}_{D}\he\bigr)^{k-1-j}m_{N}^{j}\{D_{N}\he\}_{D}^{j-1}
\Biggr]&\textrm{if}\ 1\leq i\leq N-1\\
\zeta _{Nj}\he=-\zeta _{jN}\he&\textrm{if}\ 1\leq j\leq N-1.
\end{cases}
\]
Properties (a) and
(b) of Lemma \ref{LemUnInsertA} hold, and property (c)
of the same lemma hold for $1\leq i\leq N-1$. For $i=N$, let us now
compute both members of (c). We have
\begin{align*}
\sum _{l=1}^{N-1}i_{\flcourtegd{\st D_{Nl}\he}{\st D_{N}\he *}}\zeta
_{Nl}\he&=\sum _{l=1}^{N-1}m_{l}\, \, i_{D_{N}\he}\ee\Biggl[\sum _{k=1}^{\infty }
\dfrac {(-1)^{k-1}}{k!}\sum
_{j=1}^{k-1}\dbinom{k-1}{j}\bigl(m_{1}\{D_{1}\}_{D}\he+\cdots \\
&\hspace*{34.834ex}\cdots +m_{N-1}\he\{D_{N-1}\he\}\bigr)
^{k-1-j}m_{N}^{j}\{D_{N}\}_{D}^{j-1}
\{D_{l}\he\}\Biggr]\\
&\hspace*{-6.968ex}=i_{D_{N}\he}\ee\Biggl[\sum _{k=1}^{\infty }
\dfrac {(-1)^{k-1}}{k!}\sum
_{j=1}^{k-1}\dbinom{k-1}{j}\bigl(m_{1}\{D_{1}\}_{D}\he+\cdots
+m_{N-1}\he\{D_{N-1}\he\}_{D}\he\bigr)^{k-j}m_{N}^{j}\{D_{N}\}_{D}^{j-1}
\Biggr].\tag{$*$}\label{alpha}
\end{align*}
In the first equality, we have used
\[
i_{\flcourtegd{\st D_{Nl}\he}{\st D_{N}\he *}}
\, i_{D_{lN}\he}\ee\alpha =
i_{D_{N}\he}\ee\alpha \, \{D_{lN}\he\}_{D}\he=
i_{D_{N}\he}\ee\, \bigl(\alpha \, \{D_{l}\}_{D}\he\bigr)
\]
where $\{D_{lN}\he\}_{D}\he$ is the cycle class of $D_{lN}\he$ in
$D_{N}\he$.
\par\medskip
Now,
\begin{align*}
\ch(\ba{u}_{N}\he)&\td\bigl(N_{D_{N}\he/X}\he\bigr)^{-1}-m_{N}\he\,
i_{D_{N}\he}\ee\, \mu\\
&=i_{D_{N}\he}\ee\left[
e^{\ds-m_{1}\{D_{1}\}-\dots -m_{N-1}\he\{D_{N-1}\he\}}
\left(\sum _{q=0}^{m_{N-1}\he}e^{\ds-q\{D_{N}\he\}}\right)
\dfrac {1-e^{-\ds\{D_{N}\he\}}}{\{D_{N}\he\}}
\right]-m_{N}\he\, i_{D_{N}\he}\ee\, \mu
\intertext{by (C$_{n-1}$) and Proposition
\ref{PropositionDeuxComplementQuatreUnArt1} (ii) and (iii)}\\[-4ex]
&=i_{D_{N}\he}\ee\left[
e^{\ds-m_{1}\{D_{1}\}-\dots -m_{N-1}\he\{D_{N-1}\he\}}
\, \, \dfrac {1-e^{-\ds m_{N}\he\{D_{N}\he\}}}{\{D_{N}\he\}}-m_{N}\he\, \mu
\right]\\
&=i_{D_{N}\he}\ee\left[
m_{N}\he\sum _{r=0}^{\infty }\, \, \sum _{q=1}^{\infty }
\dfrac {(-1)^{r+q-1}}{r!\, q!}\bigl(
m_{1}\{D_{1}\}+\cdots +m_{N-1}\he\{D_{N-1}\he\}
\bigr)^{r}\, \bigl(m_{N}\he\{D_{N}\he\}\bigr)^{q-1}
\right.\\
&\hspace*{11.614ex}-m_{N}\he
\left.
\sum _{k=1}^{\infty }\dfrac {(-1)^{k-1}}{k!}\sum _{j=0}^{k-1}
\dbinom{k-1}{j}
\bigl(
m_{1}\{D_{1}\}+\cdots +m_{N-1}\he\{D_{N-1}\he\}
\bigr)^{k-1-j}\, \bigl(m_{N}\he\{D_{N}\he\}\bigr)^{j}
\right]\!.
\end{align*}
In the first term, we put $k=q+r$, $p=q-1$ and we obtain
\begin{align}
m_{N}\he \, i_{D_{N}\he}\ee
\Biggl[
\sum _{k=1}^{\infty }\, \,
\sum _{p=0}^{k-1}
\dfrac {(-1)^{k-1}}{k!}
\Biggl(\dbinom{k}{p+1}-\dbinom{k-1}{p}\Biggr)
\bigl(
&m_{1}\{D_{1}\}+\cdots\notag\\
\cdots  +&m_{N-1}\he\{D_{N-1}\he\}
\bigr)^{k-1-p}\, \bigl(m_{N}\he\{D_{N}\he\}\bigr)^{p}
\Biggr].\tag{$**$} \label{beta}
\end{align}
Now $\dbinom{k}{p+1}-\dbinom{k-1}{p}$ is equal to $\dbinom{k-1}{p+1}$
for $p\leq k-2$ and to zero for $p=k-1$.
It suffices to put $j=p+1$ in the sum to obtain the
equality of (\ref{alpha}) and (\ref{beta}).
\end{proof}
\begin{lemma}\label{LemDeuxInsertA}
Using the same notations as in Lemma \ref{LemUnInsertA}, let $\alpha _{i}
$ in $H_{D}\ee(D_{i})$, $1\leq i\leq N$, be such that
$i_{\flcourtegd{\st D_{ij}}{\st D_{i}}}\ee\alpha _{i}=
i_{\flcourtegd{\st D_{ij}}{\st D_{j}}}\ee\alpha _{j}$. Then there exist $u_{i}$
in $G_{D_{i}}\he(X)$, satisfying $u_{1}+\cdots +u_{N}\he=\oo_{D}\he$ in
$G_{D^{\textrm{red}}}\he(X)$,
such that
\[
\sum _{i=1}^{N}i_{D_{i}*}\he\Bigl(
\alpha _{i}\, \ch(\ba{u}_{i})\td\bigl(N_{D_{i}/X}\he\bigr)^{-1}
\Bigr)=\left(
\sum _{i=1}^{N}m_{i}\, i_{D_{i}*}\he(\alpha _{i})
\right)\mu .
\]
\end{lemma}
\begin{proof}
We pick
$u_{1},\dots ,u_{N}\he$ given by Lemma \ref{LemUnInsertA}. Then
\begin{align*}
\sum _{i=1}^{N}i_{D_{i}*}\he\Bigl(
\alpha _{i}\, \ch(\ba{u}_{i})\, &\td\bigl(N_{D_{i}/X}\he\bigr)^{-1}
\Bigr)-\left(
\sum _{i=1}^{N}m_{i}\, i_{D_{i}*}\he(\alpha _{i})
\right)\mu\\
&=\sum _{i=1}^{N}i_{D_{i}*}\he\Bigl[
\alpha _{i}\Bigl(\ch(\ba{u}_{i})\td\bigl(N_{D_{i}/X}\he\bigr)^{-1}
-m_{i}\, i_{D_{i}}\ee
\, \mu\Bigl)\Bigr]&&\textrm{by the projection formula}\\
&=\sum _{i=1}^{N}\, \, \sum _{\genfrac{}{}{0pt}{1}{j=1}{j\not =i}}^{N}i_{D_{i}*}
\Bigl[\alpha _{i}\,\,  i_{\flcourtegd{\st D_{ij}}{\st D_{i}*}}\he\zeta _{ij}\Bigr]\\
&=\sum _{i=1}^{N}\, \, \sum _{\genfrac{}{}{0pt}{1}{j=1}{j\not =i}}^{N}
i_{D_{ij}*}\he\Bigl(
i_{\flcourtegd{\st D_{ij}}{\st D_{i}}}\ee\alpha _{i}\,\,  \zeta _{ij}
\Bigr)&&\textrm{by the projection formula.}
\end{align*}
Grouping the terms $(i,j)$ and $(j,i)$, we get $0$, since
$\zeta _{ij}=-\zeta _{ji}$.
\end{proof}
We now prove Proposition \ref{LemmeTroisChClArt1}.
\begin{proof}
(i) We write $x=x_{1}+\cdots +x_{N}\he$
in $\kan_{D^{\textrm{red}}}\he(X)$, where $x_{i}$ is an element of
$\kan_{D_{i}}\he(X)$. Then
\begin{align*}
\ch_{D}\he(\cro\eee\, .\, x)&=\sum _{i=1}^{N}\ch_{D_{i}}\he(\cro\eee\, .\, x_{i})
=\sum _{i=1}^{N}\ba\ch\he(\eee)\, .\ch_{D_{i}}\he(x_{i})&&\textrm{by
Proposition
\ref{LemmeDeuxChClArt1} (ii)}\\
&=\ba\ch(\eee)\, .\ch_{D}\he(x)&&\textrm{by the very definition of
$\ch_{D}\he(x)$.}
\end{align*}
\par\medskip
(ii) We choose $u_{1},\dots ,u_{N}\he$ such that Lemma
\ref{LemUnInsertA} holds. Then
\begin{align*}
\ch\bigl(\oo_{\ti{D}}\he\bigr)&=\sum _{i=1}^{N}\ch(u_{i})
=\sum _{i=1}^{N}i_{D_{i}*}\he
\Bigl(\ch(\ba{u}_{i})\td\bigl(N_{\ti{D}_{i}/X}\he\bigr)^{-1}\Bigr)\\
&=\left(\sum _{i=1}^{N}m_{i}\{\ti{D}_{i}\}_{D}\he\right)\mu
=1-e^{-\ds\bigl(m_{1}\{\ti{D}_{1}\}_{D}\he+\cdots +m_{N}\he\{\ti{D}_{N}\he\}
_{D}\he\bigr)}
=1-\ba{\ch}\bigl(\oo_{X}\he(-\ti{D})\bigr).
\end{align*}
\par\medskip
(iii) By d\'{e}vissage we can suppose that $D$ is a smooth hypersurface of
$Y$. Let $\ba{f}_{i}\he $ be defined by the diagram
\[
\xymatrix@C=25pt@R=20pt{
\ti{D_{i}}\ar[r]\ar[d]_{\ba{f}_{i}\he }&X\ar[d]^{f}\\
D\ar[r]&Y
}
\]
and let $y$ be an element of $G(D)$.
We put $\alpha _{i}=\ba{f}_{i}\ee\ch^{D}\be(y)$. By
the functoriality property $(\textrm{F}_{n-1})$
we have
$i_{\flcourtegd{\st \ti{D}_{ij}}{\st \ti{D}_{i}}}\ee\alpha _{i}=
i_{\flcourtegd{\st \ti{D}_{ij}}{\st \ti{D}_{i}}}\ee\alpha _{j}$. We
choose again $u_{1},\dots ,u_{N}\he$ such that Lemma \ref{LemDeuxInsertA}
holds. By Proposition \ref{TheoremeDeuxAppendixArt1} of Appendix
\ref{AppendixArt1}
we can write
$
f\pe i_{D*}\he y=\ds\sum
_{i=1}^{N}\bigl(\ba{f}_{i}\pe y\bigr)\pz{\ti{D}_{i}}\he u_{i}
$. Thus
\begin{align*}
\ch_{\tid}\he\bigl(f\pe i_{D*}\he y\bigr)&=
\sum _{i=1}^{N}i_{\ti{D}_{i}*}\he\Bigl(
\ch^{\ti{D}_{i}}_{\vphantom{i}}\bigl(\ba{f}_{i}\pe y\bigr)\ch^{\ti{D}_{i}}\be
(\ba{u}_{i})
\td\bigl(N_{\ti{D}_{i}/X}\he\bigr)^{-1}
\Bigr)&&\textrm{by $(\textrm{P}_{n-1})$}\\
&=\sum _{i=1}^{N}i_{\ti{D}_{i}*}\he\Bigl(\alpha _{i}\ch^{\ti{D}_{i}}\be(\ba{u}_{i})
\td\bigl(N_{\ti{D}_{i}/X}\he\bigr)^{-1}\Bigr)&&\textrm{by}\ (\textrm{F}_{n-1})\\
&=\left(\sum _{i=1}^{N}m_{i}\, \, i_{\ti{D}_{i}*}\he(\alpha _{i})\right)\mu
&&\textrm{by Lemma \ref{LemDeuxInsertA}}\\
&=\left[\sum _{i=1}^{N}m_{i}\, \, i_{\ti{D}_{i}*}\he\bigl(\ba{f}_{i}\ee
\ch^{D}_{\vphantom{i}}(y)\bigr)\right]f\ee\left(
\dfrac {1-e^{-\ds\{D\}_{D}\he}}{\{D\}_{D}\he}\right) \\
&=f\ee\left[i_{D*}\he\bigl(\ch^{D}_{\vphantom{i}}(y)\bigr)\cdot
\dfrac {1-e^{-\ds\{D\}_{D}\he}}{\{D\}_{D}\he}\right]
&&\textrm{by Proposition \ref{PropositionDeuxInsertC} (v)}\\
&=f\ee i_{D*}\he\Bigl(\ch^{D}_{\vphantom{i}}(y)
\td\bigl(N_{D/Y}\he\bigr)^{-1}\Bigr)&&\textrm{by the projection formula}\\
&=f\ee\, \ch_{D}\he\bigl(i_{D*}\he y\bigr).
\end{align*}
\end{proof}
(iv)
We will first prove it under the assumption that, for all $i$,
either $Y$ and
$D_{i}\he$
intersect transversally, or $Y=D_{i}\he$. By
d\'{e}vissage, we can suppose that $D$ has only one branch and that $Y$
and $D$ intersect transversally, or $Y=D$. We deal with both cases
separately.
\par\medskip
-- If $Y$ and $D$ intersect transversally, then
$i_{Y}\pe\bigl[i_{D*}\he\oo_{D}\he\bigr]=
\bigl[i_{\flcourtegd{\st Y\cap D}{\st Y*}}\he\oo_{Y\cap
D}\he\bigr]$. Thus, by Proposition \ref{TheoremeUnAppendixArt1} of
Appendix \ref{AppendixArt1},
\[
i\pe_{Y}x=i\pe_{Y}\bigl(\ba{x}\,\,  \pz{D}\,\,  \he
\bigl[i_{D*}\he\oo_{D}\he\bigr]\bigr)=
i_{\flcourtegd{\st Y\cap D}{\st D}}\pe\ba{x}\,\,  \pz{Y\cap D}\he\, \,
\bigl[i_{\flcourtegd{\st Y\cap D}{\st Y*}}\he\oo_{Y\cap
D}\he\bigr]=i_{\flcourtegd{\st Y\cap D}{\st Y*}}\he
\bigl(i_{\flcourtegd{\st Y\cap D}{\st D}}\pe\ba{x}\bigr)
\]
and we obtain
\begin{align*}
\ch^{Y}\be\bigl(i_{Y}\pe x\bigr)&=i_{\flcourtegd{\st Y\cap D}{\st Y*}}\he
\Bigl(\ch^{Y\cap D}\be\bigl(i_{\flcourtegd{\st Y\cap D}{\st D}}\pe\ba{x}\bigr)
\td\bigl(N_{Y\cap D/Y}\he\bigr)^{-1}\Bigr)&&\textrm{by $(\textrm{RR}_{n-1})$}\\
&=i_{\flcourtegd{\st Y\cap D}{\st Y*}}\he
\Bigl(i_{\flcourtegd{\st Y\cap D}{\st D}}\ee
\ch^{D}\be(\ba{x})\, \, i_{\flcourtegd{\st Y\cap D}{\st D}}\ee
\td\bigl(N_{D/X}\he\bigr)^{-1}\Bigr)&&\textrm{by $(\textrm{F}_{n-1})$}\\
&=i_{Y}\ee \, i_{D*}\he
\Bigl(\ch^{D}\be(\ba{x})\td\bigl(N_{D/X}\he\bigr)^{-1}\Bigr)&&
\textrm{by Proposition \ref{PropositionDeuxInsertC} (iv)}\\
&=i_{Y}\ee\ch_{D}\he(x).
\end{align*}
-- If $Y=D$, $i_{Y}\pe\bigl[i_{D*}\he\oo_{D}\he\bigr]=\bigl[\oo_{Y}\he\bigr]
-\bigl[N_{Y/X}\ee\bigr]$. Thus
$i_{Y}\pe\, x=\ba{x}-\ba{x}\, .\, \bigl[N_{Y/X}\ee\bigr]$ and
\begin{align*}
\ch^{Y}\be\bigl(i_{Y}\pe x\bigr)&=\ch^{Y}\be(\ba{x})-
\ch^{Y}\be(\ba{x})\ba{\ch}
\bigl(N_{Y/X}\ee\bigr)&&\textrm{by $(\textrm{P}_{n-1})$ and
$(\textrm{C}_{n-1})$}\\
&=\ch\be^{Y}(\ba{x})\biggl(1-e^{-\ds c_{1}\bigl(N_{Y/X}\he\bigr)}\biggr)\\
&=\ch\be^{Y}(\ba{x})\, \td\bigl(N_{Y/X}\he\bigr)^{-1}\, c_{1}\bigl(
N_{Y/X}\he\bigr)\\
&=i_{Y}\ee i_{Y*}\he\Bigl(\ch^{Y}\be(\ba{x})\td\bigl(N_{Y/X}\he\bigr)^{-1}\Bigr)
&&\textrm{by Proposition \ref{PropositionDeuxInsertC} (vii)}\\
&=i_{Y}\ee \ch_{Y}\he(x).
\end{align*}
We examine now the general case.
By Hironaka's desingularization theorem, we can desingularize $Y\cup D$
by a succession $\tau $ of $k$ blowups with smooth centers such that
$\tau ^{-1}(Y\cup D)$ is a divisor with simple
normal crossing. By first blowing up $X$ along $Y$, we can suppose
that $\tau ^{-1}(Y)=\breve{D}$ is a subdivisor of $\tid=\tau ^{-1}(
Y\cup D)$. We have the following diagram:
\[
\xymatrix{\breve{D}_{j}\ar[r]^{i_{\breve{D}_{j}}}\ar[d]_{q_{j}}&
\tix\ar[d]^{\tau }\\
Y\ar[r]_{i_{Y}\he}&X
}
\]
Then
\begin{align*}
q\ee_{j}\ch^{Y}\be\bigl(i\pe_{Y}x\bigr)&=\ch\be
^{\breve{D}_{j}}\bigl(q\pe_{j}i\pe_{Y}x\bigr)&&
\textrm{by $(\textrm{F}_{n-1})$}\\
&=\ch\be^{\breve{D}_{j}}\bigl(i\pe_{\breve{D}_{j}}\tau \pe x\bigr)\\
&=i\ee_{\breve{D}_{j}}\ch_{\ti{D}}\he\bigl(\tau \pe x\bigr)
&&\textrm{since $\breve{D}_{j}$ and $\ti{D}_{i}$ intersect transversally,
or $\breve{D}_{j}=\ti{D}_{i}$}\\
&=i\ee_{\breve{D}_{j}}\tau \ee\ch_{D}\he(x)&&
\textrm{by the first lemma of functoriality \ref{LemmeTroisChClArt1} (iii)}\\
&=q\ee_{j}i\ee_{Y}\ch_{D}\he(x).
\end{align*}
We can now write $q_{j}$ as $\delta \circ\mu_{j} $, where
$E$ is the exceptional divisor of the blowup of $X$ along $Y$,
$\apl{\delta }{E}{Y}$ is the canonical projection and
$\apl{\mu _{j}}{\breve{D}_{j}}{E}$ is the restriction of the last $k-1$
blowups to $\breve{D}_{j}$.
Write $\tau =\tau _{k}\he\circ\tau _{k-1}\he\circ\dots \circ\tau _{1}\he$
where $\tau _{i}$ are the blowups. Let us define a sequence of divisors
$\bigl(E_{i}\he\bigr)_{0\leq i\leq k}$ by induction: $E_{0}\he=E$, and
$E_{i+1}\he$ is the strict transform of $E_{i}\he$ under $\tau _{i+1}\he$.
Since the $E_{i}\he$ are smooth divisors, all the maps
$\apl{\tau _{i+1}\he}{E_{i+1}\he}{E_{i}\he}$ are isomorphisms. There
exists $j$ such that $E_{k}\he=\breve{D}_{j}\he$. We deduce that $\mu _{j}=
\apl{\tau _{\vert \breve{D}_{j}}\he}{\breve{D}_{j}}{D}$ is an
isomorphism.
Since
$\delta$ is the projection
of the projective bundle $\flgd{\P\bigl(N_{Y/X}\he\bigr)}{Y}$\!\!,
$\delta \ee$
is injective.
Thus $q\ee_{j}=\mu \ee_{j}\delta \ee$ is injective
and we get
$\ch\be^{Y}\bigl(i\pe_{Y}x\bigr)=i\ee_{Y}\ch_{D}\he(x)$.
\end{proof}
Now, we can clear up the problem of the dependence with respect to $D$
of $\ch_{D}\he(\ff)$.
\begin{proposition}\label{NewPropChClArt1}
If $D_{1}$ and $D_{2}$ are two divisors of $X$ with simple normal crossing
such that $\supp\ff\suq D_{1}$ and $\supp\ff\suq D_{2}$, then
$\ch_{D_{1}}\he
(\ff)=\ch_{D_{2}}\he(\ff)$.
\end{proposition}
\begin{proof}
This property is clear if $D_{1}\suq D_{2}$. We will reduce the
general situation to this case. By Hironaka's theorem, there exists
$\apl{\tau }{\tix}{X}$ such that $\tau ^{-1}(D_{1}\cup D_{2})$ is a
divisor with simple normal crossing. Let $\tid_{1}=\tau ^{-1}D_{1}$ and
$\tid _{2}=\tau ^{-1}D_{2}$. By the first functoriality lemma
\ref{LemmeTroisChClArt1} (iii), since $\tid_{1}\suq\tid$, we have
$\tau\ee\ch_{D_{1}}\he(\ff)=\ch_{\tid_{1}}\he\bigl(\tau \pe[\ff]\bigr)=
\ch_{\tid}\he\bigl(\tau \pe[\ff]\bigr)
$. The same property holds for $D_{2}$.
The map $\tau $
is a succession of blowups, thus $\tau \ee$ is injective and
we get $\ch_{D_{1}}\he(\ff)=\ch_{D_{2}}\he(\ff)$.
\end{proof}
\begin{definition}\label{NewDefinitionChClArt1}
If $\supp(\ff)\suq D$ where $D$ is a normal simple crossing divisor,
$\ch(\ff)$ is defined as $\ch_{D}\he(\ff)$.
\end{definition}
By Proposition
\ref{NewPropChClArt1}, this definition makes sense.
\subsubsection{}\label{SubSectionTrois} We can now define $\ch(\ff)$ for an
arbitrary torsion sheaf.
\par\medskip
Let $\ff$ be a torsion sheaf. We say that
a succession of blowups
with smooth centers $\apl{\tau }{\tix}{X}$ is a
desingularization of $\ff$ if there exists a divisor with simple normal
crossing $D$ in $\tix$ such that $\tau ^{-1}\bigl(\supp(\ff)\bigr)\suq
D$. By Hironaka's theorem applied to $\supp(\ff)$, there always exists
such a $\tau $. We say that $\ff$
can be desingularized in $d$ steps if
there exists a desingularization $\tau $ of $\ff$ consisting of at most $d$
blowups. In that case, $\ch(\tau \pe[\ff])$ is defined by Definition
\ref{NewDefinitionChClArt1}.
\begin{proposition}\label{LemmeQuatreChClArt1}
There exists a class $\ch(\ff)$ uniquely determined by $\ff$ such that
\begin{enumerate}
  \item [(i)] If $\tau $ is a desingularization of $\ff$, then
  $\tau \ee\ch(\ff)=\ch(\tau \pe[\ff])$.

  \item [(ii)] If $Y$ is a smooth submanifold of $X$, then
  $\ch\be^{Y}\bigl(i\pe_{Y}\cro\ff\bigr)
  =i\ee_{Y}\ch(\ff)$.
\end{enumerate}
\end{proposition}
\begin{proof}
Let $d$ be the number of blowups necessary to desingularize $\ff$. Both
assertions will be proved at the same time by induction on $d$.
\par\medskip
If $d=0$, $\supp(\ff)$ is a subset of a divisor with simple normal
crossing $D$.
The properties (i) and (ii) are immediate
consequences of the two lemmas of functoriality \ref{LemmeTroisChClArt1} (iii)
and (iv).
\par\medskip
Suppose now that Proposition \ref{LemmeQuatreChClArt1} is
proved for torsion sheaves which can be
desingularized in $d-1$ steps. Let $\ff$
be a torsion sheaf which can be desingularized with at most $d$
blowups. Let $(\tix,\tau )$ be such a desingularization. We write $\tau $
as $\ti\tau \, \circ\, \tau _{1}$, where
$\ti\tau $ is the first blowup in $\tau $ with $E$ as
exceptional divisor, as shown in the following dia\-gram:
\[
\xymatrix{
&\tix\ar[d]^{\tau _{1}}\\
E\ar[r]^{i_{E}\he}\ar[d]_{q}&\tix_{1}\ar[d]^{\ti\tau }\\
Y\ar[r]_{i_{Y}\he}&X
}
\]
Then $\tau _{1}$ consists of at most $d-1$
blowups and is a desingularization of the sheaves
$\tore{j}{\ff}{\ti\tau }$, $0\leq j\leq n$.
By induction, we can consider the following expression in
$H\ee_{D}(\ti{X}_{1},\Q)$:
\[
\gamma \bigl(\tix_{1},\ff\bigr)=\ds\sum _{j=0}^{n}\, (-1)^{j}
\ch\Bigl(
\tore{j}{\ff}{\ti\tau }\Bigr).
\]
We have
\begin{align*}
i\ee_{E}\gamma \bigl(\tix_{1},\ff\bigr)&=
\sum _{j=0}^{n}\, (-1)^{j}\ch\be^{E}\Bigl(i\pe_{E}
\bigl[\tore{j}{\ff}{\ti\tau }\bigr]\Bigr)&&
\textrm{by induction, property (ii)}\\
&=\ch\be^{E}\bigl(i\pe_{E}\ti\tau \pe\cro\ff\bigr)&&\textrm{by $(\textrm{W}_{n-1})$}\\
&=\ch\be^{E}\bigl(q\pe i\pe_{Y}\cro\ff\bigr)=
q\ee\ch\be^{Y}\bigl(i\pe_{Y}\cro\ff\bigr)&&\textrm{by
$(\textrm{F}_{n-1}).$}
\end{align*}
By Proposition \ref{PropositionDeuxInsertC} (vi),
there exists a unique class $\ch(\ff,\tau )$ on $X$ such that
$\gamma \bigl(\tix_{1},\ff\bigr)=\ti\tau \ee\ch\bigl(\ff,\tau \bigr)$.
Now
\begin{align*}
\tau \ee\ch(\ff,\tau )&=\tau \ee_{1}\gamma \bigl(\tix_{1},\ff\bigr)\\
&=\sum _{j=0}^{n}\, (-1)^{j}\ch\Bigl(\tau \pe_{1}
\bigl[\tore{j}{\ff}{\ti\tau }\bigr]\Bigr)&&
\textrm{by induction, property (i)}\\
&=\ch\bigl(\tau \pe_{1}\ti\tau \pe[\ff]\bigr)=\ch\bigl(\tau
\pe[\ff]\bigr).
\end{align*}
Suppose that we have two resolutions. We dominate them by a third
one, according to the diagram:
\[
\xymatrix{
&W\ar[dl]_{\mu }\ar[dd]^{\delta }\ar[dr]^{\breve{\mu }}&\\
\tix_{1}\ar[dr]_{\tau }&&\tix_{2}\ar[dl]^{\breve{\tau }}\\
&X& }
\]
Then
\begin{align*}
\delta \ee\ch(\ff,\tau )&=\mu \ee\ch\bigl(\tau \pe[\ff]\bigr)\\
&=\ch\bigl(\mu \pe\tau \pe[\ff]\bigr)\quad
&&\textrm{by the first lemma of functoriality \ref{LemmeTroisChClArt1} (iii)}\\
&=\ch\bigl(\delta \pe[\ff]\bigr)=\delta \ee\ch(\ff,\breve{\tau })&&
\textrm{by symmetry.}
\end{align*}
The map $\delta \ee$ being injective, $\ch(\ff,\tau )=\ch(\ff,\breve{\tau
})$, and we can therefore define $\ch(\ff)$ by $\ch(\ff)=\ch(\ff,\tau
)$ for any desingularization $\tau $ of $\ff$ with at most $d$ blowups.
\par\medskip
We have shown that (i) is true when $\tau $ consists of at most $d$
blowups.
In the general case, let
$\apl{\tau }{\ti{X}_{1}}{X}$ be an arbitrary
desingularization of $\ff$ and $\breve{\tau }$ be a desingularization of
$\ff$ with at most $d$ blowups. We can find $W$, $\mu $ and $\breve{\mu }$
as before. Then
\begin{align*}
{\mu }\ee\be\ch\bigl(\tau \pe\be[\ff]\bigr)&=
\ch\bigl(\delta \pe[\ff]\bigr)&&\textrm{by the first
functoriality lemma \ref{LemmeTroisChClArt1} (iii)}\\
&=\breve{\mu }\ee\be\ch\bigl(\breve{\tau }\pe\be[\ff]\bigr)&&\textrm{by the first
functoriality lemma \ref{LemmeTroisChClArt1} (iii)}\\
&=\breve{\mu }\ee\be\breve{\tau }\ee\be\ch(\ff)&&\textrm{since $\breve{\tau }$
consists of at most $d$ blowups}\\
&=\mu \ee\be\tau \ee\be\ch(\ff).
\end{align*}
It remains to show (ii). For this, we desingularize
$\supp (\ff)\cup Y$ exactly as
in the proof of the second lemma of functoriality \ref{LemmeTroisChClArt1} (iv).
We have a
diagram
\[
\xymatrix{
\breve{D}_{i}\ar[r]_{i_{\breve{D}_{i}}\he}\ar[d]_{q_{i}}&\tix\ar[d]^{\tau }\\
Y\ar[r]_{i_{Y}\he}&X
}
\]
where $q_{i}\ee$ is injective for at least one $i$. Then
\begin{align*}
q_{i}\ee
\bigl(i_{Y}\ee\ch(\ff)\bigr)&=i_{\breve{D}_{i}}\ee\tau \ee\ch(\ff)=
i_{\breve{D}_{i}}\ee\ch(\tau \pe[\ff])&&\textrm{by (i)}\\
&=\ch\be^{\breve{D}_{i}}\bigl(i_{\breve{D}_{i}}\pe\tau \pe[\ff]\bigr)&&
\textrm{by the second lemma of functoriality \ref{LemmeTroisChClArt1} (iv)}\\
&=\ch\be^{\breve{D}_{i}}\bigl(q_{i}\pe\, i_{Y}\pe[\ff]\bigr)
=q_{i}\ee\ch\be^{Y}\bigl(i_{Y}\pe
[\ff]\bigr)&&\textrm{by $(\textrm{F}_{n-1}).$}
\end{align*}
Thus $i_{Y}\ee\ch(\ff)=\ch\be^{Y}\bigl(i_{Y}\pe\, [\ff]\bigr)$, which proves
(ii).
\end{proof}
We have now completed the existence part of Theorem
\ref{TheoremTroisTrois} for torsion sheaves.
\par\medskip
We turn to the proof of Proposition
\ref{PropositionUnChClArt1}. So doing,
we establish almost all the properties listed in the induction hypotheses
for torsion sheaves.
\begin{proof}
(i) Let $(\ti{X},\tau )$ be a desingularization of
$\supp(\ff)\cup\supp(\hh)$ and $D$ be the associated simple
normal crossing divisor. Then $\tau \pe\ff$, $\tau \pe\g$
and $\tau \pe\hh$ belong to $\kan_{D}\he(\tix)$ and
$\tau \pe\ff+\tau \pe\hh=\tau \pe\g$ in $\kan_{D}\he(X)$. Thus, by
Proposition \ref{LemmeQuatreChClArt1} (i),
\[
\tau \ee\bigl[\ch(\ff)+\ch(\hh)\bigr]=\ch\bigl(\tau \pe[\ff]\bigr)
+\ch\bigl(\tau \pe[\hh]\bigr)
=\ch\bigl(\tau \pe[\g]\bigr)=\tau \ee\ch(\g).
\]
The map $\tau \ee$ being injective, we get the Whitney formula
for torsion sheaves.
\par\medskip
(ii) The method is the same:
let $x=[\g]$ and
let $\tau $ be a desingularization of
$\g$. Then, by Proposition \ref{LemmeQuatreChClArt1} (i)
and Proposition \ref{LemmeTroisChClArt1} (i),
\[
\tau \ee\ch([\eee]\, .\, [\g])=\ch\bigl(\tau \pe[\eee]\, .\, \tau \pe[\g]\bigr)=
\ba\ch\bigl(\tau \pe[\eee]\bigr)\, .\, \ch\bigl(\tau \pe[\g]\bigr)=
\tau \ee\bigl(\ba\ch(\eee)\, .\, \ch(\g)\bigr).
\]
\par\medskip
(iii) This property is known when $f$ if the immersion of a
smooth submanifold and when $f$ is a bimeromorphic morphism
by Proposition \ref{LemmeQuatreChClArt1}. Let us consider now
the general case. By Grauert's direct image theorem, $f(X)$ is an
irreducible analytic
subset of $Y$. We desingularize $f(X)$ as an abstract complex space. We
get a connected smooth manifold $W$ and a
bimeromorphic morphism $\apl{\tau }{W}{f(X)}$ obtained as a succession of
blowups with smooth centers in $f(X)$.
We perform a similar sequence of blowups, starting from $Y_{1}=Y$ and
blowing up at each step in $Y_{i}$ the smooth center blown up at the
\mbox{$i$-th} step of the desingularization of $f(X)$.
Let $\apl{\pi _{Y}\he}{\tiy}{Y}$ be
this morphism.
The strict transform of $f(X)$ is $W$. The map $\xymatrix{\tau :\tau ^{-1}
\bigl(f(X)_{\mathrm{reg}}\bigr)\ar[r]^-{\sim}&f(X)_{\mathrm{reg}}}$ is an
isomorphism. So we get a morphism $\flgd{f(X)_{\mathrm{reg}}}{W}$ which
is in fact a meromorphic map
from $f(X)$ to $W$, and finally, after
composition on the left by $f$, from $X$ to $W$.
We desingularize this morphism:
\[
\xymatrix{
\tix\ar@{->>}[dr]^{\ti{f}}\ar[d]_{\pi _{X}\he}&\\
X\ar@{.>}[r]&W
}
\]
and we get the following global
diagram, where $\pi _{X}\he$ is a bimeromorphic map:
\[
\xymatrix@C=8ex{
\tix\ar[r]^{\ti{f}}\ar[d]_{\pi _{X}\he}&W\ar[r]^{i_{W}\he}\ar[d]^{\tau
}&\tiy\ar[d]^{\pi _{Y}\he}\\
X\ar[r]_-{f}&f(X)\ar[r]&Y
}
\]
Now $f\circ \pi_{X}\he=\pi _{Y}\he\circ\bigl(
i_{W}\he\circ\ti{f}\, \bigr)$, and we know the functoriality formula for $\pi
_{X}\he$, $\pi _{Y}\he$ and $i_{W}\he$
by Proposition \ref{LemmeQuatreChClArt1}. Since $\pi _{X}\ee$ is injective,
it is enough
to show the functoriality formula for
$\ti{f}$. So we will assume that $f$ is onto.
\par\medskip
Let $(\tau ,\tiy)$ be a
desingularization of $\ff$. We have the diagram
\[
\xymatrix@C=8ex{\hspace*{-5ex}X\times_{Y}\he\tiy\ar[r]\ar[d]&\tiy\ar[d]
^{\tau}\\
X\ar[r]^{f}&Y}
\]
where
$\ti{\tau} ^{-1}(\supp\ff)=D\suq\tiy$ is a divisor with simple normal
crossing and the map
$\flgd{X\times_{Y}\he\tiy}{X}$ is a bimeromorphic morphism.
We have a meromorphic map $\xymatrix{X\ar@{.>}[r]&
X\times_{Y}\he\tiy}$\!\!\!\!,
and we desingularize it by a morphism
$\flgd{T}{X\times_{Y}\he\tiy}$\!\!\!. Therefore, we obtain the following
commutative diagram, where $\apl{\pi }{T}{X}$ is a bimeromorphic map:
\[
\xymatrix@C=8ex
{
T\ar[r]^-{\ti{f}}\ar[d]_{\pi }&\tiy\ar[d]^{\tau }\\
X\ar[r]^{f}&Y
}
\]
Therefore we can assume that $\supp(\ff)$ is included in a divisor
with simple normal crossing $D$. We desingularize $f^{-1}(D)$ so that
we are led to the case $\supp(\ff)\suq D$, where $D$ and $f^{-1}(D)$ are
divisors with simple normal crossing in $Y$ and $X$ respectively.
In this case, we can use the first lemma of functoriality
\ref{LemmeTroisChClArt1} (iii).
\end{proof}
\subsection{The case of sheaves of positive rank}\label{CasPositif}
In this section, we consider the case of sheaves of arbitrary rank.
We are going to introduce the main tool of the construction, namely a
d\'{e}vissage theorem for coherent analytic sheaves.
Let $X$ be a
complex compact manifold and $\ff$
an analytic coherent sheaf on $X$. We have
seen in section \ref{SubsectionTorsionsheaves} how to define $\ch(\ff)$
when $\ff$ is a torsion sheaf.
\par\medskip
Suppose that $\ff$ has strictly positive generic rank.
When $\ff$ admits a global locally free
resolution, we could try to define $\ch(\ff)$ the usual way.
As explained in the introduction, this condition on $\ff$ is not
necessarily fulfilled.
Even if such a resolution exists, the definition of $\ch(\ff)$
depends a priori on this resolution.
A good substitute for a locally free resolution
is a locally free quotient with maximal rank,
since the kernel is
then a torsion sheaf.
Let $\ff_{\textrm{tor}}\he\suq\ff$ be the maximal torsion
subsheaf of $\ff$. Then $\ff$ admits a locally free quotient $\eee$ of maximal
rank if and only if $\ff\bigm/_{\ds\ff_{\textrm{tor}}\he}$ is locally
free. In this case, $\eee=\ff\bigm/_{\ds\ff_{\textrm{tor}}\he}$.
\par\medskip

Unfortunately, the existence of such a quotient is
not assured (for instance,
take a torsion-free sheaf which is not locally free),
but we will show that it exists up to a bimeromorphic morphism.
\begin{theorem}\label{TheoremeUnChClArt1}
Let $X$ be a complex compact manifold and $\ff$ a coherent analytic
sheaf on $X$. There exists a bimeromorphic morphism
$\apl{\sigma }{\tix}{X}$\!\!, which is a finite composition
of blowups with smooth centers,
such that $\sigma \ee\ff$ admits a locally free quotient of
maximal rank on $\tix$. Such quotients are unique,
up to a unique isomorphism.
\end{theorem}
\begin{proof}
Let $r$ be the rank of $\ff$.
We define a universal set $\tix$ by $\tix=\coprod_{x\in X}\gr\ee\bigl(
r,\ff_{\vert x}\bigr)$ where
$\gr\ee\bigl(
r,\ff_{\vert x}\bigr)$ is the dual grassmannian of
quotients of $\ff_{\vert x}$ with rank $r$.
The set $\tix$ is the disjoint union of all the
quotients of rank $r$ of all fibers of $\ff$. The
canonical map $\apl{\sigma }{\tix}{X}$
is an isomorphism on
$\sigma ^{-1}\bigl(\ff_{\reg}\bigr)$. We will now endow $\tix$ with
the
structure of a reduced complex space.
\par\medskip
Let us first argue locally. Let
$\xymatrix{\oo_{\vert U}
^{p}\ar[r]^-{M}&\oo_{\vert U}^{q}\ar[r]&\ff_{\vert U}\he\ar[r]&0
}$ be a presentation of $\ff$ on an open set $U$. Here, $M$ is an
element of $\Mg_{q,p}\bigl(\oo_{U}\he\bigr)$.
Then, for every $x$ in $U$, we get the exact sequence
\[
\xymatrix@C=8ex{\C^{p}\ar[r]^-{M(x)}&\C^{q}\ar[r]^{\pi _{x}\he}
&\ff_{\vert x}\he\ar[r]&0.
}
\]
We have an inclusion
\[
\xymatrix@C=8ex{
\gr\ee\bigl(r,\ff_{\vert x}\he\bigr)\, \ar@{^{(}->}[r]&\gr\ee\bigl(r,\C^{q}\bigr)
\ar[r]^-{\sim}&\gr(q-r,q)
}
\]
given by
\[
\xymatrix@C=8ex{
(q,Q)\ar@{|->}[r]&(q\circ\pi _{x}\he,Q)\ar@{|->}[r]&\ker(q\circ\pi
_{x}\he),
}
\]
where $q$ and $Q$ appear in the following diagram:
\[
\xymatrix@C=8ex{\C^{p}\ar[r]^-{M(x)}
&\C^{q}\ar[r]^{\pi _{x}\he} \ar@{.>}[dr]_{\ti{q}}
&\ff_{\vert x}\ar[r]\ar[d]^{q}&0
\\
&&Q\ar[d]&\\
&&0&
}
\]
Therefore, we have a fibered inclusion
$\xymatrix{\sigma ^{-1}(U)\, \ar@{^{(}->}[r]&
U\times\gr(q-r,q)}$. We know that
\[
\sigma ^{-1}(U)=\bigl\{(x,E)\in U\times\gr(q-r,q)\ \textrm{such that\ }
\im M(x)\suq E\bigr\}.
\]
Let
$(e_{1},\dots ,e_{q})$ be the canonical basis of $\C^{q}$, and
$e_{1}\ee,\dots ,e_{q}\ee$ its dual basis. We can suppose that $e_{1}\ee,\dots
,e_{q-r}\ee$ are linearly independent on $E$. We parametrize $\gr(q-r,q)$
in the neighborhood of $x$ by a matrix $A=
\bigl(a_{i,j}\bigr)\in\Mg_{r,q-r}(\C\, )$. The
associated vector space will be spanned by the columns of the matrix
$
\begin{pmatrix}
\id_{q-r}\\
A
\end{pmatrix}
$.
Writing
$M=\Bigl(M_{j}^{i}\Bigr)
_{\genfrac{}{}{0pt}{}{1\leq j\leq q}{1\leq i\leq p}}\he$,
$\im(x)$
is a subspace of $E$ if and only if for all $i$, $1\leq i\leq p$,
\[
\Bigl(M_{1}^{i}(x)e_{1}+\cdots +M_{q}^{i}(x)e_{q}\Bigr)\we\bwe_{l=1}^{q-r}
\bigl(e_{l}+a_{1,l}e_{q-r+1}+\cdots +
a_{r,l}e_{q}\bigr)=0
\]
in $\bwe^{q-r+1}\C^{q}$.
This is clearly an analytic condition in the
variables $x$ and $a_{ij}$, thus $\sigma ^{-1}(U)$
is an analytic subset of $U\times \gr(q-r,q)$. We endow $\sigma ^{-1}(U)$
with the associated \emph{reduced} structure.
\par\medskip
We must check carefully that the structure
defined above does not depend on the chosen presentation.
\par\medskip
Let us consider two resolutions of $\ff$ on $U$
\[
\xymatrix@C=8ex@R=1.5ex{
\oo_{\vert U}^{p}\ar[r]^{M}&\oo_{\vert U}^{q}\ar[r]^{\pi }&\ff_{\vert
U}\he
\ar[r]&0
\\
\oo_{\vert U}^{p'}\ar[r]^{M'}&\oo_{\vert U}^{q'}\ar[r]^{\pi '}&\ff_{\vert U}
\he\ar[r]&0
}
\]
and a $(q',q)$ matrix $\apl{\alpha }{\oo_{\vert U}^{q}}{\oo_{\vert U}^{q'}}$
such that the diagram
\[
\xymatrix@C=8ex{
\oo_{\vert U}^{p}\ar[r]^{M}&\oo_{\vert U}^{q}\ar[r]^{\pi }
\ar[d]^{\alpha }&\ff_{\vert U}\he\ar[d]^{\id}\ar[r]&0\\
\oo_{\vert U}^{p'}\ar[r]^{M'}&\oo_{\vert U}^{q'}\ar[r]^{\pi '}&\ff_{\vert U}
\he\ar[r]&0
}
\]
commutes. The morphism
\quad $\xymatrix{
\sigma ^{-1}(U)\vphantom{\Bigl(}\ar[r]^{\id}\ar@{^{(}->}[d]&\sigma ^{-1}(U)
\vphantom{\Bigl(}
\ar@{^{(}->}[d]\\
U\times \gr(q-r,q)\ar[r]&U\times \gr(q'-r,q')
}$
\par
is given by
$\xymatrix{
(x,E)\ar@{|->}[r]&\bigl(x,\pi _{x}'{^{-1}}\pi _{x}\he(E)\bigr)
}$
according to the following diagram:
\[
\xymatrix@C=8ex{
\C^{p}\ar[r]^{M(x)}&\C^{q}\ar[r]^{\pi _{x}\he}
\ar[d]^{\alpha (x)}&\ff_{\vert x}\he\ar[r]\ar[d]^{\id}&0\\
\C^{p'}\ar[r]_{M'(x)}&\C^{q'}\ar[r]_{\pi '_{x}}&\ff_{\vert
x}\he\ar[r]&0
}
\]
Since $\pi '_{x}\bigl(\alpha (x)(E)\bigr)=\pi _{x}^{\vphantom{1}}(E)$, we have
$\pi _{x}'{^{-1}}\pi _{x}^{\vphantom{1}}(E)=\pi _{x}'{^{-1}}\pi '_{x}
\bigl(\alpha (x)(E)\bigr)$. We can write this
\[
\pi _{x}'{^{-1}}\pi _{x}^{\vphantom{1}}(E)=\alpha (x)(E)+\ker \pi '_{x}=
\alpha (x)(E)+\im M'(x).
\]
If $(x_{0},E_{0})$ is an element of $\sigma ^{-1}(U)\times \gr(q-r,q)$,
then
$\alpha (x_{0})(E_{0})+\im M'(x_{0})$ belongs to $\gr(q'-r,q')$. We can
suppose as above that $e_{1}\ee,\dots ,e_{q-r}\ee$ are linearly independent
on $E_{0}$. Therefore, $E_{0}$ is spanned by $q-r$ vectors $A_{1},\dots
,A_{q-r}$ where
$A_{i}=\bigl(0,\dots ,1,\dots ,a_{1,i},\dots ,a_{q-r,i}\bigr)
{}^{\mathrm{T}}$,
the first ``$1$'' having index $i$.
Then $\alpha (x_{0})(E_{0})+\im M'(x_{0})$ is spanned by the vectors
$\bigl(\alpha (x_{0})(A_{i})+M'{}^{j}(x_{0})\bigr)
_{\genfrac{}{}{0pt}{}{1\leq i\leq q-r}{1\leq j\leq p'}}
\he
$, where the $M'{}^{j}$ are the columns of $M'$. We can
find $q'-r$ independent vectors in this family, and this property holds
also in a neighborhood of $(x_{0},E_{0})$. We denote these vectors by
$\bigl(\alpha (x)(A_{i_{k}})+M'{}^{j_{k}}(x)\bigr)_{1\leq k\leq q'-r}$.
Let us
define
\[
f\bigl(A_{1},\dots ,A_{q-r},x\bigr)=\spa\bigl(
\alpha (x)(A_{_{i_{k}}})+M'{}^{j_{k}}(x)\bigr)_{1\leq k\leq q'-r}.
\]
This is a holomorphic map defined in a neighborhood of $(x_{0},E_{0})$
with values in $U\times \gr(q'-r,q')$.
On the same pattern, we can define another map $g$ on
a neighborhood of $f\bigl(x_{0},E_{0}\bigr)$ with values in
$U\times \gr(q-r,q)$.
The couple $(f,g)$ defines an isomorphism of complex spaces.
This proves that $\tix$ is endowed with the structure of an intrinsic
reduced complex analytic space (not generally smooth).
\par\medskip
We define a subsheaf $\nn$ of $\sigma \ee\ff$ by
\[
\nn(V)=\bigl\{s\in\sigma \ee\ff(V)\textrm{\ such that\ }
\forall \bigl(x,\{q,Q\}\bigr)\in V,\ s_{x}\he\in\ker q
\bigr\}.
\]
Remark that $\nn$ is supported in the singular locus of $\sigma \ee\ff$.
\begin{lemma}\label{LemmeSixChClArt1}
$\nn$ satisfies the following properties:
\begin{enumerate}
  \item [(i)] $\nn$ is a coherent subsheaf of $\sigma \ee\ff$.

  \item [(ii)] $\sigma \ee\ff/_{\ds\nn}$ is locally free and
  $\ran\bigl(\sigma \ee\ff/_{\ds\nn}\bigr)=\ran(\ff)$.
\end{enumerate}
\end{lemma}
\begin{proof}
We take a local presentation
$
\xymatrix
{
\oo_{\vert U}^{p}\ar[r]^{M}&\oo_{\vert U}^{q}\ar[r]^{\pi }&\ff_{\vert
U}\he\ar[r]&0
}
$ of $\ff$. Then $\sigma \ee\ff$ has the presentation
\[
\xymatrix
{
\oo_{\vert \sigma ^{-1}(U)}^{p}\ar[r]^{M\circ\, \sigma }
&\oo_{\vert \sigma ^{-1}(U)}^{q}\ar[r]&\sigma \ee\ff_{\vert
\sigma ^{-1}(U)}\he\ar[r]&0.
}
\]
Let
$\xymatrix{(x,E)\ar@{|->}[r]&\bigl(f_{1}(x,E),\dots ,f_{q}(x,E)\bigr)}$
be a
section of $\sigma \ee\ff$ on $V\suq \sigma ^{-1}(U)$. Then $s$ is a
section of $\nn$ if and only if for every $(x,E)$ in $V$,
$\bigl(f_{1}(x,E),\dots ,f_{q}(x,E)\bigr)$ is an element of $E$.
Let $U_{q-r,q}\he$ be the universal bundle on $\gr(q-r,q)$.
Then $U_{q-r,q\, \vert V}\he$ is a subbundle
of $\oo_{\vert V}^{q}$ and $\nn_{\vert V}\he$ is the image of
$U_{q-r,q\, \vert V}\he$ by the morphism
$
\xymatrix{U_{q-r,q\, \vert V}\ar@{^{(}->}[r]&
\oo_{\vert \sigma ^{-1}(U)}^{q}\ar[r]^{\pi }
&\sigma \ee\ff_{\vert
\sigma ^{-1}(U)}\he}
$. So $\nn$ is coherent.
\par\medskip
(ii) Let us define $\eee$ by $\eee=\sigma \ee\ff/_{\ds\nn}$. For all $(x,E)$
in $V$, we have an exact sequence
\[
\sutrd{\nn_{\vert (x,E)}\he}{\ff_{\vert x}\he}
{\eee_{\vert (x,E)}\he}
\]
and a
commutative diagram
\[
\xymatrix@C=8ex{E\, \ar@{^{(}->}[r]\ar[d]&\C^{q}\ar[d]^{\pi _{x}\he}\\
\nn_{\vert (x,E)}\he\ar[r]&\ff_{\vert x}\he
}
\]
The first vertical arrow is the morphism
$\flgd{U_{q-r,q\, \vert U}\he}{\nn}$
restricted at $(x,E)$, so it is onto.
Thus $\pi _{x}\he(E)$ is the image of the
morphism $\flgd{\nn_{\vert (x,E)}\he}{\ff_{\vert x}\he}$\!\!. Since we have an
exact sequence
\[
\xymatrix{
0\ar[r]&\pi _{x}\he(E)\ar[r]&\ff_{\vert x}\he\ar[r]^{q}&Q\ar[r]&0
}
\]
where $(x,E)=(x,Q)$, then $\dim\pi _{x}\he(E)=\dim\ff_{\vert x}\he-r$. Since
$\dim\ff_{\vert x}\he=\dim\pi _{x}\he(E)+\dim\eee_{\vert (x,E)}\he$, we get
$\dim\eee_{\vert (x,E)}\he=r$.
We can see that $\nn$ is a torsion sheaf,
for $\eee$ is locally free of rank $r$.
\end{proof}
\par\medskip
We can now finish the proof of Theorem \ref{TheoremeUnChClArt1}. Using
Hironaka's theorem, we desingularize the complex space $\tix$. We get a
succession of blowups
with smooth centers $\apl{\ti\sigma }{\tix'}{\tix}$ where $\tix'$
is smooth. By the Hironaka-Chow lemma (see \cite[Th.\!\! 7.8]{AG}),
we can suppose that $\tau =\sigma \circ\ti\sigma $
is a succession of blowups with smooth centers.
Since $\eee$ is locally free,
the following sequence is exact:
\[
\sutrgd{\ti\sigma \ee\nn}{\tau \ee\ff}{\ti\sigma \ee\eee}\!.
\]
\par\medskip
Therefore $\tau \ee\ff$ admits a locally free quotient of maximal rank.
The unicity is clear.
This finishes the proof.
\end{proof}
\subsection{Construction of the classes in the general case}\label{CasGeneral}
Let $X$ be
a complex compact manifold of
dimension $n$.
\subsubsection{} Let $\ff$ be a coherent sheaf on $X$ which
has a locally free quotient of maximal rank.
We have an exact sequence
\[
\sutrgd{\ttt}{\ff}{\eee}
\]
where $\ttt$ is a torsion sheaf and $\eee$ is
locally free. Then we define $\ch(\ff)$ by
$\ch(\ff)=\ch(\ttt)+\ba\ch(\eee)$,
where $\ch(\ttt)$ has been constructed in part
\ref{SubsectionTorsionsheaves}. Remark that $\ch(\ff)$ depends only on
$\ff$, the exact sequence $\sutrgd{\ttt}{\ff}{\eee}$ being unique up to
(a unique) isomorphism.
\par\medskip
We state now the Whitney formulae which apply to the Chern
characters we have defined above.
\begin{proposition}\label{PropXChClArt1}
Let $\sutrgd{\ff}{\g}{\hh}$ be an exact sequence of coherent analytic
sheaves on $X$. Then $\ch(\ff)$, $\ch(\g)$ and $\ch(\hh)$
are well defined and verify $\ch(\g)=\ch(\ff)+\ch(\hh)$ under
any of the following hypotheses:
\begin{enumerate}
  \item [(i)] $\ff$, $\g$, $\hh$ are locally free sheaves on $X$.

  \item [(ii)] $\ff$, $\g$, $\hh$ are torsion sheaves.

  \item [(iii)] $\g$ admits a locally free quotient of maximal rank and $\ff$
  is a torsion sheaf.
\end{enumerate}
\end{proposition}
\begin{proof}
(i) If $\ff$, $\g$, $\hh$ are locally free sheaves on $X$, then $\ch(\ff)
=\ba\ch(\ff)$, $\ch(\g)
=\ba\ch(\g)$, $\ch(\hh)
=\ba\ch(\hh)$ and we use Proposition
\ref{PropositionDeuxComplementQuatreUnArt1} (i).
\par\medskip
(ii) This is Proposition \ref{PropositionUnChClArt1} (i).
\par\medskip
(iii) Let
$\eee$ be the locally free quotient of maximal rank of $\g$. We have an
exact sequence
\[
\sutrgd{\ttt}{\g}{\eee}
\]
where $\ttt$ is a torsion sheaf.
Since $\ff$ is a torsion sheaf, the morphism
$\sutr{\ff}{\g}{\eee}$ is identically zero.
Let us define $\ttt'$ by the exact sequence
\[
\xymatrix{0\ar[r]&\ttt'\ar[r]&\hh\ar[r]&\eee\ar[r]&0.}
\]
Then $\ttt'$ is a torsion sheaf and we have the exact sequence
of torsion sheaves
\[
\xymatrix{0\ar[r]&\ff\ar[r]&\ttt\ar[r]&\ttt'\ar[r]&0.}
\]
Thus $\hh$ admits a locally free quotient of maximal rank, so that $\ch(\hh)$
is defined and
\begin{align*}
\ch(\hh)&=\ba\ch(\eee)+\ch(\ttt')=\ba\ch(\eee)+\ch(\ttt)-\ch(\ff)
&&\textrm{by (ii)}\\
&=\ch(\g)-\ch(\ff).
\end{align*}
\end{proof}
Let us now look at the functoriality properties with respect to
pullbacks.
\begin{proposition}\label{PropYChClArt1}
Let $\apl{f}{X}{Y}$ be a holomorphic map. We assume that
\begin{enumerate}
  \item [--] $\dim Y=n$ and $\dim X\leq n$,

  \item [--] if $\dim X=n$, $f$ is surjective.
\end{enumerate}
Then for every coherent sheaf on $Y$ which admits a locally free
quotient of maximal rank, the following properties hold:
\begin{enumerate}
  \item [(i)] The Chern characters $\ch\bigl(
  \tore{i}{\ff}{f}\bigr)$ are well defined.

  \item [(ii)]
  $f\ee\ch(\ff)=\ds\sum _{i\geq 0}(-1)^{i}\ch\bigl(
  \tore{i}{\ff}{f}\bigr)$.
\end{enumerate}
\end{proposition}
\begin{proof}
(i) If $\dim X<n$, the classes
$\ch\bigl(
  \tore{i}{\ff}{f}\bigr)$ are
defined by the induction property $(\textrm{E}_{n-1})$. If $\dim X=n$
and
$f$ is surjective, then $f$ is generically finite. Thus all the sheaves
$\tore{i}{\ff}{f}$, $i\geq 1$,
are torsion sheaves on $X$, so their Chern classes are defined
by Proposition \ref{PropositionUnChClArt1}.
The sheaf $f\ee\ff$ admits on $X$ a locally
free quotient of maximal rank so that $\ch(f\ee\ff)$ is well defined.
\par\medskip
(ii) We have
an exact sequence
\[
\sutrgd{\ttt}{\ff}{\eee}
\]
where $\ttt$ is a torsion sheaf and $\eee$
is a locally free sheaf. Remark that, for $i\geq 1$,
$\tore{i}{\ff}{f}\simeq\tore{i}{\ttt}{f}$. Thus, by Proposition
\ref{PropositionDeuxComplementQuatreUnArt1}
(ii) and Proposition \ref{PropositionUnChClArt1} (iii),
\begin{align*}
\sum _{i\geq 0}(-1)^{i}\ch
\bigl(\tore{i}{\ff}{f}\bigr)&=\ba\ch \bigl(f\ee\eee\bigr)+
\ch\bigl(f\ee\ttt\bigr)+
\sum_{i\geq 1}(-1)^{i}\ch\Bigl(\tore{i}{\ttt}{f}\Bigr)\\
&=f\ee\ba\ch(\eee)+\ch\bigl(f\pe[\ttt]\bigr)
=f\ee\bigl(\ba\ch(\eee)+\ch(\ttt)\bigr)=
f\ee\ch(\ff).
\end{align*}
\end{proof}
\subsubsection{} We consider now an arbitrary coherent sheaf $\ff$
on $X$. By Theorem \ref{TheoremeUnChClArt1}, there exists
$\apl{\sigma } {\tix}{X}$ obtained as a finite composition of
blowups with smooth centers
such that $\sigma \ee\ff$ admits a locally free quotient
of maximal rank.
This is the key property for the definition of $\ch(\ff)$ in full
generality.
\begin{theorem}\label{LemmeHuitChClArt1}
There exists a class $\ch(\ff)$ on $X$ uniquely determined
by $\ff$ such
that:
\begin{enumerate}
  \item [(i)]
If $\apl{\sigma }{\tix}{X}$ is a succession of
blowups with smooth centers such that $\sigma \ee\ff$
admits a locally free quotient of maximal rank, then
$\sigma \ee\ch(\ff)=\ds\sum _{i\geq 0}
(-1)^{i}\ch\bigl(\tore{i}{\ff}{\sigma }\bigr)$.

  \item [(ii)] If $Y$ is a smooth submanifold of $X$, $\ch^{Y}\bigl(
  i_{Y}\pe[\ff]\bigr)=i\ee_{Y}\ch(\ff)$.
\end{enumerate}
\end{theorem}
\begin{remark}
By Proposition \ref{PropYChClArt1} (i), all the terms in (i) are defined.
\end{remark}
\begin{proof}
The proof of (i) will use Lemmas \ref{SousLemmeUnArt1},
\ref{SousLemmeDeuxArt1}, and \ref{LemmeOnzeChClArt1}. We will prove the
result by induction on the number $d$ of blowups in $\sigma
$. If $d=0$, $\ff$ admits a locally free quotient
of maximal rank and we can use
Proposition \ref{PropYChClArt1}.
\par\medskip
Suppose now that (i) and (ii) hold at step $d-1$. As usual, we
look at the first blowup in $\sigma $
\[
\xymatrix@C=8ex{
&\tix\ar[d]^{\sigma _{1}}\ar@/^4ex/[dd]^{\sigma }\\
E\ar[r]^{i_{E}\he}\ar[d]_{q}&\tix_{1}\ar[d]^{\ti\sigma }\\
Y\ar[r]_{i_{Y}\he}&X.
}
\]
The sheaves $\tore{j}{\ff}{\sigma}$
are torsion sheaves for $j\geq 1$ and
$\sigma _{1}\ee
\tore{0}{\ff}{\ti{\sigma }}
=\sigma \ee\be\ff$ admits a locally free quotient
of maximal rank. Since $\sigma _{1}\he$ consists of $d-1$ blowups, we
can define by induction a class $\gamma (\ff)$ in
$H\ee_{D}(\ti{X}_{1},\Q)$
as follows:
\[
\gamma (\ff)=\sum _{j\geq 0}\, (-1)^{j}\ch\Bigl(\tore{j}{\ff}{\ti\sigma }\Bigr).
\]
\begin{lemma}\label{SousLemmeUnArt1}
$\sigma _{1}\ee\gamma (\ff)=
\ds\sum _{i\geq 0}(-1)^{i}\ch\bigl(
\tore{i}{\ff}{\sigma}
\bigr)$.
\end{lemma}
\begin{proof}
We have by induction
\[
\sigma \ee_{1}\gamma(\ff)
=
\sum _{p,q\geq 0}\, (-1)^{p+q}
\ch\Bigl[\tore{p}{\tore{q}{\ff}{\ti\sigma}}{\sigma_{1}}
\Bigr]
=
\sum _{p,q\geq 0}\,(-1)^{p+q}\ch(E_{2}^{\, p,q})
\]
where the Tor spectral sequence satisfies
\begin{align*}
E_{2}^{\, p,q}&=\tore{p}{\tore{q}{\ff}{\ti\sigma}}{\sigma_{1}}\\
E_{\,\infty }^{\, p,q}&=\gr^{p}\, \tore{p+q}{\ff}{\sigma}.
\end{align*}
All the $E_{r}^{\, p,q}$, $2\leq r\leq \infty $, are torsion sheaves
except perhaps $E_{r}^{\, 0,0}$.
Remark that no arrow $d_{r}^{\, p,q}$ starts or arrives at $E_{r}^{\,
0,0}$. Thus we have
\[
\sum _{\genfrac{}{}{0pt}{2}{p,q}{p+q\geq 1}}(-1)^{p+q}[E_{2}^{\, p,q}]=
\sum _{\genfrac{}{}{0pt}{2}{p,q}{p+q\geq 1}}(-1)^{p+q}[E_{\infty }^{\, p,q}]=
\sum _{i\geq 1}\, (-1)^{i}\bigl[\tore{i}{\ff}{\sigma}\bigr]
\]
in $\kan_{\torr}\he(X)$. Using Proposition \ref{PropXChClArt1} (ii), we get
\begin{align*}
\sigma _{1}\ee\gamma(\ff)&
=
\ch\bigl(E_{2}^{\, 0,0}\bigr)+
\ch\, \Bigl(\sum _{i\geq 1}\, (-1)^{i}
\tore{i}{\ff}{\sigma}\Bigr)\\
&
=
\sum _{i\geq 0}\, (-1)^{i}\ch\Bigl(
\tore{i}{\ff}{\sigma}\Bigr).
\end{align*}
\end{proof}
\begin{lemma}\label{SousLemmeDeuxArt1}
There exists a unique class $\ch(\ff,\sigma )$ on $X$ such that $\gamma
(\ff)=\ti{\sigma }\ee\ch(\ff,\sigma )$.
\end{lemma}
\begin{proof} We have
\begin{align*}
i_{E}\ee\gamma (\ff)&=i\ee_{E}\Bigl(
\sum _{j\geq 0}\, (-1)^{j}\ch\bigl(
\tore{j}{\ff}{\sigma}
\bigr)
\Bigr)\\
&=\sum _{j\geq 0}\, (-1)^{j}\ch^{E}\bigl(i_{E}\pe
\, \bigl[\tore{j}{\ff}{\sigma}\bigr]
\bigr)&&\textrm{by induction property (ii)}\\
&=\ch\be^{E}\bigl(i_{E}\pe\ti{\sigma }{}\pe[\ff]\bigr)=\ch\be^{E}\bigl(q\pe
i\pe_{Y}[\ff]\bigr)=q\ee\ch\be^{Y}\bigl(i_{Y}\pe[\ff]\bigr)&&
\textrm{by $(\textrm{F}_{n-1}).$}
\end{align*}
By Proposition \ref{PropositionDeuxInsertC} (vi), there exists a
unique class $\ch(\ff,\sigma )$ on $X$ such that
$\gamma (\ff)=\ti\sigma \ee\ch(\ff,\sigma )$.
\end{proof}
Putting Lemma \ref{SousLemmeUnArt1} and Lemma \ref{SousLemmeDeuxArt1}
together
\[
\sigma \ee\ch(\ff,\sigma )=\sigma \ee_{1}\gamma (\ff)=
\sum _{i\geq 0}(-1)^{i}\ch\bigl(\tore{i}{\ff}{\sigma}\bigr).
\]
\begin{lemma}\label{LemmeOnzeChClArt1}
The class $\ch(\ff,\sigma )$ does not depend on $\sigma $.
\end{lemma}
\begin{proof}
As usual, if we have two resolutions, we dominate them by a third one,
as shown in the diagram
\[
\xymatrix{
&\tix\ar[dl]_{\sigma  _{1}}\ar[dr]^{\sigma _{1}'}\ar[dd]^{\sigma}&\\
\tix_{1}\ar[dr]_{\widetilde\sigma }&&\tix_{1}'\ar[dl]^{\widetilde\sigma'}\\
&X&
}
\]
Now
\begin{align*}
\sigma \ee\be\ch(\ff,\widetilde\sigma )&=\sigma_{1} \ee\gamma (\ff)&&
\textrm{by Lemma \ref{SousLemmeDeuxArt1}}\\
&= \sum _{i\geq 0}(-1)^{i}\ch\bigl(
\tore{i}{\ff}{\sigma}
\bigr)&&\textrm{by Lemma \ref{SousLemmeUnArt1}}.
\end{align*}
By symmetry
$\sigma \ee\ch(\ff,\widetilde\sigma)=\sigma \ee\ch(\ff,\widetilde\sigma ')$ and
we get the result.
\end{proof}
We proved the existence statement and part (i) of Theorem
\ref{LemmeHuitChClArt1} if $\sigma $ consists of at most $d$
blowups.
The general case follows using the diagram above.
\par\medskip
We must now prove Theorem \ref{LemmeHuitChClArt1}
(ii).
Let $Y$ be a smooth submanifold of $X$. We choose $\apl{\sigma }{\tix}{X}$
such that $\sigma \ee\ff$ admits a locally free quotient of maximal rank
and $\sigma ^{-1}(Y)$ is a simple
normal crossing divisor with
branches $D_{j}$.
We choose as usual $j$ such that $q_{j}\ee$ is injective,
$q_{j}$ being defined by the
diagram
\[
\xymatrix@C=8ex{
D_{j}\ar[r]^{i_{D_{j}}\he}\ar[d]_{q_{j}}&\tix\ar[d]^{\sigma }\\
Y\ar[r]^{i_{Y}\he}&X
}
\]
We have
\begin{align*}
q_{j}\ee\ch\be^{Y}\bigl(i\pe_{Y}[\ff]\bigr)&=\ch\be^{D_{j}}
\bigl(q_{j}\pe
i_{Y}\pe[\ff]\bigr)=\ch\be^{D_{j}}\bigl(i\pe_{D_{j}}\sigma
\pe[\ff]\bigr)\\
&=\sum _{i\geq 0}(-1)^{i}\, \, i_{D_{j}}\ee\ch\bigl(
\tore{i}{\ff}{\sigma}
\bigr)&&\textrm{by Proposition \ref{PropYChClArt1} (ii).}
\end{align*}
Now, by the point (i),
we have
$\sum _{i\geq 0}(-1)^{i}\ch\bigl(
\tore{i}{\ff}{\sigma}
\bigr)=\sigma \ee\ch(\ff)$. Thus we get
\[
q_{j}\ee\ch\be^{Y}(i_{Y}\pe[\ff])= i\ee_{D_{j}}\sigma
\ee\ch(\ff)=q\ee_{j}\bigl(i\ee_{Y}\ch(\ff)\bigr).
\]
Therefore $\ch\bigl(i\pe_{Y}[\ff]\bigr)=i\ee_{Y}\ch(\ff)$ and the proof is
complete.
\end{proof}
\section{The Whitney formula}\label{SectionWhitneyFormulaArt1}
In the previous section, we achieved an important step in the induction
process by defining the classes $\ch(\ff)$ when $\ff$ is any coherent
sheaf on a \mbox{$n$-dimensional} manifold.
To conclude the proof of Theorem \ref{TheoremTroisTrois}, it remains to
check properties $(\textrm{W}_{n})$, $(\textrm{F}_{n})$ and
$(\textrm{P}_{n})$. The crux of the proof is in fact property
$(\textrm{W}_{n})$. The main result of this section is Theorem
\ref{EncoreUnTheoWhFoArt1}.
The other induction hypotheses will be
proved in Theorem \ref{PropositionDeuxWhitneyFormulaArt1}.
\begin{theorem}\label{EncoreUnTheoWhFoArt1}
$(\emph{W}_{n})$ holds.
\end{theorem}
\par\medskip
To prove Theorem
\ref{EncoreUnTheoWhFoArt1}, we need several reduction steps.
\subsection{Reduction to the case where $\ff$ and $\g$ are locally free and
$ \hh $ is a torsion sheaf}\label{CasLibre}
\label{SectionUnWhitneyFormulaArt1}
\begin{proposition}\label{LemmeUnWhitneyFormulaArt1}
Suppose that $(\emph{W}_{n})$ holds when $\ff$ and $\g$ are locally
free sheaves and $\hh$ is a torsion sheaf. Then $(\emph{W}_{n})$ holds for
arbitrary sheaves.
\end{proposition}
\begin{proof}
We proceed by successive reductions.
\begin{lemma}\label{LemAWhFoArt1}
It is sufficient to prove $(\emph{W}_{n})$ when $\ff$, $\g$, $\hh$
admit a locally free quotient of maximal rank.
\end{lemma}
\begin{proof}
We take a general exact sequence
$\sutrgdpt{\ff}{\g}{\hh}{.}$ Let $\apl{\sigma }{\tix}{X}$ be a
bimeromorphic
morphism such that $\sigma \ee\ff$, $\sigma \ee\g$ and $\sigma \ee\hh$
admit locally free quotients of maximal rank
(we know that such a $\sigma $ exists by Theorem \ref{TheoremeUnChClArt1}).
We have an exact sequence defining $\qq$ and $\ttt_{1}$:
\par\vspace*{2ex}
\begin{center}
\begin{pspicture}(-4.646ex,-2)(10,2)
\psset{xunit=0.2ex,yunit=0.2ex}
\psset{linewidth=0.09292ex}
\psset{arrowinset=0.6}
\rput(-8,0){$\cdots $}\psline{->}(5,1)(25,1)
\rput(60,0){$\tore{1}{\g}{\sigma }$}
\psline{->}(92,1)(112,1)
\rput(145,0){$\tore{1}{\hh}{\sigma }$}
\psline{->}(178,1)(198,1)
\rput(213,2){$\sigma \ee\ff$}
\psline{->}(228,1)(248,1)
\rput(263,2){$\sigma \ee\g$}
\psline{->}(278,1)(298,1)
\rput(313,2){$\sigma \ee\hh$}
\psline{->}(328,1)(348,1)
\rput(356,2){$0$}
\rput(188,-30){$\ttt_{1}$}
\psline{->}(172,-10)(182,-22)
\psline{->}(193,-22)(203,-10)
\rput(167,-55){$0$}
\rput(208,-55){$0$}
\psline{->}(176,-50)(186,-38)
\psline{->}(193,-38)(203,-50)
\rput(238,30){$\qq$}
\rput(219,55){$0$}
\rput(257,55){$0$}
\psline{->}(243,22)(253,10)
\psline{->}(243,38)(253,51)
\psline{->}(223,10)(233,22)
\psline{->}(223,50)(233,38)
\end{pspicture}
\end{center}
\par
Remark that $\ttt_{1}$ is a torsion sheaf.
By Proposition \ref{PropXChClArt1} (iii),
$\qq$ admits
a locally free quotient of maximal rank and
$\ch(\sigma \ee\ff)=\ch(\ttt_{1})+\ch(\qq)$.
Furthermore,
\[
[\ttt_{1}]-\bigl[\tore{1}{\hh}{\sigma }\bigr]+
\bigl[\tore{1}{\g}{\sigma }\bigr]-\dots =0 \mathrm{\ in\ }
\kan_{\tors}\he(\tix).
\]
Then by Proposition \ref{LemmeQuatreChClArt1} (i) and Proposition
\ref{PropXChClArt1} (ii),
\begin{align*}
\sigma \ee\bigl(\ch(\ff)+\ch(\hh)-\ch(\g)\bigr)&=
\ds\sum _{i\geq 0}(-1)^{i}\Bigl[\ch\bigl(\tore{i}{\ff}{\sigma }\bigr)
+\ch\bigl(\tore{i}{\hh}{\sigma }\bigr)-
\ch\bigl(\tore{i}{\g}{\sigma }\bigr)\Bigr]\\
&=\ch\bigl(\sigma \ee\ff\bigr)+\ch\bigl(\sigma
\ee\hh\bigr)-\ch\bigl(\sigma \ee\g\bigr)-\ch(\ttt_{1})\\
&=\ch(\qq)+\ch\bigl(\sigma \ee\hh\bigr)-\ch\bigl(\sigma \ee\g\bigr).
\end{align*}
Since $\sigma \ee$ is injective, Lemma \ref{LemAWhFoArt1} is proved.
\end{proof}
\begin{lemma}\label{LemBWhFoArt1}
It is sufficient to prove $(\emph{W}_{n})$ when $\ff$, $\g$ admit a
locally free quotient of maximal rank and $\hh$ is a torsion sheaf.
\end{lemma}
\begin{proof}
By Lemma \ref{LemAWhFoArt1}, we can assume that $\ff$, $\g$, $\hh$ admit
a locally free quotient of maximal rank.
In the sequel, the letter ``$\ttt$'' will denote
a torsion sheaf and the letter ``$\eee$'' a locally free
sheaf. Let $\eee_{1}$ be the locally free quotient of maximal rank of
$\hh$, so we have an exact sequence
\[
\sutrgdpt{\ttt_{1}}{\hh}{\eee_{1}}{.}
\]
We define $\ff_{1}$ by the exact sequence
\begin{align*}
&\sutrgdpt{\ff_{1}}{\g}{\eee_{1}}{.}
\intertext{Then we get a third exact sequence}\\[-4ex]
&\sutrgdpt{\ff}{\ff_{1}}{\ttt_{1}}{.}
\end{align*}
We have by definition
$\ch(\hh)=\ba\ch(\eee_{1})+\ch(\ttt_{1})$. Thus,
\begin{align*}
\ch(\ff)+\ch(\hh)-\ch(\g)&=\bigl(\ch(\ff)+\ch(\ttt_{1})-\ch(\ff_{1})\bigr)+
\bigl(\ch(\ff_{1})+\ba\ch(\eee_{1})-\ch(\g)\bigr)\\
&\hspace*{13 ex}-\bigl(\ch(\ttt_{1})+\ba\ch(\eee_{1})-\ch(\hh)\bigr)\\
&=\bigl(\ch(\ff)+\ch(\ttt_{1})-\ch(\ff_{1})\bigr)+\bigl(
\ch(\ff_{1})+\ba\ch(\eee_{1})-\ch(\g)\bigr).
\end{align*}
Let
$\eee_{2}$ be the locally
free quotient of maximal rank of $\g$. We define $\ttt_{2}$ by
the exact sequence
\[
\sutrgdpt{\ttt_{2}}{\g}{\eee_{2}}{.}
\]
The morphism from $\g$ to $\eee_{1}$
(via $\hh$) induces a morphism $\flgd{\eee_{2}}{\eee_{1}}$ which remains
of course surjective.
Let $\eee_{3}$ be the kernel of this morphism, then $\eee_{3}$
is a locally free sheaf. We get an exact sequence
\[
\sutrgdpt{\ttt_{2}}{\ff_{1}}{\eee_{3}}{.}
\]
Therefore $\ff_{1}$ admits a
locally free quotient of maximal rank and
$\ch(\ff_{1})=\ch(\ttt_{2})+\ba\ch(\eee_{3})$. On the other hand, by
Proposition \ref{PropXChClArt1} (i),
$
\ba\ch\bigl(\eee_{1}\bigr)+\ba\ch\bigl(\eee_{3}\bigr)=
\ba\ch\bigl(\eee_{2}\bigr)
$,
and we obtain
$
\ch\bigl(\ff_{1}\bigr)+\ba\ch\bigl(\eee_{1}\bigr)-\ch(\g)=
\bigl(\ch(\ttt_{2})+\ba\ch(\eee_{3})\bigr)+
\bigl({\ba\ch}(\eee_{2})-{\ba\ch}(\eee_{3})\bigr)-
\bigl(\ch(\ttt_{2})+\ba\ch(\eee_{2})\bigr)=0
$.
Therefore,
\[
\ch(\ff)+\ch(\hh)-\ch(\g)=\ch(\ff)+
\ch(\ttt_{1})-\ch(\ff_{1}).
\]
Since $\ttt_{1}$ is a torsion sheaf, we are
done.
\end{proof}
We can now conclude the proof of Proposition
\ref{LemmeUnWhitneyFormulaArt1}.
\par\medskip
By Lemma \ref{LemBWhFoArt1}, we can
suppose that $\ff$, $\g$ admit locally free quotients of maximal rank
and $\hh$ is a torsion sheaf.
Let $\eee_{1}$ and $\eee_{2}$ be the locally free
quotients of maximal rank of $\ff$ and $\g$. We define
$\ttt_{1}$ and $\ttt_{2}$ by the two exact
sequences
\begin{align*}
&\sutrgd{\ttt_{1}}{\ff}{\eee_{1}}\\
&\sutrgdpt{\ttt_{2}}{\g}{\eee_{2}}{.}
\end{align*}
The morphism $\flgd{\ff}{\g}$ induces a morphism
$\flgd{\ttt_{1}}{\ttt_{2}}$.
We get a morphism  $\flgd{\eee_{1}}{\eee_{2}}$ with torsion
kernel and cokernel. Since $\eee_{1}$ is a locally free sheaf, this morphism is
injective. In the following diagram, we introduce the cokernels $\ttt_{3}$
and $\ttt_{4}$:
\[
\xymatrix{
&0\ar[d]&0\ar[d]&0\ar[d]&\\
0\ar[r]&\ttt_{1}\ar[d]\ar[r]&\ff\ar[r]\ar[d]&\eee_{1}\ar[r]\ar[d]&0\\
0\ar[r]&\ttt_{2}\ar[d]\ar[r]&\g\ar[r]\ar[d]&\eee_{2}\ar[r]\ar[d]&0\\
&\ttt_{3}\ar[d]&\hh\ar[d]&\ttt_{4}\ar[d]&\\
&0&0&0&
}
\]
By the nine lemma,
$\sutrgd{\ttt_{3}}{\hh}{\ttt_{4}}$ is an exact sequence of torsion sheaves.
Then by Proposition \ref{PropositionUnChClArt1} (i),
\begin{align*}
\ch(\ff)+\ch(\hh)-\ch(\g)&=\ch\bigl(\ttt_{1}\bigr)+\ba\ch\bigl(\eee_{1}\bigr)
+\ch\bigl(\ttt_{3}\bigr)+\ch\bigl(\ttt_{4}\bigr)
-\ch\bigl(\ttt_{2}\bigr)-\ba\ch\bigl(\eee_{2}\bigr)\\
&=
\ba\ch\bigl(\eee_{1}\bigr)+\ch\bigl(\ttt_{4}\bigr)-\ba\ch\bigl(\eee_{2}\bigr).
\end{align*}
This finishes the proof.
\end{proof}
\subsection{A structure theorem for
coherent torsion sheaves of projective dimension
one} \label{SectionDeuxWhitneyFormulaArt1}
In section \ref{SectionUnWhitneyFormulaArt1}
we have reduced the Whitney formula to the particular case where $\ff$ and $\g$
are locally free sheaves and $\hh$ is a torsion sheaf. We are now going to prove
that it is sufficient to suppose that $\hh$ is
the push-forward of a locally free sheaf on a smooth
hypersurface of $X$. The main tool of this section is the following
proposition:
\begin{proposition}\label{PropositionUnWhitneyFormulaArt1}
Let $\hh$ be a torsion sheaf which admits a global locally free
resolution of length two. Then there exist a bimeromorphic morphism
$\apl{\sigma }{\tix}{X}$ obtained by a finite number
of blowups with smooth centers,
a simple normal crossing divisor $D$, $D\suq X$,
and an increasing sequence $\bigl(D_{i}\bigr)
_{1\leq i\leq r}$ of subdivisors of $D$ such
that $\sigma \ee\hh$ is
everywhere locally isomorphic to
$\bop_{i=1}^{r}\oo_{\tix}/_{\ds \ii_{D_{i}}\he}$.
\end{proposition}
\begin{proof}
Let $\sutrgd{\eee_{1}}{\eee_{2}}{\hh}$ be a locally free resolution of
$\hh$, such that $\ran\bigl(\eee_{1}\bigr)=\ran\bigl(\eee_{2}\bigr)=r$.
Recall that the $k$th Fitting ideal of $\hh$ is the coherent ideal sheaf
generated by the determinants of all the $k\times k$ minors of $M$ when $M$
is any local matrix realization in coordinates of the morphism
$\flgd{\eee_{1}}{\eee_{2}}$ (for a general presentation of the Fitting ideals,
see \cite{Eis}). We have
\[
\fit_{1}(\hh)\supseteq\fit_{2}(\hh)\supseteq\cdots
\supseteq\fit_{r}(\hh)\supsetneq\{0\}.
\]
These ideals have good functoriality properties.
Indeed, if $\apl{\sigma }{\tix}{X}$ is a bimeromorphic morphism, the
sequence
$\sutrgd{\sigma \ee\eee_{1}}{\sigma \ee\eee_{2}}{\sigma \ee\hh}$ is exact
and $\fit_{j}\bigl(\sigma \ee\hh\bigr)=\sigma \ee\fit_{j}(\hh)$ (by
$\sigma \ee\fit_{j}(\hh)$, we mean of course its image in $\oo_{\tix}$).
By the Hironaka theorem, we can suppose, after taking a finite number of
pullbacks under blowups with smooth centers, that all the Fitting ideals
$\fit_{k}(\ff)$ are ideal sheaves
associated with effective normal crossing divisors $D'_{k}$.
Now, take an element $x$ of $X$. Consider an exact sequence
\[
\xymatrix{
\oo_{U}^{r}\ar[r]_{M}&\oo_{U}^{r}\ar[r]&\hh_{\vert U}\he\ar[r]&0
}
\]
in a neighbourhood of $x$. The matrix $M$ is a $r\tim r$ matrix of holomorphic
functions on $U$.
Let $\bigl\{\phi _{1}=0\bigr\}$ be an equation of
$D_{1}$ around $x$. Then we  can write $M=\phi _{1}M_{1}$ and the
coefficients of $M_{1}$ generate $\oo_{x}\he$.
Thus, at least one of these coefficients
does not vanish at $x$. We can suppose that it is the upper-left one.
By Gauss elimination process, we can write
\[ M=\phi _{1}
\newcommand*{\tempa}{\multicolumn{1}{|c}{}}
\newcommand*{\tempb}{\multicolumn{1}{c|}{}}
\left(
\begin{array}{cccr}
1&0&\cdots &0\\ \cline{2-4}
0&\tempa&&\tempb\\
\vdots&\tempa&\, \raisebox{3pt}{$M_{2}$}&\tempb\\
0&\tempa&&\tempb\\
\cline{2-4}
\end{array}\
\right).
\]
Then, since $\fit_{k}(\ff)_{\vert U}\he=\fit_{k}(M)$, we get
$\fit_{2}(M)=\phi _{1}^{2}\fit_{1}(M_{2})$. Since $\fit_{2}(M)$ is
principal, so is $\fit_{1}(M_{2})$ and we write
$\fit_{1}(M_{2})=\bigl(\phi _{2}\bigr)$. Then, by the same argument as
above, we can write
\[ M=
\newcommand*{\tempa}{\multicolumn{1}{|c}{}}
\newcommand*{\tempb}{\multicolumn{1}{c|}{}}
\left(
\begin{array}{ccccr}
\phi _{1}&0&\cdots &\cdots &0\\
0&\phi _{1}\phi _{2}&0&\cdots &0\\ \cline{3-5}
\vdots&0&\tempa&&\tempb\\
\vdots&\vdots&\tempa&\hspace*{-7pt}\raisebox{5pt}{$M_{3}$}&\tempb\\
0&0&\tempa&&\tempb\\
\cline{3-5}
\end{array}\
\right).
\]

By this algorithm, we get
\[
M=
\begin{pmatrix}
\phi _{1}&0&\cdots &0\\
0&\phi _{1}\phi _{2}&&\\
\vdots&&\ddots&\\
0&&&\phi _{1}\cdots \phi _{r}
\end{pmatrix}
\]
\par\medskip
and then
$
\ff_{\vert U}\he\simeq\Bigl(\oo_{X}\he\mkern - 4 mu
\bigm/_{\ds\phi _{1}\oo_{X}\he}
\Bigr)_{\vert U}\he\oplus
\cdots\oplus\Bigl(\oo_{X}\he \mkern - 4 mu
\bigm/_{\ds\phi _{1}\dots \phi _{r}\oo_{X}\he}\Bigr)_{\vert
U}\he
$.
Thus, if $D_{1},\dots ,D_{r}$ are the divisors
of $\phi _{1},\phi _{1}\, \phi _{2},\dots ,
\phi _{1}\, \phi _{2}\dots \phi _{r}$, we have $D_{k}=D'_{k}-D'_{k-1}$,
which
shows that the divisors $D_{k}$ are intrinsically defined by $\ff$.
\end{proof}
From now on, we will say that a torsion sheaf $\hh$ is \textit{principal}
if it is
everywhere locally isomorphic to a fixed sheaf $\bop_{i=1}^{r}\oo_{X}\he/_{\ds
\ii_{D_{i}}\he}$ where the $D_{i}$ are
(non necessarily reduced) effective normal crossing divisors and
$D_{1}\leq D_{2}\leq \cdots \leq D_{r}$. We will
denote by $\nu (\hh)$ the number of branches of $D$, counted with their
multiplicities.
\begin{proposition}\label{LemmeDeuxWhitneyFormulaArt1}
It suffices to prove the Whitney formula when $\ff$ and $\g $ are
locally free sheaves and $\hh$ is the push-forward
of a locally free sheaf on a smooth hypersuface.
\end{proposition}
\begin{proof} We proceed in several steps.
\begin{lemma}\label{NouveauLemmeClChArt1}
Consider an exact sequence
$
\sutroiszerogd{\ff}{\g}{i_{Y*}\he\eee}
$
where $Y$ is a smooth hypersurface of $X$, $\g$ is a locally free sheaf
on $X$ and $\eee$ is a locally free sheaf on $Y$. Then $\ff$ is locally
free on $X$.
\end{lemma}
\begin{proof}
Let $m_{x}\he$ be the maximal ideal of the local ring $\oo_{x}\he$.
By Nakayama's lemma, it suffices to show that for every $x$ in $X$,
$\tor_{1}^{\oo_{x}\he}\bigl(\ff_{x}\he,\oo_{x}\he/{\ds
m_{x}\he}\bigr)=0$. Since $Y$ is a hypersurface, $i_{Y*}\he\eee$
admits a locally free resolution of length two. Thus
$\tor_{1}^{\oo_{x}\he}\bigl(\ff_{x}\he,\oo_{x}\he/{\ds
m_{x}\he}\bigr)\simeq\tor_{2}^{\oo_{x}\he}\bigl((i_{Y*}\he\eee)_{x}\he,
\oo_{x}\he/{\ds
m_{x}\he}\bigr)=0$.
\end{proof}
\begin{lemma}\label{EncoreUnLemUnWhitneyFormulaArt1}
It suffices to prove $(\emph{W}_{n})$ when $\ff$, $\g$ are locally
free sheaves and $\hh$ is principal.
\end{lemma}
\begin{proof}
By Proposition \ref{LemmeUnWhitneyFormulaArt1},
it is enough to prove the Whitney formula
when $\ff$, $\g$ are locally free sheaves and $\hh$ is a torsion sheaf. So we
suppose that $\ff$, $\g$ and $\hh$ verify these hypotheses.
By Proposition \ref{PropositionUnWhitneyFormulaArt1},
there exists a bimeromorphic morphism
$\apl{\sigma }{\ti{X}}{X}$ such that $\sigma \ee\hh$ is principal. We
have an exact sequence
\[
\xymatrix
{
0\ar[r]& \tore{1}{\hh}{\sigma }\ar[r]&{\sigma \ee\ff}\ar[r]&{\sigma \ee\g}\ar[r]&
{\sigma \ee\hh}\ar[r]&0.
}
\]
But $\tore{1}{\hh}{\sigma }$ is a torsion sheaf and $\sigma \ee\ff$ is
locally free, so we get an exact sequence
\[
\sutrgdpt{\sigma \ee\ff}{\sigma \ee\g}{\sigma \ee\hh}{.}
\]
and we have $\tore{i}{\hh}{\sigma }=0$ for $i\geq 1$.
By Proposition \ref{PropositionDeuxComplementQuatreUnArt1} (ii), we
obtain the equalities
$\ba\ch(\sigma \ee\ff)=\sigma \ee\ba\ch(\ff)$ and
$\ba\ch(\sigma \ee\g)=\sigma \ee\ba\ch(\g)$, and by Proposition
\ref{PropositionUnChClArt1} (iii) we get
$\ch(\sigma \ee\hh)=\sigma \ee\ch(\hh)$. Thus
\[
\sigma \ee\bigl(\ba\ch(\ff)+\ch(\hh)-\ba\ch(\g)\bigr)=\ba\ch
(\sigma \ee\ff)+\ch(\sigma \ee\hh)-\ba\ch(\sigma \ee\g).
\]
\end{proof}
\begin{lemma}\label{EncoreUnLemDeuxWhitneyFormulaArt1}
Suppose that $(\emph{W}_{n})$ holds
if $\ff$, $\g$ are locally free sheaves and $\hh$ is the push-forward of
a locally
free sheaf on a smooth hypersurface.
Then $(\emph{W}_{n})$ holds when $\ff$, $\g$ are
locally free sheaves and $\hh$ is principal.
\end{lemma}
\begin{proof}
We argue by induction on $\nu (\hh)$.
If $\nu (\hh)=0$, $\hh=0$ and $\ff\simeq\g$.
If $\nu (\hh)=1$, $\hh$ is
the push-forward of a locally free sheaf on a smooth hypersurface and
there is nothing to prove.
\par\medskip
In the general case,
let $Y$ be a branch of $D_{1}$. Since $Y\leq D_{i}$
for every $i$ with $1\leq i\leq
r$, we can see that $\eee=\hh_{\vert Y}\he$ is locally free
on $Y$. Besides,
if we define $\ti\hh$ by the exact sequence
\[
\sutrgdpt{\ti{\hh}}{\hh}{i_{Y*}\he\eee}{,}
\]
$\ti\hh$ is
everywhere locally isomorphic
to $\bop_{i=1}^{r}\oo_{X}/_{\ds
\ii_{D_{i}-Y}\he}$. Thus $\ti\hh$ is principal and $\nu (\ti\hh)=\nu (\hh)-1$.
We define $\ti\eee$ by the exact sequence:
\[
\sutrgdpt{\ti\eee}{\g}{i_{Y*}\he\eee}{.}
\]
By Lemma \ref{NouveauLemmeClChArt1}, $\ti\eee$ is locally free.
Furthermore, we have an exact sequence
\[
\sutrgdpt{\ff}{\ti\eee}{\ti\hh}{.}
\]
By induction, $\ba\ch\bigl(\ti\eee\bigr)=\ba\ch(\ff)+\ch\bigl(\ti\hh\bigr)$
and by our hypothesis
$\ch(\g)=\ba\ch\bigl(\ti\eee\bigr)+\ch\bigl(i_{Y*}\he\eee\bigr)$. Since
$\ti\hh$, $\hh$ and $i_{Y*}\he\eee$ are torsion sheaves,
$\ch(\hh)=\ch\bigl(\ti\hh\bigr)+\ch\bigl(i_{Y*}\eee\bigr)$ and we get
$\ba\ch(\g)=\ba\ch(\ff)+\ch(\hh)$.
\end{proof}
Putting the two lemmas together, we obtain Proposition
\ref{LemmeDeuxWhitneyFormulaArt1}.
\end{proof}
\subsection{Proof of the Whitney formula}\label{ProofWhitney}
We are now ready to prove Theorem \ref{EncoreUnTheoWhFoArt1}.
\par\medskip
In the sections \ref{SectionUnWhitneyFormulaArt1}
and \ref{SectionDeuxWhitneyFormulaArt1},
we have made successive reductions in order
to prove the Whitney formula in a tractable context,
so that we are reduced to the case where
$\ff$ and $\g$ are locally free sheaves and $\hh=
i_{Y*}\he\eee$, where $Y$ is a smooth hypersurface
of $X$ and $\eee$ is a locally free sheaf on $Y$.
Our working hypotheses will be these.
\par\medskip
Let us briefly explain the sketch of the argument. We consider
the sheaf $\ti\g$ on $X\times\P^{1}$ obtained by deformation of the second
extension class of the exact sequence
$\sutrgdpt{\ff}{\g}{\hh}{.}$ Then
$\ti\g_{\vert X\times\{0\}}\simeq \ff\oplus\hh$ and
$\ti\g_{\vert X\times\{t\}}\simeq\g$ for $t\neq 0$. It will turn out
that $\ti\g$ admits a locally free quotient
of maximal rank $\qq$ on the blowup of
$X\times \P^{1}$ along $Y\times \{0\}$, and the associated kernel $\nn$
will be the push-forward of a locally free
sheaf on the exceptional divisor $E$, say
$\nn=i_{E*}\LL$. Then we consider the class
$\alpha =\ba\ch(\qq)+i_{E*}\bigl(
\ba\ch(\LL)\td(N_{E/X})^{-1}
\bigr)$ on the blowup. After explicit computations, it will appear that
$\alpha $ is the pullback of a form $\beta $ on the base
$X\times\P^{1}$. By the \mbox{$\P^{1}$-homotopy} invariance of Deligne cohomology
(Proposition \ref{PropositionUnComplementQuatreUnArt1} (vi)),
$\beta _{\vert X\times\{t\}}$ does not depend on $t$. This will give the
desired result.
\par\medskip
Let us first introduce some notations.
The morphism $\xymatrix{\ff\ar[r]&\g}$ will be denoted by $\gamma $.
Let $s$ be a global section of $\oo_{\P^{1}}\he(1)$ which vanishes exactly
at $\{0\}$.
Let
$\apl{\pr_{1}}{X\tim\P^{1}}{X}$  be the projection on the first
factor. The relative $\oun$, namely $\oo_{X}\he\boxtimes\,  \oo
_{\P^{1}}\he(1)$, will still be denoted by $\oun$. We define a sheaf
$\ti\g$ on $X\tim\P^{1}$ by the exact sequence
\[
\xymatrix@C=7ex{0\ar[r]
&\pr_{1}\ee\ff\ar[r]_-{(\id\oti s,\, \gamma )}&\pr_{1}\ee
\ff(1)\oplus\pr_{1}\ee
\g\ar[r]&\ti\g\ar[r]&0.
}
\]
Remark that $\ti{\g}_{0}\simeq\ff\oplus\hh$ and
$\ti{\g}_{t}\simeq\g$ if $t\not = 0$.
\begin{lemma}\label{LemmeTroisWhitneyFormulaArt1}
There exist two exact sequences
\begin{align}
&\sutrgd{\pr\ee_{1}\ff(1)}{\ti{\g}}{\pr\ee_{1}\hh}\tag{i}
\label{EquationUnWhitneyFormulaArt1}\\
&\sutrgdpt{\ti{\g}}{\pr_{1}\ee\g(1)}{i_{X_{0}*}\he\hh}{.}\tag{ii}
\label{EquationDeuxWhitneyFormulaArt1}
\end{align}
\end{lemma}
\begin{remark}\label{LastRemarkArt1}
(i) implies that $\ti\g$ is flat over $\P^{1}$.
\end{remark}
\begin{proof}
(i) The morphism
$\xymatrix{
\pr\ee_{1}\ff(1)\oplus \pr\ee_{1}\g\ar@{->>}[r]&\pr\ee_{1}\g\ar@{->>}[r]&
\pr\ee_{1}\hh}$
induces a morphism
$\xymatrix{\ti{\g}\ar@{->>}[r]&\pr_{1}\ee\hh}$\!. If $\kk$ is
the kernel of this
morphism, we obtain an exact sequence
\[
\xymatrix@C=10ex{0\ar[r]
&\pr_{1}\ee\ff\ar[r]_-{(\id\oti s,\, \id)}&\pr_{1}\ee
\ff(1)\oplus\pr_{1}\ee
\ff\ar[r]&\kk\ar[r]&0.
}
\]
Thus $\kk=\pr\ee_{1}\ff(1)$.
\par\medskip
(ii) We consider the morphism
$
\xymatrix@R=1ex{
\pr\ee_{1}\ff(1)\oplus \pr\ee_{1}\g\ar@{->>}[r]&\pr\ee_{1}\g(1)\\
f+g\ar[r]&\gamma (f)-g\oti s.
}$
\par
It induces a morphism $\aplpt{\phi }{\ti\g}{\pr_{1}\ee\g(1)}.$
The last morphism of (ii) is just the composition
\[
\xymatrix{
\pr\ee_{1}\g(1)\ar@{->>}[r]&i_{X_{0}*}\he\g\ar@{->>}[r]&i_{X_{0}*}\he\hh\!.
}
\]
The cokernel of this morphism has support in $X\tim\{0\}$. Besides,
the action of $t$ on this cokernel is zero. The restriction
of $\phi $ to the fiber
$X_{0}=X\tim\{0\}$
is the morphism $\flgd{\ff\oplus\hh}{\g}$\!, thus
the
sequence $\sutrd{\ti\g}{\pr_{1}\ee\g(1)}{i_{X_{0}*}\he\hh}$ is exact.
The kernel of $\phi $, as its cokernel,
is an \mbox{$\oo_{X_{0}}\he$-module}.
Thus we can find $\zz$ such that
$\ker\phi =i_{X_{0}*}\he
\zz$. Since $X_{0}$ is a hypersurface of $X\tim\P^{1}$,
for every coherent sheaf $\LL$ on
$X\tim\P^{1}$, we have
$\tore{2}{\LL}{i_{X_{0}}\he}=0$.
Applying this to
$\LL=\ti{\g}/_{\ds i_{X_{0}*}^{\vphantom{-1}}\zz}$
and using Remark \ref{LastRemarkArt1}, we
get
\[
\tore{1}{i_{X_{0}*}\he\zz}{i_{X_{0}}\he}
\suq \tore{1}{\ti{\g}}{i_{X_{0}}\he}=\{0\}.
\]
But
$\tore{1}{i_{X_{0}*}\he\zz}{i_{X_{0}}\he}\simeq
\zz\oti N\ee_{X_{0}/X\tim\P^{1}}
\simeq \zz
$,
so $\zz=\{0\}$.
\end{proof}
Recall now that $\hh=i_{Y*}\he\eee$ where $Y$ is a smooth hypersurface
of $X$ and $\eee$ is a locally free sheaf on $Y$. We consider the space $M_{Y/X}\he$
of the deformation of the normal cone of $Y$ in $X$ (see \cite{Ful}).
Basically, $M_{Y/X}\he$ is the blowup of $X\tim\P^{1}$ along
$Y\tim\{0\}$.
Let $\apl{\sigma }{M_{Y/X}\he}{X\tim\P^{1}}$ be the canonical
morphism. Then $\sigma \ee {X_{0}}$ is a Cartier divisor in $M_{Y/X}\he$
with two simple branches:
$E=\P\bigl(N_{Y/X}\he\oplus\oo_{Y}\he\bigr)$ and
$D=\bl_{Y}X\simeq X$, which intersect at $\P\bigl(N_{Y/X}\he\bigr)\simeq Y$.
The projection of
the blowup
from $E$ to $Y\times\{0\}$ will be denoted by $q$, and the canonical isomorphism
from $D$ to $X\tim\{0\}$ will be denoted by $\mu $.
\par\medskip
We now show:
\begin{lemma}\label{LemmeQuatreWhitneyFormulaArt1}
The sheaf
$\sigma \ee\ti\g$ admits a locally free quotient with maximal rank  on
$M_{Y/X}\he$, and the associated kernel $\nn$ is the push-forward
of a locally free sheaf on
$E$. More explicitly, if $F$ is the excess conormal bundle of $q$, $\nn
=i_{E*}\he\bigl(q\ee
\eee\oti F\bigr)$.
\end{lemma}
\begin{proof}
We start from the exact sequence
$\sutrgdpt{\ti{\g}}{\pr_{1}\ee\g(1)}{i_{X_{0}*}\he\hh}{.}$ We define $\qq$ by
the exact sequence $\sutrgdpt{\qq}{\sigma \ee\pr_{1}\ee\g(1)}
{\sigma \ee i_{X_{0}*}\he\hh}{.}$ Since
${\sigma \ee i_{X_{0}*}\he\hh}$ is
the push-forward of a locally free sheaf on $E$,
by Lemma \ref{NouveauLemmeClChArt1}, the sheaf
$\qq$ is locally free on $M_{Y/X}\he$. The sequence
\[
\sutrgd{\tore{1}{i_{X_{0}*}\he\hh}
{\sigma }}{\sigma \ee\ti\g}{\qq}
\]
is exact. The first sheaf being a torsion sheaf, $\qq$ is a locally free
quotient of $\ti\g$ with maximal rank. Besides, using the notations
given in the following diagram
\[
\xymatrix@C=8ex{
E\, \, \ar@{^{(}->}[r]^-{i_{E}\he}\ar[d]_{q}&M_{Y/X}\he\ar[d]^{\sigma }\\
Y\tim \{0\}\, \ar@{^{(}->}[r]_-{i_{Y\tim\{0\}}\he}&X\tim\P^{1}
}
\]
we have $\tore{1}{i_{X_{0}*}\he\hh}
{\sigma }=i_{E*}\he\bigl(q\ee
\eee\oti F\bigr)$ where $F$ is the excess conormal bundle
of $q$ (see \cite[\S\,15]{BoSe}. Be aware of the fact that
what we note $F$ is  $F\ee$ in \cite{BoSe}). This finishes the proof.
\end{proof}
We consider now the exact sequence
$\sutrgd{\nn}{\sigma \ee\ti\g}{\qq}$ where
$\qq$ is locally free on $M_{Y/X}\he$
and $\nn=i_{E*}\he\bigl(q\ee\eee\oti F\bigr)=i_{E*}\he\LL$.
We would like to introduce the form
$\ch(\sigma \ee\ti\g)$, but it is not defined since
$M_{Y/X}$ is of dimension $n+1$. However, $\sigma \ee\ti\g$ fits in a
short exact sequence where the Chern classes of the two other sheaves
can be defined. Remark that we need Lemma
\ref{LemmeQuatreWhitneyFormulaArt1}
to perform this trick, since it cannot be done on $X\times\P^{1}$.
\begin{lemma}\label{LemmeSixWhitneyFormulaArt1}
Let $\alpha $ be the Deligne class on $M_{Y/X}\he$ defined by
$\alpha =\ba\ch(\qq)+i_{E*}\he\Bigl(
\ba\ch(\LL)\td\bigl(N_{E/M_{Y/X}\he}\he\bigr)^{-1}\Bigr)$.
\begin{enumerate}
  \item [(i)] The class $\alpha $ is the pullback of a
  Deligne class on $X\tim\P^{1}$.

  \item [(ii)]
  We have $i\ee_{D}\alpha ={\mu }\ee\be\ch^{X_{0}}\be\bigl(
  \ti{\g}_{0}\bigr)$.
\end{enumerate}
\end{lemma}
\begin{proof}
We compute explicitly $\ie\ee\alpha $.
\begin{align*}
i\ee_{E}\, i_{E*}\he\, \Bigl(
\ba\ch(\LL)\td\bigl(N_{E/M_{Y/X}\he}\he\bigr)^{-1}\Bigr)
&=\ba{\ch}(\LL)\, \, \td\bigl(N_{E/M_{Y/X}\he}\he\bigr)^{-1}\, \,
c_{1}\bigl(N_{E/M_{Y/X}\he}\he\bigr)&&\textrm{by Proposition
\ref{PropositionDeuxInsertC} (vii)}\\
&=\ba{\ch}(\LL)\biggl(
1-e^{-\ds c_{1}\bigl(N_{E/M_{Y/X}\he}\bigr)}\biggr)\\
&=\ba{\ch}(\LL)-\ba{\ch}
\Bigl(\LL\oti N\ee_{E/M_{Y/X}\he}\Bigr)
&&\textrm{by Proposition \ref{PropositionDeuxComplementQuatreUnArt1} (iii)}\\
&=\ba{\ch}\bigl(\ie\ee\nn\bigr)-\ba{\ch}
\Bigl(\LL\oti N\ee_{E/M_{Y/X}\he}\Bigr).
\end{align*}
From the exact sequence $\sutrgdpt{\nn}
{\sigma \ee\ti\g}{\qq}{,}$ we get the exact sequence of locally free
sheaves on $E$:
$\sutrgdpt{i_{E}\ee\nn}{\ie\ee\, \, \sigma \ee\ti{\g}}{\ie\ee\qq}{.}$
Since $\ie\ee\, \, \ba{\ch}(\qq)=\bh\bigl(\ie\ee\qq\bigr)$, we obtain
\begin{align*}
\ie\ee\alpha
&=\ba\ch\bigl(i_{E}\ee\qq\bigr)+\ba\ch\bigl(i_{E}\ee\nn\bigr)
-\ba\ch\bigl(\LL\oti N\ee_{E/M_{Y/X}\he}\bigr)\\
&=\bh\bigl(\ie\ee\sigma \ee\ti\g\bigr)-\bh\bigl(
\LL\oti N\ee_{E/M_{Y/X}\he}
\bigr)&&\textrm{by Proposition \ref{PropositionDeuxComplementQuatreUnArt1}
(i)}\\
&=\bh\bigl(q\ee i_{Y}\ee\ff\bigr)+\bh\bigl(q\ee i_{Y}\ee\hh\bigr)-
\bh\bigl(\LL\oti N\ee_{E/M_{Y/X}\he}\bigr)\\
&=q\ee\ba\ch\bigl(i_{Y}\ee\ff\bigr)+q\ee\ba\ch(\eee)-
\bh\bigl(q\ee\eee\oti F\oti N\ee_{E/M_{Y/X}\he}\bigr)&&
\textrm{by Proposition \ref{PropositionDeuxComplementQuatreUnArt1} (ii).}
\end{align*}
Recall that the conormal excess bundle $F$ is
the line bundle
defined by the exact
sequence
\[
\sutrgdpt{F}{q\ee N\ee_{Y/X\tim\P^{1}}}{N\ee_{E/M_{Y/X}}}{.}
\]
Thus, we have $\det\bigl(q\ee N\ee_{Y/X\tim\P^{1}}\bigr)=
F \oti N\ee_{E/M_{Y/X}\he}$. Since $\det\bigl(q\ee N\ee_{Y/X\tim\P^{1}}\bigr)
=q\ee \det\bigl(N\ee_{Y/X\tim\P^{1}}\bigr)$, we get
by Proposition \ref{PropositionDeuxComplementQuatreUnArt1} (ii)
again
\[
\ie\ee\alpha =q\ee\bigl[\, \ba{\ch}\bigl(
i_{Y}\ee\ff\bigr)+\ba{\ch}(\eee)-\ba{\ch}\bigl(\eee\oti
\det\bigl(N\ee_{Y/X\tim\P^{1}}\bigr)\bigr)\bigr].
\]
This proves (i).
\par\medskip
(ii) The divisors $E$ and $D$ meet transversally. Then, by
Proposition \ref{PropositionDeuxInsertC} (iv),
\begin{align*}
i_{D}^{\, *}\alpha &=i_{D}^{\, *}\ba{\ch}(\qq)+i_{D}^{\, *}
i_{E*}\he\Bigl(\ba{\ch}(\LL)\td\bigl(N_{E/M_{Y/X}\he}\he\bigr)^{-1}\Bigr)\\
&=\ba{\ch}\bigl(i_{D}^{\, *}\qq\bigr)+
i_{\flcourtegd{\st E\cap D}{\st D*}}\he\, \,
i_{\flcourtegd{\st E\cap D}{\st E}}^{\, *}
\Bigl(\ba{\ch}(\LL)\td\bigl(N_{E/M_{Y/X}\he}\he\bigr)^{-1}\Bigr)\\
&=\ba{\ch}\bigl(i_{D}^{\, *}\qq\bigr)+
i_{\flcourtegd{\st E\cap D}{\st D*}}\he\Bigl(
\ba{\ch}\bigl(i_{\flcourtegd{\st E\cap D}{\st E}}^{\, *}\LL\bigr)
\td\bigl(N_{E\cap D/D}\he\bigr)^{-1}
\Bigr).
\end{align*}
We remark now that $i_{\flcourtegd{\st E\cap D}{\st E}}^{\, *}\LL
=i_{D\cap E}^{\, *}\, \, \nn$. Since $\dim(E\cap D)=n-1$,
we obtain
\[
i_{D}^{\, *}\alpha =\ba{\ch}\bigl(i_{D}^{\,
*}\qq\bigr)+\ch^{D}\be
\bigl(i_{\flcourtegd{\st E\cap D}{\st D*}}\he i_{E\cap D}^{\, *}
\nn\bigr)=\ba{\ch}\be\bigl(i_{D}^{\,
*}\qq\bigr)+\ch^{D}\be \bigl(i_{D}\ee\nn\bigr).
\]
We have the exact sequence on $D$:
$\sutrgdpt{i_{D}\ee\nn}{i_{D}\ee\, \, \sigma \ee\ti\g}{i_{D}\ee\qq}{.}$
Therefore
$i_{D}\ee\, \, \sigma \ee\ti\g$ admits a locally free quotient of maximal rank
and, $\mu $ being an isomorphism,
\[
\ba\ch\bigl(i_{D}\ee\qq\bigr)+\ch^{D}\be(i_{D}\ee\nn)=
\ch^{D}\be\bigl(i_{D}\ee\, \, \sigma \ee\ti\g
\bigr)=\ch^{D}\be\bigl({\mu }\ee\ti{\g}_{0}\bigr)={\mu }\ee\ch
^{X_{0}}\be\bigl(\ti{\g}_{0}\bigr).
\]
\end{proof}
We are now ready to use the homotopy property for Deligne cohomology
(Proposition \ref{PropositionUnComplementQuatreUnArt1} (vi)).
\par\medskip
Let $\alpha $ be the form defined in Lemma \ref{LemmeSixWhitneyFormulaArt1}.
Using (i) of this lemma and Proposition \ref{PropositionDeuxInsertC} (vi), we
can write  $\alpha =\sigma \ee\beta $. Thus
$i_{D}\ee\, \alpha =i_{D}\ee\, \sigma \ee\beta ={\mu }\ee \be i_{X_{0}}\ee\beta $.
By (ii) of the same lemma,
$i_{D}\ee\alpha ={\mu }\ee\be\ch^{X_{0}}\be\bigl(\ti{\g}_{0}\bigr)$
and we get $i_{X_{0}}\ee\beta =\ch^{X_{0}}\be\bigl(\ti{\g}_{0}\bigr)$.
If $t\in\P^{1}
\backslash\{0\}$, we have clearly $\beta _{\vert X_{t}}=\ba\ch(\g)$.
Since
$\beta _{\vert X_{t}}\he=\beta _{\vert X_{0}}\he $
we obtain
\[
\ba\ch(\g)=\ch^{X_{0}}\be\bigl(\ti{\g}_{0}\bigr)
=\ba\ch(\ff)+\ch(\hh).
\]
\begin{flushright}
$\Box$
\end{flushright}
We can now establish the remaining induction properties.
\begin{theorem}\label{PropositionDeuxWhitneyFormulaArt1}
The following assertions are valid:
\begin{enumerate}

  \item [(i)] Property $(\emph{F}_{n})$ holds.

  \item [(ii)] Property $(\emph{P}_{n})$ holds.
\end{enumerate}
\end{theorem}
\begin{proof}
(i) We take $y=[\ff]$. Let us first suppose that $f$ is a bimeromorphic
map. Then there exists a bimeromorphic map $\apl{\sigma }{\ti{X}}{X}$
such that $(f\circ \sigma )\ee \ff$ admits a locally free quotient
of maximal rank. Then by Theorem \ref{LemmeHuitChClArt1} (i),
\[
\sigma \ee\ch\bigl(f\pe\be[\ff]\bigr)=\ch\bigl(\sigma \pe\be\, f\pe\be
[\ff]\bigr)=(f\circ\sigma )\ee\be\ch\ff=\sigma
\ee\be\bigl[f\ee\be\ch(\ff)\bigr].
\]
Suppose now that $f$ is a surjective map.
Then there exist two bimeromorphic maps
$\aplpt{\pi _{X}\he}{\tix}{X}{,}$ $\apl{\pi _{Y}\he}{\tiy}{Y}$ and a
surjective map
$\apl{\ti{f}}{\ti{X}}{\ti{Y}}$ such that:
\begin{enumerate}
  \item [--] the diagram
  $
\xymatrix{
\tix\ar[r]^{\ti{f}}\ar[d]_{\pi _{X}\he}&\tiy\ar[d]_{\pi _{Y}\he}\\
X\ar[r]_{f}&Y
}
  $ \quad is commutative.
  \item [--] the sheaf $\pi _{Y}\ee\ff$ admits a locally free quotient
  $\eee$ of maximal
  rank.
\end{enumerate}
We can write $\pi _{Y}\pe[\ff]=[\eee]+\ti{y}$\quad in $G(\ti{Y})$,
where $\ti{y}$ is in the image of the natural map $\aplpt{\iota }
{G_{\textrm{tors}}\he(\ti{Y})}{G(\ti{Y})}{.}$
The functoriality property being known for bimeromorphic maps, it
holds for $\pi _{X}\he$ and $\pi _{Y}\he$. The result is now a consequence of
Proposition \ref{PropositionDeuxComplementQuatreUnArt1} (ii) and
Proposition \ref{PropositionUnChClArt1} (iii).
\par\medskip
In the general case, we consider the diagram used in the proof of
Proposition \ref{PropositionUnChClArt1} (iii)
\[
\xymatrix@C=8ex{
\tix\ar[r]^{\ti{f}}\ar[d]_{\pi _{X}\he}&W\ar[r]^{i_{W}\he}\ar[d]^{\tau
}&\tiy\ar[d]^{\pi _{Y}\he}\\
X\ar[r]_-{f}&f(X)\ar[r]&Y
}
\]
where $\ti{f}$ is surjective. Then the functoriality property holds for
$\ti{f}$ by the argument above and for $i_{W}\he$ by Theorem
\ref{LemmeHuitChClArt1} (ii). This finishes the proof.

\par\medskip
(ii) We can suppose that $x=[\ff]$, $y=[\g]$ and that
$\ff$ and $\g$ admit locally free quotients $\eee_{1}$, $\eee_{2}$ of
maximal rank. Let $\ttt_{1}$ and $\ttt_{2}$ be the associated kernels.
We can also suppose that $\textrm{supp}(\ttt_{1})$ lies in a simple
normal crossing divisor.
Then
\begin{align*}
\ch([\ff].[\g])&=\ba\ch([\eee_{1}\he].[\eee_{2}\he])+\ch([\eee_{1}\he].
[\ttt_{2}])+\ch([\eee_{2}\he].[\ttt_{1}])+\ch([\ttt_{1}].
[\ttt_{2}])\\
&=\ba\ch(\eee_{1}\he)\ba\ch(\eee_{2}\he)+\ba\ch(\eee_{1}\he)\ch(\ttt_{2})
+\ba\ch(\eee_{2}\he)\ch(\ttt_{1})+\ch([\ttt_{1}].[\ttt_{2}])
\end{align*}
by $(\textrm{W}_{n})$, Proposition \ref{PropositionUnChClArt1} (ii) and
Proposition \ref{PropositionDeuxComplementQuatreUnArt1} (iii).
By d\'{e}vissage, we can suppose that $\ttt_{1}$ is a
\mbox{$\oo_{Z}\he$-module}, where $Z$ is a smooth hypersurface of $X$.
We write
$[\ttt_{1}]=i_{Z!}\he u$\quad  and $[\ttt_{2}]=v$.
Then
$[\ttt_{1}]\, .\, [\ttt_{2}]=i_{Z!}\he\bigl(u\, .\, i_{Z}\pe v\bigr)$. So
\begin{align*}
\ch\bigl([\ttt_{1}].[\ttt_{2}]\bigr)&=
i_{Z*}\he\Bigl(\ch^{Z}(u\, .\, i_{Z}\pe v)\, \,
\td\bigl(N_{Z/X}\bigr)^{-1}\Bigr)\\
&=i_{Z*}\he\Bigl(\ch^{Z}(u)\, \,
i_{Z}\ee\ch(v)\,\,\td\bigl(N_{Z/X}\bigr)^{-1}\Bigr)
&&\textrm{by $(\textrm{P}_{n-1})$ and
Proposition \ref{LemmeQuatreChClArt1} (ii)}\\
&=i_{Z*}\he\Bigl(\ch^{Z}(u)\,\,
\td\bigl(N_{Z/X}\bigr)^{-1}\Bigr)\ch(v)&&\textrm{by the projection formula}\\
&=\ch\bigl(i_{Z!}\he u\bigr)\ch(v)=\ch\bigl([\ttt_{1}]\bigr)
\ch\bigl([\ttt_{2}]\bigr)
\end{align*}
\end{proof}
The proof of Theorem \ref{TheoremTroisTrois} is
now concluded with the exception of property $(\textrm{RR}_{n})$.
\section{The Grothendieck-Riemann-Roch theorem for
projective morphisms}
\label{GRRForAnImmersionClassesArt1}
\subsection{Proof of the GRR formula}
We have already obtained the (GRR) formula for
the immersion of a smooth
divisor. We reduce now the general case to the divisor case by  a
blowup. This construction is classical (see \cite{BoSe}).
\begin{theorem}\label{SansRef}
Let $Y$ be a smooth submanifold of $X$. Then, for all
$y$ in $G(Y)$, we have
\[
\ch\bigl(i_{Y!}\he y\bigr)=i_{Y*}\he
\bigl(\ch(y)\td(N_{Z/X}\he)^{-1}\bigr).
\]
\end{theorem}
\begin{proof}
We blow up $Y$ along $X$ as shown below, where $E$ is the exceptional
divisor.
\[
\xymatrix{
E\ar[r]^{i_{E}\he}\ar[d]_{q}&\tix\ar[d]^{p}\\
Y\ar[r]_{i_{Y}\he}&X
}
\]
Let $F$ be the excess conormal bundle of $q$ defined
by the exact sequence
\[
\sutrgdpt{F}{q\ee N\ee_{Y/X}}{N\ee_{E/\tix}}{.}
\]
If $d$ is the codimension of $Y$ in $X$, then $\textrm{rank}(F)=d-1$.
Recall the following formulae:
\begin{enumerate}
  \item [(a)] Excess formula in \mbox{$K$-theory}
  (Proposition \ref{CorollaireUnAppendixArt1} (ii)): for all $y$ in
  $\kan(Y)$,
  $p\pe\, i_{!}\he y=i_{E!}\he\bigl(q\pe y\, .\, \lambda_{-1}\he
  F\bigr)$.

  \item [(b)] Excess formula in Deligne cohomology
  (Proposition \ref{PropositionDeuxInsertC} (vi)):
  for each Deligne class $\beta $ on $Y$,
  \[
p\ee\, i_{Y*}\he\, \beta =i_{E*}\he
  \bigl(q\ee\beta \, \, c_{d-1}(F\ee)\bigr).
\]
 \item [(c)] If $G$ is a vector bundle of rank $r$, then
  $\ch\bigl(\lambda_{-1}\he [G]\bigr)=c_{r}(G\ee\be)\td(G\ee\be)^{-1}$
  (\cite[Lemme 18]{BoSe}).
\end{enumerate}
We compute now
\begin{align*}
p\ee\ch\bigl(i_{Y!}\he y\bigr)&=\ch\bigl(p\pe\, i_{Y!}\he y\bigr)
=\ch\bigl(i_{E!}\he(q\pe y\, .\, \lambda _{-1}\he [F])\bigr)
&&\textrm{by $(\textrm{F}_{n})$ and (a)}\\
&=i_{E*}\he\Bigl(\ch^{E}
\bigl(q\pe y\, .\, \lambda _{-1}\he[F] \bigr)\td\bigl(N_{E/\tix}\bigr)
^{-1}\Bigr)&&\textrm{by (GRR) for $i_{E}\he$}\\
&=i_{E*}\he\Bigl(q\ee\ch^{Y}(y)\ch\bigl(\lambda _{-1}\he [F]\bigr)\, \,
q\ee\td\bigl(N_{Y/X}\he\bigr)^{-1}\td\bigl(F\ee\bigr)\Bigr)
&&\textrm{by $(\textrm{F}_{n})$, $(\textrm{P}_{n})$ and Proposition
\ref{PropositionDeuxComplementQuatreUnArt1} (i)}\\
&=i_{E*}\he
\Bigl(q\ee\bigl(\ch^{Y}(y)\td\bigl(N_{Y/X}\he\bigr)^{-1}\bigr)\,
c_{d-1}(F\ee)\Bigr)&&\textrm{by (c)}\\
&=p\ee\, i_{Y*}\he\Bigl(\ch^{Y}(y)\td\bigl(N_{Y/X}\he\bigr)^{-1}\Bigr)&&
\textrm{by (b).}
\end{align*}
Thus $\ch\bigl(i_{Y!}\he y\bigr)=i_{Y*}\he\bigl(\ch^{Y}(y)
\td\bigl(N_{Y/X}\he\bigr)^{-1}\bigr)$.
\end{proof}
Now we can prove a more general Grothendieck-Riemann-Roch
theorem:
\begin{theorem}\label{ThComplementaire}
The GRR theorem holds in rational Deligne cohomology for
projective morphisms between smooth complex compact manifolds.
\end{theorem}
\begin{proof}
Let $\apl{f}{X}{Y}$ be a projective morphism. Then we can write $f$
as the composition of an immersion $\apl{i}{X}{Y\times\P^{N}}$
and the second projection $\apl{p}{Y\times\P^{N}}{Y}$. By
Theorem \ref{SansRef}, GRR is true for $i$. Now the arguments in
\cite{Bei} show that the canonical map from $G(Y)\oti_{\Z}\he
G(\P^{N})$ to $G(Y\times\P^{N})$
is surjective. Therefore, it is enough to prove GRR for $p$ with
elements of the form $y\, .\, w$, where $y$ belongs to $G(Y)$
and $w$ belongs to $G(\P^{N})$. By the product formula for the
Chern character, we are led to the Hirzebruch-Riemann-Roch
formula for $\P^{N}$, which is well known.
\end{proof}
\subsection{Compatibility of Chern classes and the GRR formula}
We will show that the GRR formula for immersions combined with
some basic properties can be sufficient to characterize
completely a theory of Chern classes. This will give various
compatibility theorems.
\par\medskip
We assume to be given for each smooth complex compact manifold $X$
a graded commutative cohomology ring $A(X)=\oplus_{i=0}^{\dim X}
A^{i}(X)$ which is an algebra over $\Q$,
$\Q\subset A^{0}\be(X)$, with the following properties:
\begin{enumerate}\label{unicity}
  \item [$(\alpha)$] For each holomorphic map
  $\apl{f}{X}{Y}$, there is a pull-back morphism
  $\apl{f\ee}{A(Y)}{A(X)}$ which is functorial and compatible
  with the products and the gradings.

  \item [$(\beta) $] If $\sigma $ is the blowup of a smooth
  complex compact manifold along a smooth submanifold, then
  $\sigma \ee$ is injective.

  \item [$(\gamma) $] If $E$ is a holomorphic vector bundle
  on $X$ and  $\apl{\pi }{\P(E)}{X}$
  is the projection of the associated projective bundle, then
  $\pi \ee\be$ is injective.

  \item [$(\delta) $] If $X$ is a smooth complex compact manifold and $Y$
  is a smooth submanifold of codimension $d$, then there is a
  Gysin morphism $\apl{i_{*}\he}{A\ee\be(Y)}{A^{*+d}\be(X)}$.
\end{enumerate}
The main examples are:
\[
A^{i}(X)=H^{2i}_{D}(X,\Q(i)),\
\HHH^{2i}(X,\Omega _{X}^{  {\ds\bullet}\geqslant i}),\
H^{i}(X,\Omega ^{i}_{X}),\  F^{i}H^{2i}(X,\C),
\ H^{2i}(X,\Q)\ \textrm{and}\   H^{2i}(X,\C).
\]
Then we have the following theorem:
\begin{theorem}\label{ThPartSixTrois}
Suppose that we have two theories of Chern classes $c_{i}$ and
$c'_{i}$ for coherent sheaves on any smooth complex compact
manifold $X$ with values in $A^{i}(X)$ such that $c_{0}
=c'_{0}=1$ and
\begin{enumerate}
  \item [(i)] The Whitney formula holds for the total classes
  $c$ and $c'$.

  \item [(ii)] The functoriality formula holds for $c$ and $c'$.

  \item [(iii)] If $L$ is a holomorphic line bundle,
  then
  \[
  c(L)=1+c_{1}(L)=c'(L)=1+c'_{1}(L).
  \]

  \item [(iv)] In both theories,
  the GRR theorem holds for immersions.
\end{enumerate}
Then for every coherent sheaf $\ff$, $c(\ff)$ and $c'(\ff)$ are
equal.
\end{theorem}
\begin{remark}
\begin{enumerate}
  \item [1.] The same conclusion holds
  for cohomology algebras over $\Z$ if we assume GRR
  \emph{without denominators}.

  \item [2.] If $X$ is projective, (i) and (iii) are sufficient
  to imply the equality of $c$ and $c'$ because of the existence
  of global locally free resolutions.

  \item [3.] In (iv), the Todd classes of the normal bundle are
  defined in $A(X)$ since $A(X)$ is a \mbox{$\Q$-algebra}.
\end{enumerate}
\end{remark}
\begin{proof}
We start by proving that for any holomorphic vector bundle
$E$,
$c(E)$ and $c'(E)$ are equal. Actually, this proof is a
variant of the splitting principle. We argue by induction on the
rank of $E$. Let $\apl{\pi }{\P(E)}{X}$ be the projective bundle
of $E$. Then we have the exact sequence
\[
\sutrgd{\oo_{E}\he(-1)}{\pi \ee\be E}{F}
\]
on $\P(E)$, where $F$ is a holomorphic vector bundle on $\P(E)$
whose rank is the rank of $E$ minus one. By induction,
$c(F)=c'(F)$ and by (iii), $c(\oo_{E}(-1))=c'(\oo_{E}(-1))$.
By (i), $c(\pi \ee\be E)=c'(\pi \ee\be E)$ and by (ii),
$\pi \ee\be[c(E)-c'(E)]=0$. By $(\gamma )$, $c(E)=c'(E)$.
\par\medskip
We can now prove Theorem \ref{ThPartSixTrois}.
As usual, we deal with exponential Chern classes. The proof
proceeds by induction on the dimension of the base manifold $X$.
\par\medskip
Let $\ff$ be a coherent sheaf on $X$. By Theorem
\ref{TheoremeUnChClArt1}
there exists a bimeromorphic morphism
$\apl{\sigma }{\ti{X}}{X}$
which is a finite composition of
blowups with smooth centers and a locally free sheaf $\eee$ on
$\ti{X}$ which is a quotient of maximal rank of $\sigma \ee\be
\ff$. Furthermore, by Hironaka's theorem, we can suppose that
the exceptional locus of $\sigma $ and the kernel of the
morphism
$\xymatrix{\sigma \ee\be\ff\ar[r]&\eee}$ are both contained in a
simple normal crossing divisor $D$ of $\ti{X}$. Thus
\[
\sigma \pe\be[\ff]=\sum _{i=0}^{n}(-1)^{i}\bigl[
\tore{i}{\ff}{\sigma }\bigr]=[\eee]+\sum _{i=1}^{n}
(-1)^{i}\bigl[\tore{i}{\ff}{\sigma }\bigr]
\]
and then $\sigma \pe\be[\ff]$ belongs to $[\eee]+G_{D}\he
(\ti{X})$. Now there is a surjective morphism
\[
\xymatrix{\bigoplus\limits_{i=1}^{N}G_{D_{i}}\he(\ti{X})\ar[r]&
G_{{D}}(\ti{X})\he}.
\]
Moreover, by Proposition \ref{LemmeUnAppendixArt1},
$G(D_{i})$ is isomorphic to $G_{D_{i}}\he(\ti{X})$.
Remark that $\td \bigl(N_{D_{i}/\ti{X}}\he\bigr)
=\td'\bigl(N_{D_{i}/\ti{X}}\he\bigr)$.
By the
GRR formulae (iv) and the induction hypothesis,
$\ch$ and $\ch'$ are
equal on each $G_{D_{i}}\he(\ti{X})$. By the first part of
the proof,
$\ch(\eee)=\ch'(\eee)$. Thus $\ch\bigl(\sigma \pe\be[\ff]\bigr)=
\ch'\bigl(\sigma \pe\be[\ff]\bigr)$. By (ii),
$\sigma \ee\be\bigl[\ch(\ff)-\ch'(\ff)\bigr]=0$.
Since $\sigma \ee\be$ is injective by $(\beta )$,
$\ch(\ff)=\ch'(\ff)$.
\end{proof}
\begin{corollary}\label{DernierCorollaire}
Let $\ff$ be a coherent analytic sheaf on $X$. Then:
\begin{enumerate}
  \item [(i)] The classes $c_{i}(\ff)$ in $H^{2i}_{D}
  (X,\Q(i))$ and $c_{i}(\ff)^{\textrm{top}}$ in
  $H^{2i}(X,\Z)$ have the same image in $H^{2i}(X,\Q)$.

  \item [(ii)] The image of $c_{i}(\ff)$ via the natural
  morphism from $H^{2i}_{D}(X,\Q(i))$ to $H^{i}(X,
  \Omega _{X}^{i})$ is the $i\textrm{th}$ Atiyah Chern class of
  $\ff$
\end{enumerate}
\end{corollary}
\begin{proof}
If $L$ is a holomorphic line bundle, then (i) and (ii) hold for
$L$. Indeed, using the isomorphism between $\textrm{Pic}(X)$ and
$H^{2}_{D}(X,\Z(1))$, the topological and Atiyah first Chern
classes are obtained by the two morphisms of complexes
\[
\xymatrix{
\Z\ar[r]^{2i\pi }\ar[d]&\oo\\
\Z&
}\qquad\textrm{and}\qquad
\xymatrix{
\Z\ar[r]^{2i\pi }&\oo\ar[d]^{d/2i\pi }\\
&\Omega ^{1}_{X}
}
\]
On the other hand,
GRR for immersions holds for topological Chern classes by
\cite{AtHi} and for Atiyah Chern classes by \cite{OBToTo3}.
\end{proof}
\begin{remark}
If $X$ is a K\"{a}hler complex manifold, the Green Chern classes
are the same as the Atiyah Chern classes and the
complex topological
Chern classes. If $X$ is non K\"{a}hler, GRR does not seem to be
known for the Green Chern classes, except for a constant
morphism (see \cite{ToTo}). If this were true
for immersions, it would imply
the compatibility of $c_{i}(\ff)$ and
$c_{i}(\ff)^{\textrm{Gr}}$, via the map
from $H^{2i}_{D}(X,\Q(i))$ to
$\HHH^{2i}(X,\Omega _{X}^{  {\ds\bullet}\geqslant i})$.
On the other hand, if this compatibility holds,
it implies GRR for
immersions for the Green Chern classes.
\end{remark}
\section{Appendix.
Analytic \mbox{$K$-theory} with support}\label{AppendixArt1}
Our aim in this appendix is to answer the following question:
if $\apl{f}{X}{Y}$ is a morphism, $\ff$ a torsion sheaf on $Y$ with
support $W$, and $Z=f^{-1}(W)$, can we express
the sheaves $\tore{i}{\ff}{f}$ in terms of the sheaves
$\tore{i}{\ff_{\vert W}}{f_{\vert Z}}$?
\subsection{Definition of the analytic \mbox{$K$-theory} with support}
Let $X$ be a smooth compact complex manifold and $Z$ be an analytic subset
of $X$. We will denote by $\coh_{Z}\he(X)$ the abelian category of coherent
sheaves  on $X$ with support in $Z$. Then, $X$ being compact,
\[
\coh_{Z}\he(X)=\bigl\{\ff,\ \ff\in \coh(X)\ \textrm{such that\ }\ii_{Z}^{n}
\ff=0\ \textrm{for\ }n\gg 0 \bigr\}.
\]
\begin{definition}\label{DefinitionUnAppendixArt1}
We define $\kan_{Z}\he(X)$ as the quotient of $\Z\bigl[\coh_{Z}\he(X)\bigr]$
by the relations: $\ff+\hh=\g$ if there exists an exact sequence
$\sutrgd{\ff}{\g}{\hh}$ with $\ff$, $\g$ and $\hh$ elements of
$\coh_{Z}\he(X)$. In other words, $\kan_{Z}\he(X)$ is the Grothendieck group
of $\coh_{Z}\he(X)$. The image of an element $\ff$ of $\coh_{Z}\he(X)$ in
$\kan_{Z}\he(X)$ will be denoted by $[\ff]$.
\end{definition}
We first prove
\begin{proposition}\label{LemmeUnAppendixArt1}
The map $\apl{i_{Z*}}{\kan(Z)}{\kan_{Z}\he(X)}$ is a group isomorphism.
\end{proposition}
\begin{proof}
We consider the inclusion $\coh(Z)\suq\coh_{Z}\he(X)$. If $\ff$ is in
$\coh_{Z}\he(X)$, we can filter $\ff$ by the formula
$F^{i}\ff=\ii_{Z}^{i}\ff$. This is a finite filtration and the
associated quotients $F^{i}\ff/_{\ds F^{i+1}\ff}$ are in
$\coh(Z)$. Therefore, we can use the d\'{e}vissage theorem for the
Grothendieck group, which is in fact valid for all \mbox{$K$-theory}
groups (see \cite[\S\,5, Theorem 4]{Qui}).
\end{proof}

\subsection{Product on the \mbox{$K$-theory} with support}
\label{SousSectionDeuxAppendixArt1}
Let $Z$ be a \emph{smooth} submanifold of $X$. For any $x$ in
$\kan\iz\he(X)$, by Proposition \ref{LemmeUnAppendixArt1}, there exists a
unique $\ba{x}$ in $\kan(Z)$ such that $i_{Z*}\he\ba{x}=x$.
\begin{definition}\label{DefinitionDeuxAppendixArt1}
We define a product $\xymatrix{
\kan(Z)\oti_{\Z}\he\kan\iz\he(X)\ar[r]_-{\pz{Z}}&\kan\iz\he(X)
}$
by $x\pz{Z}  y=i_{Z*}\he(x.\, \ba{y})$.
\end{definition}
In other words, the following diagram is commutative:
\[
\xymatrix@C=8ex{
\kan(Z)\oti_{\Z}\he\kan\iz\he(X)\ar[r]_-{\pz{Z}}
\ar[d]^{\simeq}&\kan\iz\he(X)\ar[d]^{\simeq}\\
\kan(Z)\oti_{\Z}\he\kan(Z)\ar[r]_-{\pti{}}&\kan(Z)
}
\]
\begin{remark}\label{RemarqueUnAppendixArt1}
Definition \ref{DefinitionDeuxAppendixArt1} has some obvious consequences:
\begin{enumerate}
  \item [\textbf{1.}] For any $x$ in $\kan(Z)$,
  $i_{Z*}\he x=x\, \, \pz{Z}\, \,  [i_{Z*}\he\oo_{Z}\he]$.

  \item [\textbf{2.}] More generally,
  $i_{Z*}\he(x.\, y)=x\, \, \pz{Z}\, \,  i_{Z*}\he y $.

  \item [\textbf{3.}] The product $\pz{Z}$
  endows $\kan\iz\he(X)$ with the structure of a
  \mbox{$\kan(Z)$-module}.
\end{enumerate}
\end{remark}
\subsection{Functoriality}
Let $\apl{f}{X}{Y}$ be a holomorphic map
and $W$ be a smooth
submanifold of $Y$ such that $Z=f^{-1}(W)$ is
smooth.
We can define a morphism $\apl{f\pe}{\kan_{W}\he(Y)}{\kan\iz\he(X)}$\!
by the usual formula
$
f\ee[\kk]=\sum _{i\geq 0}\, (-1)^{i}\Bigl[\tore{i}{\kk}{f}\Bigr]
$.
Let $\ba{f}$ be the restriction of $f$ to $Z$, as shown in the following
diagram:
\[
\xymatrix@C=8ex{
Z\ar[r]^-{i_{Z}\he}\ar[d]_{\ba f}&X\ar[d]^{f}\\
W\ar[r]_-{i_{W}\he}&Y
}
\]
Then we have the following proposition:
\begin{proposition}\label{TheoremeUnAppendixArt1}
For all $x$ in $\kan(W)$ and for all $y$ in $\kan_{W}\he(Y)$, we have
$f\pe\bigl(x\, \, \pz{W}\, \, y\bigr)=\ba{f}{}\pe x\, \, \pz{Z}\, \, f \pe y$.
\end{proposition}

Recall that if $E$ is a holomorphic vector bundle on $X$,
$\lambda _{-1}[E]$ is the element of $\kan(X)$ defined by
\[
\lambda _{-1}[E]=1-[E]+\bigl[\bigwedge ^{2}E\bigr]-
\bigl[\bigwedge ^{3}E\bigr]+\cdots
\]
\begin{proposition}\label{CorollaireUnAppendixArt1}
\emph{(In the algebraic case, see
\mbox{\cite[Proposition 12 and Lemme 19 c.]{BoSe}})}
\begin{enumerate}
  \item [(i)] If $Z$ is a smooth submanifold of $X$,
  then for all $x$ in
  $\kan(Z)$, $i\iz\pe i_{Z*}\he x=x\, .\,
  \lambda _{-1}[N\ee_{Z/X}]$.

  \item [(ii)] Consider the blowup $\ti{X}$ of $X$ along a smooth submanifold
  Y, and let $E$ be the exceptional divisor as shown in the following
  diagram:
  \[
\xymatrix{
  E\ar[r]^{j}\ar[d]_{q}&\tix\ar[d]^{p}\\
  Y\ar[r]_{i}&X
  }
\]
  Let $F$ be the excess conormal bundle on $E$ defined by the exact
  sequence
  \[
\sutrgdpt{F}{q\ee N\ee_{Y/X}}{N\ee_{E/\tix}}{.}
  \]
  Then for
  all $x$ in $\kan(Y)$, $p\pe
  i_{\, !}\he x=j_{\, !}\he\bigl(q\pe x\, .\, \lambda _{-1}F\bigr)$.
\end{enumerate}
\emph{The assertion (ii) of Proposition \ref{CorollaireUnAppendixArt1} is the
\textit{excess formula} in \mbox{$K$-theory}}.
\end{proposition}
\begin{proof}
(i) We write $i\iz\pe i_{Z*}\he x=
i\iz\pe\bigl(x\pz{Z}[i_{Z*}\he\oo_{Z}\he]\bigr)=
x\, .\, i\iz\pe [i_{Z*}\he\oo_{Z}\he]$.
Now
$\tore{i}{i_{Z*}\he\oo_{Z}\he}{i_{Z}\he}=\bigwedge^{i}
N\ee_{Z/X}$ (\cite[Proposition 12]{BoSe}) and we are done.
\par\medskip
(ii) The equality $p\pe\be\bigl[i_{*}\he\oo_{Y}\he\bigr]=
j_{!}\he(\lambda _{-1}\, [F])$ is proved in \cite[Lemme 19.c]{BoSe}.
Thus
\[
p\pe i_{\, !}\he x=p\pe\bigl(x\, \, \pz{Y}\, \, [i\ee\be\oo_{Y}]\he\bigr)=q\pe
x\, \, \pz{\ti{Y}}\, \, p\pe [i_{*}\he\oo_{Y}]\he=
q\pe x\, \, \pz{\ti{Y}}\, \, j_{!}\he\bigl(\lambda _{-1}[F]\bigr)
=j_{*}\he\bigl(q\pe x\, .\, \lambda _{-1}[F]\bigr).
\]
\end{proof}
We will now need a statement similar to Proposition
\ref{TheoremeUnAppendixArt1},
with the hypothesis
that $Z$ is a divisor with simple
normal crossing.
\par\medskip
Let $\apl{f}{X}{Y}$ be a surjective holomorphic map, $D$ be a smooth
hypersurface of $Y$, and $\ti{D}=f\ee\be(D)$.
We suppose that $\ti{D}$ is a divisor of $X$ with simple normal crossing. Let
$\ti{D}_{k}\he$,
$1\leq k\leq N$,
be the branches of $\ti{D}^{\textrm{red}}$. We denote by $\ba{f}_{k}$ the
restricted map
$\aplpt{{f}}{\ti{D}_{k}}{D}{.}$
\begin{proposition}\label{TheoremeDeuxAppendixArt1}
For all elements $u_{k}$ in $\kan_{\ti{D}_{k}}\he(X)$ such that
  $\oo_{\ti{D}}\he=u_{1}+\cdots +u_{N}\he$ and for all $y$ in $\kan(D)$,
\[
f\pe\bigl(i_{D!}\he\,  y\bigr)=\ds\sum _{k=1}^{N}
  \ba{f}{}\pe_{k}(y)\, \, \pz{\ti{D}_{k}}\, \, u_{k}
\]
\end{proposition}
\subsection{Analytic \mbox{$K$-theory} with support in a divisor with simple
normal crossing}
Let $X$ be a smooth complex compact manifold and $D$ be a reduced
divisor with simple normal crossing.
The branches of $D$ will be denoted by $D_{1}\he,\dots ,D_{N}\he$ and
$D_{ij}=D_{i}\cap D_{j}$.
By a d\'{e}vissage argument, the canonical map from $\bigoplus_{i=1}^{n}
G_{D_{i}}\he(X)$ to $G_{D}\he(X)$ is surjective. If $x$ belongs to
$G_{D_{ij}}\he(X)$, then the element $(x,-x)$ of
$G_{D_{i}}\he(X)\oplus G_{D_{j}}\he(X)$ is in the kernel of
this map. We will show that this kernel is generated by these elements:
\begin{proposition}\label{PropositionUnTraduction}
There is an exact sequence:
\[
\xymatrix{
\bigoplus_{i<j}G_{D_{ij}}\he(X)\ar[r]&\bigoplus_{i}G_{D_{i}}\he(X)\ar[r]&
G_{D}\he(X)
\ar[r]&0.}
\]
\end{proposition}
\begin{proof}
We will deal with the exact sequence
\[
\xymatrix{
\bigoplus_{i<j}G(D_{ij})\he\ar[r]&\bigoplus_{i}G(D_{i})\he\ar[r]&
G(D)\he
\ar[r]&0.
}
\]
By the d\'{e}vissage theorem \ref{LemmeUnAppendixArt1}, this sequence is
isomorphic to the initial one.
We proceed by induction on the number $N$ of the branches of $D$.
Let $D'$ be the divisor whose branches are
$D_{1}\he,\dots ,D_{N-1}\he$.
We have a complex
\begin{equation}\label{Etoile}
\xymatrix{
\kan(D'\cap D_{N}\he)\ar[r]&\kan(D')\oplus\kan(D_{N}\he)
\ar[r]^(.62){\pi}&\kan(D)\ar[r]&0
}\tag{$*$}
\end{equation}
where the first map is given by
$
\xymatrix{\alpha \ar@{|->}[r]&(\alpha ,-\alpha )}
$.
Let us verify that this complex is exact. Consider the map
$\apl{\psi\,  }{ \Z[\, \coh(D)\, ]}{\kan(D')
\oplus\kan(D_{N}\he)\bigm/\kan(D'\cap D_{N}\he)}
$
defined by $\psi (\ff)=[i_{D'}\ee\ff]+[\ii_{D'}\he\ff]$.
Remark that
$\ii_{D'}\he$ is a sheaf of $\oo_{D_{N}\he}$-modules, namely
the sheaf of ideals of $D'\cap D_{N}\he$ in
$D_{N}\he$
extended to $D$ by zero.
Let us show that $\psi $ can be defined in $K$-theory.
\par
We consider an exact sequence
$\sutroiszerogd{\ff}{\g}{\hh}$ of coherent sheaves on $D$. Let us
define the sheaf $\nn$
by the exact sequence
\[
\xymatrix{
0\ar[r]&\nn\ar[r]&i_{D'}\ee\ff\ar[r]&i_{D'}\ee\g
\ar[r]&i_{D'}\ee\hh\ar[r]&0.
}
\]
It is clear that
$\nn$ is a sheaf of $\oo_{D'}\he$-modules with support in
$D'\cap D_{N}\he$,
and $[i_{D'}\ee\g]-[i_{D'}\ee\ff]-[i_{D'}\ee\hh]=-[\nn]$
in $\kan
(D')$.
Let us consider the following exact sequence of complexes:
\[
\xymatrix{
0\ar[r]&\ii_{D'}\he\ff\ar[r]\ar[d]&\ff\ar[r]\ar[d]&i_{D'}\ee\ff\ar[r]\ar[d]&0\\
0\ar[r]&\ii_{D'}\he\g\ar[r]\ar[d]&\g\ar[r]\ar[d]&i_{D'}\ee\g\ar[r]\ar[d]&0\\
0\ar[r]&\ii_{D'}\he\hh\ar[r]&\hh\ar[r]&i_{D'}\ee\hh\ar[r]&0}
\]
Let $\cc$ be the first complex, that is the first column of the diagram
above. It is a complex of \mbox{$\oo_{D_{N}\he}\he$-modules}.
If we denote by
$\hh^{k}(\cc)$, $0\leq k\leq 2$, the cohomology sheaves of
$\cc$, we have the long exact sequence
\[
\xymatrix{
0\ar[r]&\hh^{0}\be(\cc)\ar[r]&0\ar[r]
&\nn\ar[r]&\hh^{1}\be(\cc)\ar[r]&0\ar[r]&0\ar[r]&
\hh^{2}\be(\cc)\ar[r]&0.
}
\]
Since
$\hh^{1}(\cc)$ is a sheaf of $\oo_{D_{N}\he}\he\!\!$-modules, $\nn$
is also a sheaf of $\oo_{D_{N}\he}\he\!\!$-modules.
Therefore, $\nn$ is a sheaf of
$\oo_{\dpp\cap D_{N}\he}\he\!\!$-modules.
\par\medskip
In $\kan(D_{N}\he)$ we have
$
[\ii_{\dpp}\he\ff]-[\ii_{\dpp}\he\g]+[\ii_{\dpp}\he\hh]=
[\hh^{0}(\cc)]-[\hh^{1}(\cc)]+[\hh^{2}(\cc)]=-[\nn]
$. Thus
$\psi (\ff)-\psi (\g)+\psi (\hh)=([\nn],-[\nn])=0$
in the quotient.
\par\medskip
If $\ff$ belongs to $\kan(D)$, then $[\ff]=[i_{\dpp}\ee\ff]+
[\ii_{\dpp}\he\ff]$ in $\kan(D)$.
This means that
$\pi \circ\psi =\id$.
We consider now
$\hh$ in $\kan(\dpp)$ and $\kk$ in $\kan(D_{N}\he)$. Then
\[
\psi (\pi (\hh,\kk))=([i_{\dpp}\ee\hh]+[i_{\dpp}\ee\kk])\oplus
([\ii_{\dpp}\he\hh]+[\ii_{\dpp}\he\kk])=
([\hh]+[\kk_{|\dpp\cap D_{N}\he}\he])\oplus
[\ii_{\dpp\cap D_{N}\he}\he\kk].
\]
Remark that
$[\ii_{\dpp\cap D_{N}\he}\he\kk]=[\kk]-
[\kk_{|\dpp\cap D_{N}\he}\he]$
in $G\bigl(D_{N}\he\bigr)$.
Thus
$([\hh]+[\kk_{|\dpp\cap D_{N}\he}\he])\oplus([\kk]-[\kk_{|\dpp\cap
D_{N}\he}\he])=[\hh]\oplus[\kk]$ modulo $G(D'\cap D_{N}\he)$,
so that $\psi \circ\pi =\id$.
This proves that (\ref{Etoile}) is exact.
\par\medskip
We can now use the induction hypothesis with $D'$. We obtain the
following diagram,
where the columns as well as
the first line are exact:
\[
\xymatrix{
0&0&&\\
\kan(\dpp\cap D_{N}\he)\ar[r]^{r}\ar[u]&\kan(\dpp)\oplus\kan(D_{N}\he)
\ar[r]^(.6){u}\ar[u]&\kan(D)\ar[r]&0\\
\bigoplus_{i<N}\kan(D_{iN}\he)\ar[r]^(.4){s}\ar[u]_{q}&\bigoplus_{i<N}\kan(D_{i})
\oplus\kan(D_{N}\he)\ar[u]_{p}&&\\
& \bigoplus_{i<j<N}G(D_{ij})\ar[u]_{t}&&
}
\]
The map ${\bigoplus_{i}
\kan(D_{i})}\xymatrix{\ar[r]^{\pi }&}{\kan(D)}$ is clearly onto. Let
$\alpha $ be an element of ${\bigoplus_{i}
\kan(D_{i})}$ such that $\pi (\alpha )=0$. Then
$u(p(\alpha ))=0$, so that there exists $\beta $ such that
$r(\beta )=p(\alpha )$. There exists $\gamma $ such that
$q(\gamma )=\beta $. Then
$p(\alpha -s(\gamma ))=p(\alpha )-r(q(\gamma ))=0$.
So there exists $\delta $ such that $\alpha =s(\gamma )+t(\delta )$.
It follows that $\alpha $
is in the image of $\bigoplus_{i<j}\kan(D_{ij})$.
Hence we have the exact
sequence
\[
\xymatrix@C=9ex{
\bigoplus_{i<j}\kan(D_{ij})\ar[r]^(.53){s+t}&\bigoplus_{i}
\kan(D_{i})\ar[r]^(.55){\pi }
&\kan(D)\ar[r]&0,
}
\]
which finishes the proof.
\end{proof}

\end{document}